\documentclass[letterpaper,11pt,reqno]{amsart}
\usepackage{amsmath,amsthm,amssymb,mathtools,amsfonts,mathrsfs}
\usepackage{xspace,xcolor}
\usepackage[breaklinks,colorlinks,citecolor=teal,linkcolor=teal,urlcolor=teal,pagebackref,hyperindex]{hyperref}
\usepackage{amsrefs}
\usepackage{tikz-cd}
\usepackage[all]{xy}
\usepackage{bbm}

\newtheorem{thm}{Theorem}[section]
\newtheorem{prop}[thm]{Proposition}
\newtheorem{conj}[thm]{Conjecture}

\newtheorem{lem}[thm]{Lemma}
\newtheorem{cor}[thm]{Corollary}

\theoremstyle{definition}
\newtheorem{defn}[thm]{Definition}
\newtheorem{ex}[thm]{Example}
\newtheorem{rmk}[thm]{Remark}

\newtheorem*{nota*}{Notation}

\newcommand{\bM}{\overline{\mathcal M}}
\newcommand{\ualpha}{{\underline\alpha}}
\newcommand{\uepsilon}{{\underline\varepsilon}}
\newcommand{\on}{\operatorname} \newcommand{\partn}{\mathcal \pi}

\newcommand{\g}{\text{g}}
  
  \newcommand{\PE}{\mathbb{P}^2\!/\!E}
  \newcommand{\Etw}{{E,\mathrm{tw}}}

\newcommand{\vir}{{\operatorname{vir}}}

\newcommand{\GM}{\C^*}
\newcommand{\Cont}{{\operatorname{Cont}}}
\newcommand{\Tot}{{\operatorname{Tot}}}
\newcommand{\cD}{\mathcal D}

\newcommand\A{\mathbb A}
\newcommand\C{\mathbb C}

\newcommand\QMod{\mathrm{QMod}}
\newcommand\Mod{\mathrm{Mod}}

\newcommand\Fcla{{F^{K_S}_{\text{classical}}}}
\newcommand\gab{(\bg,\ba,\bb)}
\newcommand\bb{\mathbf B}
\newcommand\ba{\mathbf a}
\newcommand\bg{\mathbf g}
\newcommand\bd{\mathbf d}

\newcommand\PP{\mathbb P}

\newcommand\bR{\mathbf R}

\newcommand\Q{\mathbb Q}
\newcommand\Z{\mathbb Z}
\newcommand\cA{\mathcal A}
\newcommand\cB{\mathcal B}
\newcommand\cC{\mathcal C}

\newcommand\ccL{\mathcal L}

\newcommand\cO{\mathcal O}

\newcommand\cQ{\mathcal Q}

\newcommand\cX{\mathcal X}

\newcommand\Aut{\operatorname{Aut}}

\newcommand\Hom{\operatorname{Hom}}

\newcommand\Ima{\operatorname{Im}\,}

\newcommand\ch{\operatorname{ch}}

\newcommand\tw{\operatorname{tw}}
\newcommand\ev{\operatorname{ev}}

\newcommand\reg{\mathrm{reg}}
\newcommand\sG{\mathcal{G}}
\newcommand\con{\mathrm{con}}

\title[HAE for $(\PP^2,E)$ and the NS limit of local $\PP^2$]{Holomorphic anomaly equation for $(\PP^2,E)$ and the Nekrasov-Shatashvili limit of local $\PP^2$}

\author{Pierrick Bousseau}
\author{Honglu Fan}
\author{Shuai Guo}
\author{Longting Wu}

\begin{document}

\maketitle

\begin{abstract}
We prove a higher genus version of the genus $0$ local-relative correspondence of van Garrel-Graber-Ruddat: for $(X,D)$ a pair with $X$ a smooth projective variety and $D$ a nef smooth divisor, maximal contact Gromov-Witten theory of $(X,D)$ with $\lambda_g$-insertion is related to Gromov-Witten theory of the total space of $\cO_X(-D)$ and local Gromov-Witten theory of $D$. 

Specializing to $(X,D)=(S,E)$ for $S$ a del Pezzo surface
or a rational elliptic surface and $E$ a smooth anticanonical divisor, we show that maximal contact Gromov-Witten theory of $(S,E)$ is determined by the Gromov-Witten theory of the Calabi-Yau 3-fold $\cO_S(-E)$ and the stationary Gromov-Witten theory of the elliptic curve $E$. 

Specializing further to $S=\PP^2$, we prove that higher genus generating series of maximal contact 
Gromov-Witten invariants of $(\PP^2,E)$ are quasimodular and satisfy a holomorphic anomaly equation.
The proof combines the quasimodularity results and the holomorphic anomaly equations previously known for local 
$\PP^2$ and the elliptic curve.  

Furthermore, using the connection between maximal contact Gromov-Witten invariants of $(\PP^2,E)$ and Betti numbers of moduli spaces of semistable one-dimensional sheaves on $\PP^2$, we obtain a proof of the quasimodularity and holomorphic anomaly equation predicted in the physics literature for the refined topological string free energy of local $\PP^2$ in the Nekrasov-Shatashvili limit.
\end{abstract}

\tableofcontents

\section{Introduction}
\subsection{Higher genus local-relative correspondence}
Let $X$ be a smooth projective complex variety and $D$
a smooth effective divisor on $X$. We assume that $D$ is nef, that is $C \cdot D \geq 0$ for every curve $C$ on $X$.
The main topic of the present paper is the comparison of the relative Gromov-Witten theory of the pair $(X,D)$ and of the local Gromov-Witten theory of the total space $\text{Tot}(\cO_X(-D))$ of the line bundle $\cO_X(-D)$.
 
Let $\beta$ be a curve class on $X$ such that $\beta \cdot D >0$. We denote by $\bM_g(\cO_X(-D),\beta)$ the moduli space of genus $g$ stable maps of class $\beta$ to the total space of $\cO_X(-D)$, and by 
$\bM_{g}(X/D,\beta)$ the moduli space of genus $g$ relative stable maps of class $\beta$ to $(X,D)$ with only one contact condition of maximal tangency along $D$.

As $D$ is nef and $\beta \cdot D>0$,  $\bM_g(\cO_X(-D),\beta)$ coincides with the moduli space $\bM_g(X,\beta)$ of genus $g$ stable maps of class $\beta$ to $X$. On the other hand, there is a natural morphism 
$F \colon \bM_{g}(X/D,\beta) \rightarrow \bM_g(X,\beta)$
obtained by forgetting the relative marking and stabilizing. 
Therefore, it makes sense to try to compare the virtual fundamental classes $[\bM_g(\cO_X(-D),\beta)]^\vir$ and $F_* [\bM_{g}(X/D,\beta)]^\vir$ both living on $\bM_g(X,\beta)$.

In genus $0$, van Garrel, Graber and Ruddat \cite{GGR} 
proved that
\[ [\bM_0(\cO_X(-D),\beta)]^\vir 
= \frac{(-1)^{\beta \cdot D-1}}{\beta \cdot D} F_* [\bM_0(X/D,\beta)]^\vir\,.\]
Our first main result is a generalization of this formula 
in arbitrary genus.

\begin{thm}[=Theorem \ref{loc2rel}]\label{thm_locrel_intro}
For every $g \geq 0$, we have 
\begin{eqnarray*}
[\bM_g(\cO_X(-D),\beta)]^{\vir}=&&\frac{(-1)^{\beta \cdot D-1}}{\beta \cdot D}F_*\left((-1)^g\lambda_g \cap [\bM_{g}(X/D,\beta)]^{\vir}\right)\\
                              && + \sum_{\sG\in G_{g,\beta}}\frac{1}{|\Aut(\sG)|}(\tau_{\sG})_*\left(C_{\sG}\cap [\bM_{\sG}]^{\vir}\right).
\end{eqnarray*}
\end{thm}
The right-hand side of Theorem 
\ref{thm_locrel_intro} is the sum of a leading term and correction terms. The leading term is obtained by capping 
$[\bM_{g}(X/D,\beta)]^{\vir}$ with the top Chern class $\lambda_g$ of the Hodge bundle. The correction terms
are explicitly described in Section \ref{sec:locrel_statement}
in terms of the classes 
$(-1)^{g'} \lambda_{g'} \cap [\bM_{g'}(X/D,\beta')]^{\vir} $ with $g'<g$, $\beta'\leq\beta$ (see Remark \ref{rmk:loc2rel_1}) and the Gromov-Witten theory of the rank $2$ vector bundle $\cO_D(D) \oplus \cO_D(-D)$
over $D$.




The genus $0$ result of \cite{GGR} is proved by an application of the degeneration formula in Gromov-Witten theory. We prove Theorem \ref{thm_locrel_intro} using the same strategy. The main novelty for $g>0$ is that the degeneration formula contains new terms 
which are not present for $g=0$ and come from the bubble
geometry $\PP(\cO_D \oplus \cO_D(D))$. 
We compute these correction terms using the relative virtual localization formula in Gromov-Witten theory applied to the scaling action of $\C^{*}$ on the fibers of the $\PP^1$-bundle $\PP(\cO_D \oplus \cO_D(D))$.


\subsection{The case of log Calabi-Yau surfaces with smooth boundary}\label{sec:log_K3_intro}
We specialize the higher genus local-relative  correspondence given by Theorem 
\ref{thm_locrel_intro} to 
$(X,D)=(S,E)$, where $S$ is a smooth projective surface over 
$\C$ and $E$ is a smooth effective anticanonical divisor on $S$ which is nef. By the adjunction formula, $E$ is a genus $1$ curve.
Examples of such surface $S$ include del Pezzo surfaces and rational elliptic surfaces.

The local geometry $\Tot(\cO_S(-E))$ is the total space of
the canonical line bundle $\cO_S(-E)=K_S$ and is a non-compact Calabi-Yau 3-fold. Let 
$\beta$ be a curve class on $S$ such that $\beta \cdot E>0$.
The moduli space $\bM_g(\cO_S(-E),\beta)$ has virtual dimension $0$ and we define Gromov-Witten invariants 
\[ N_{g,\beta}^{K_S} \coloneqq \int_{[\bM_g(\cO_S(-E),\beta)]^\vir} 1 \,.\]
The moduli space $\bM_g(S/E,\beta)$ has virtual dimension $g$
and we define maximal contact relative Gromov-Witten invariants 
\[ N_{g,\beta}^{S/E}\coloneqq \int_{[\bM_g(S/E,\beta)]^\vir}
(-1)^g \lambda_g \,.\]
The stationary Gromov-Witten invariants of the elliptic curve $E$ are defined as
\begin{equation}\label{eq:gw_E}
\left< \omega \psi_1^{a_1},\dots,\omega \psi_n^{a_n}
\right>_{g,n,d}^E
\coloneqq \int_{[\bM_{g,n}(E,d)]^{\vir}}\prod_{j=1}^n\psi_j^{a_j}\ev_j^*(\omega)
\end{equation}
for every $\ba=(a_1,\cdots,a_n)$ and $g,d \in \Z_{\geq 0}$,
where $\bM_{g,n}(E,d)$ is the moduli space of $n$ pointed genus $g$
degree $d$ stable maps to $E$ and $\omega \in H^2(E)$ is the (Poincar\'e dual)
class of a point.

We show that the local 
invariants $N_{g,\beta}^{K_S}$ are explicitly determined by the relative invariants $N_{g,\beta}^{S/E}$
and the stationary theory of $E$. 
Moreover, this relation can be inverted: the relative  invariants $N_{g,\beta}^{S/E}$ are explicitly determined by the local invariants $N_{g,\beta}^{K_S}$ and the stationary Gromov-Witten theory of $E$. Okounkov and Pandharipande \cite{OP1} have completely solved the stationary Gromov-Witten theory of $E$. Therefore, we establish the complete equivalence of the local theory of $K_S$ and of the relative theory of 
$(S, E)$ with maximal contact and $\lambda_g$ insertion.

In order to write down the formula, we introduce some notation. For every effective $\beta \in H_2(S,\Z)$,
we denote by $Q^{\beta}$ the corresponding monomial in the 
algebra of the monoid of effective curve classes.
In particular, we denote by $Q^E$ the monomial $Q^\beta$
for $\beta$ the class of $E$.
We form the generating series
\begin{equation} \label{eq:F_K_S}
F_g^{K_S}\coloneqq  \delta_{g,0} \Fcla + {\delta_{g,1}} F_{\text{unstable}}^{K_S} +\sum_{\beta,\, \beta \cdot E>0} N_{g,\beta}^{K_S}Q^\beta\,,
\end{equation}
\begin{equation} \label{eq:F_SE}
F_g^{S/E}\coloneqq \delta_{g,0} \Fcla + {\delta_{g,1}} F_{\text{unstable}}^{S/E} +\sum_{\beta,\,\beta \cdot E>0} \frac{(-1)^{\beta \cdot E+g-1}}{\beta \cdot E}N_{g,\beta}^{\PE}Q^\beta\,,
\end{equation}
where
\[F_{\text{classical}}^{K_S} \coloneqq -\frac{1-\delta_{(E\cdot E),0}}{3!(E\cdot E)^2}(\log Q^E)^3,\]
\[F_{\text{unstable}}^{K_S} \coloneqq\left(\frac{1-\delta_{(E\cdot E),0}}{(E\cdot E)}\frac{\chi(S)}{24}-\frac{1}{24}\right)\log Q^E,\]
\[F_{\text{unstable}}^{S/E} \coloneqq - \frac{1-\delta_{(E\cdot E),0}}{(E\cdot E)}\frac{\chi(S)}{24}\log Q^E.\]
Here $\chi(S)$ is the Euler characteristic of $S$
and we adopt the convention that 
$\frac{1-\delta_{(E\cdot E),0}}{(E \cdot E)^2}=0$ and $\frac{1-\delta_{(E\cdot E),0}}{(E\cdot E)}=0$ if $E \cdot E=0$.



According to the genus $0$ result of \cite{GGR}, 
we have 
\[ N_{0,\beta}^{S/E}=(-1)^{\beta \cdot E-1} (\beta \cdot E) N_{0,\beta}^{K_S}\,,\] 
that is 
\[ F_0^{S/E}=F_0^{K_S}\,.\]

For every $\ba=(a_1,\cdots,a_n) \in \Z_{\geq 0}^n$,
we consider the generating series
\begin{equation}\label{eq:F_E}
F^E_{g,\ba}
\coloneqq \delta_{g,1} \delta_{n,0} F^E_{\text{unstable}}+\sum_{d\geq 0}{\tilde{\cQ}^d}\left< \omega \psi_1^{a_1},\dots,\omega \psi_n^{a_n}
\right>_{g,n,d}^E
\end{equation}
of stationary Gromov-Witten invariants of $E$, where
\[F^E_{\text{unstable}}=-\frac{1}{24}\log 
\left((-1)^{E \cdot E}\tilde{\cQ}\right).\]

We view the series $F^E_{g,\ba}$ as a function of the variables $Q^\beta$ through the change of variables
\begin{align}\label{eq:cQ}
\tilde{\cQ}=&(-1)^{E \cdot E}Q^E\exp\left(\sum_{\beta,\, \beta \cdot E>0}(-1)^{\beta \cdot E} (\beta \cdot E)N_{0,\beta}^{S/E}Q^\beta\right)\\ \label{eq:cQ1}
=&(-1)^{E \cdot E}Q^E\exp\left(-\sum_{\beta,\, \beta \cdot E>0}(\beta \cdot E)^2 N_{0,\beta}^{K_S}Q^\beta\right)\\ \label{eq:cQ2}
=& (-1)^{E \cdot E} 
\exp \left( -D^2 F_0^{K_S} \right)
\,,
\end{align}
where $D$ is the differential operator defined by 
$DQ^\beta=(\beta \cdot E)Q^\beta$.

\begin{thm} \label{thm_locrel_SE}
For every $g \geq 0$, we have
\[
    F_g^{K_S}  = (-1)^g F_g^{S/E}+\]
\[    \sum_{n\geq 0} 
\sum_{\substack{g=h+g_1+\dots+g_n,
\\ \ba=(a_1,\dots,a_n)\in \Z_{\geq 0}^n \\ 
(a_j,g_j) \neq (0,0),\,
\sum_{j=1}^n a_j=2h-2} } 
\frac{(-1)^{h-1} F_{h, \ba}^E}{|\Aut (\ba, \bg)|}
\prod_{j=1}^n   (-1)^{g_j-1} 
 D^{a_j+2} F_{g_j}^{S/E}
\,.
\]
\end{thm}

According to Theorem \ref{thm_locrel_SE}, the local series 
$F_g^{K_S}$ is completely determined by the relative series $F_{g'}^{S/E}$ with $g' \leq g$
and the stationary theory of $E$. This relation is clearly invertible, that is the relative series $F_g^{S/E}$ is completely determined by the local series 
$F_{g'}^{K_S}$ with $g' \leq g$ and the stationary theory of $E$.

Theorem \ref{thm_locrel_SE} is a corollary of the  specialization of Theorem \ref{thm_locrel_intro} to $(X,D)=(S,E)$. The non-trivial part of the proof is to express explicitly the correction terms present in Theorem \ref{thm_locrel_intro}
in terms of the stationary Gromov-Witten theory of $E$. 
This is done using the quantum Riemann-Roch theorem
in the form given by Coates and Givental \cite{CG}.
The relatively simple form of Theorem \ref{thm_locrel_SE}
relies on several algebraic identities and in particular on the presence of the classical and unstable terms in the definition of $F_g^{K_S}$, $F_g^{S/E}$ and 
$F_{g,\ba}^E$.

\subsection{Finite generation for $(\PP^2,E)$}  \label{localp2generators}

Finite generation, quasimodularity and holomorphic anomaly
equation for the local series $F_g^{K_{\PP^2}}$
have been recently proven using various techniques by Lho,
Pandharipande \cite{LhoP}, Coates, Iritani
\cite{CI} and Fang, Ruan, Zhang, Zhou
\cite{FRZZ}. We show  that similar properties also hold for the relative series $F_g^{\PP^2/E}$. 

Specializing the formulae 
\eqref{eq:F_K_S} and \eqref{eq:F_SE} to 
$\PP^2$, we get
\begin{equation}\label{eq:F_KP2}
F_g^{K_{\PP^2}}\coloneqq -\frac{\delta_{g,0}}{18}(\log Q)^3-\frac{\delta_{g,1}}{12}\log Q+\sum_{d \geq 1} N_{g,d}^{K_{\PP^2}}Q^d\,,
\end{equation}
\begin{equation} \label{eq:F_P2}
F_g^{\PE}\coloneqq -\frac{\delta_{g,0}}{18}(\log Q)^3-\frac{\delta_{g,1}}{24}\log Q+\sum_{d \geq 1} \frac{(-1)^{d+g-1}}{3d}N_{g,d}^{\PE}Q^d\,.
\end{equation}

The $I$-functions for the Gromov-Witten theory of local $\PP^2$ is
\begin{equation} \label{IfunctionofKP2}
  I^{K_{\PP^2}}(z,q) = \sum_{k=0}^2 I_k(q) H^{k} z^{1-k}:= z\, e^{\frac{H}{z}\log q} \sum_{d\geq 0} q^d \frac{\prod_{k=0}^{3d-1}(-3H-kz)}{\prod_{k=1}^d (H+kz)^3},
\end{equation}
where $H$ is the hyperplane class of $\PP^2$. 
In particular, we have 
$I_0=1$, $I_1=\log q +\bar{I}_1$, $I_2=\frac{1}{2} (\log q)^2
+\bar{I}_1 \log q + O(q)$, where 
\[ \bar{I}_1= 3 \sum_{k \geq 1} \frac{(3k-1)!}{(k!)^3}(-1)^k q^k \,.\]
The functions $I_0$, $I_1$ and $I_2$ form a basis of solutions of the linear differential equation
\begin{equation}\label{mirror_ODE}
\left[\left(q \frac{d}{dq}\right)^3
+3q\left(q \frac{d}{dq}\right)
\left(3 \left(q \frac{d}{dq}\right) +1 \right)
\left(3 \left(q \frac{d}{dq}\right) +2 \right) \right]I=0 \,.
\end{equation}
The variables $q$ and $Q$ are related by the mirror map
\begin{equation} \label{eq:mirror_map}
Q=e^{I_1}\,. 
\end{equation}
Explicitly, we have 
\[ Q=q-6q^2+63q^3-866q^4+13899q^5-246366q^6+\dots\]
and
\[ q=Q+6Q^2+9Q^3+56Q^4-300Q^5+3942Q^6+\dots\]
In particular, we have 
$\frac{d}{dI_1}=Q\frac{d}{dQ}$.
The genus $0$ mirror theorem for $K_{\PP^2}$ \cite{Giv4, LLY,CKYZ} computes the generating series $F_0^{K_{\PP^2}}$
in terms of $I_1$ and $I_2$:
\begin{equation}\label{thm_mirror}
    -3 Q \frac{dF_0^{K_{\PP^2}}}{dQ}=I_2\,.
\end{equation}
More precisely, we have
\[F_0^{K_{\PP^2}}=-\frac{(\log Q)^3}{18}+3Q-\frac{45}{8}Q^2+\frac{244}{9}Q^3-\frac{12333}{64}Q^4+\frac{211878}{125}Q^5+\dots\]

In order to describe higher genus invariants, we introduce the functions\footnote{Our $S$ and $X$ are related to the $A_2$ and $L$ of \cite{LhoP} by 
 $A_2 = \frac{3\,S}{X}+\frac{1}{2}$
 and $L=X^{\frac{1}{3}}$.
 We are defining our generators in the current way in order to have a $X$-degree bound $3g-3$ for the genus $g$ generating function.
 }
\begin{align} X &\coloneqq (1+27q)^{-1}\,, \label{def_X}\\
 I_{11} &\coloneqq q \frac{dI_1}{dq}\,, \label{def_I11}\\
S &\coloneqq q \frac{d}{dq} \left( \log I_{11}-\frac{1}{3} \log X \right)
=q \frac{d}{dq}\log I_{11} - \frac{1}{3}(X-1)\,.\label{def_S}
\end{align}

The functions $S$, $X$ and $I_{11}$ are algebraically independent over $\C$
(see Lemma \ref{lem_alg_ind_S_X}).
Therefore, the ring of functions generated by $S$ and $X$ is the polynomial ring
 $$
\bR\coloneqq \mathbb Q[X,S]\,.
 $$
We define a grading on $\bR$
by $\deg X=\deg S =1$ 
and denote by $\bR_k$ the subspace of polynomials with degree no more than $k$. 
 
The finite generation property for the
Gromov-Witten theory of  $K_{\PP^2}$, proved in  \cite{LhoP,LhoP2,CI} states that, for every $g \geq 2$, we have 
$$
 F_g^{K_{\PP^2}} \in [ X^{-(g-1)} \cdot \bR_{3g-3} ]^{\reg}\,,
$$ where $[ - ]^\reg$   (the ``orbifold regularity" condition) is defined by
\begin{equation} \label{orbifoldreg}
    [ - ]^{\reg}:=\{f(X,S) :  3 \deg_X f +  \deg_S f \geq 0\}  .
\end{equation}

We prove a finite generation result for the series $F_g^{\PE}$ of relative Gromov-Witten invariants of  $(\PP^2,E)$.

\begin{thm}\label{fgproperty}
For every $g \geq 2$, we have $$
 F_g^{\PE} \in [ X^{-(g-1)} \cdot \bR_{3g-3} ]^{\reg}\,.
$$
Moreover, we have $\deg_S F_g^{\PE} \leq 2g-3$.
\end{thm}

The bound of the $S$-degree of $F_g^{\PE}$ by $2g-3$
is specific to the relative theory. In general, 
the local series $F_g^{K_{\PP^2}}$ is of $S$-degree $3g-3$.

For small genera, the series $F_g^{K_{\PP^2}}$
are explicitly known. 
For example,
we have \cite{Hu}
\[ F_1^{K_{\PP^2}}
=-\frac{1}{12} \log q - \frac{1}{2}
\log I_{11}-\frac{1}{12} \log(1+27q)\]
and \cite{LhoP}
\[ F_2^{K_{\PP^2}} =  \frac{1}{X} \Big( \frac{5}{8} \,{S}^{3}+ \frac{ X}{ 8} {S}^{2}+{\frac {X^{2}}{96}}\,S+{\frac {X^{3
	}}{4320}}+{\frac {X^{2}}{4320}}-\frac{X}{2160}\Big)\,.\]
Note that these formulae were first predicted in the string theory literature \cite{ABK, ALM, HKR, KZ}.
We prove similar explicit formulae for the series 
$F_g^{\PE}$ in low genera.

\begin{thm}[=Theorem \ref{f1formula}] \label{f1formula_intro} 
We have
\[F_1^{\PE}
=-\frac{1}{24} \log q+\frac{1}{24} \log(1+27 q)\,.\]
\end{thm}

\begin{thm} \label{f2formula}
We have
\begin{equation}
    \label{eqnF2}
F_2^{\PE} = {\frac {X}{384}} \, S-{\frac {X^2}{360}}+{\frac 
	{X}{240}}-{\frac{1}{720}}\,.
\end{equation}
\end{thm}

The proof of Theorems \ref{fgproperty},
\ref{f1formula}, \ref{f2formula} relies on Theorem \ref{thm_locrel_SE},
the corresponding properties for local $\PP^2$,
and results on the Gromov-Witten theory of the elliptic curve. In particular, the proof of the finite generation in Theorem \ref{fgproperty} uses the quasimodularity properties of the stationary theory of the elliptic curve \cite{OP1}
and a rewriting of the generators $S$ and $X$ in terms of quasimodular forms discussed in Section
\ref{section_quasimod_intro} below. The proof of the $S$-degree bound in Theorem \ref{fgproperty} is more difficult and needs the holomorphic anomaly equation described in Section \ref{section_HAE_intro}.

\subsection{Quasimodularity for $(\PP^2,E)$}
\label{section_quasimod_intro}
The mirror geometry of local $\PP^2$ is naturally related to modular forms. Indeed, the functions 
$I_{11}=q\frac{dI_1}{dq}$ and 
$I_{12}=q\frac{dI_2}{dq}$ are periods of the fibers of the universal family of elliptic curves over the modular curve $Y_1(3) \simeq \{ q \in \C | q \neq -\frac{1}{27},0\} \cup \{ \infty \}$\footnote{More precisely, it is a description of the coarse moduli space of $Y_1(3)$. As a stack, $Y_1(3)$ has a 
$\Z/3$-orbifold point at $q=\infty$.}.
In the context of mirror symmetry, where $Y_1(3)$ is viewed as the stringy K\"ahler moduli space of local $\PP^2$, the point $q=0$ is the large volume point, $q=-\frac{1}{27}$ is the conifold point and $q=\infty$ is the orbifold point.

The modular curve $Y_1(3)$ is the quotient of the 
upper half-plane $\mathbb{H}=\{ \tau \in \C |\, \Ima{\tau}>0\}$ by the action of the congruence subgroup 
$\Gamma_1(3)$. The identification between 
$Y_1(3)$ and $\{ q \in \C | q \neq -\frac{1}{27},0\} \cup \{ \infty \}$
is given by 
\[ \tau =\frac{1}{2}+\frac{1}{2\pi i}\frac{I_{12}(q)}{I_{11}(q)}\,.\]
We denote $\cQ = e^{2 \pi i \tau}$.
Remark that we have $\cQ^3=\tilde{\cQ}$, where $\tilde{\cQ}=-\exp(-D^2 F_0^{K_{\PP^2}})$ was 
introduced in \eqref{eq:cQ2}. Indeed, we have $D=3Q\frac{d}{dQ}=\frac{3}{I_{11}}q \frac{d}{dq}$
and so $-D^2 F_0^{K_{\PP^2}}=3\frac{I_{12}}{I_{11}}$
follows from the mirror theorem \eqref{thm_mirror}.

We define\footnote{The $A,B,C$ defined here are respectively denoted by $A,E,B^3$ in \cite{ASYZ,Zho}.} 
\[
A(\tau)\coloneqq \left( \frac{\eta(\tau)^9}{\eta(3\tau)^3}+27\frac{\eta(3\tau)^9}{\eta(\tau)^3} \right)^{\frac{1}{3}},\quad
B(\tau)\coloneqq \frac{1}{4}\big( E_2(\tau) +3E_2(3\tau) \big),\]
\[ C(\tau)\coloneqq \frac{\eta(\tau)^9}{ \eta(3\tau)^3}\,,\]
where 
$$
\eta(\tau)\coloneqq \cQ^{\frac{1}{24} } \prod_{n=1}^\infty (1-\cQ^n)
$$
is the Dedekind eta function and 
$$E_{2}(\tau)\coloneqq 1- 24	 \sum_{n=1}^{\infty} \frac{n  \cQ^n}{1-\cQ^n}$$
is the weight $2$ Eisenstein series. 
The functions $A$, $B$, and $C$ are quasimodular forms
for $\Gamma_1(3)$. More precisely, $A$ and $C$ are modular respectively of weight $1$ and $3$, and $B$ is quasimodular of weight $2$ and depth $1$.
In fact, $A$, $B$, and $C$ freely generate the ring of 
quasimodular forms 
of $\Gamma_1(3)$:
\begin{equation}
    \QMod(\Gamma_1(3))=\C[A,B,C]\,.
\end{equation}

 

\medskip

Going from the variable $\tau$ to the variable $q$, we can express the quasimodular forms $A$, $B$, $C$ in terms of the functions $X$, $I_{11}$, $S$ introduced in Section \ref{localp2generators}
(\cite{ASYZ,Maier1,Maier2,Zho}):
\begin{align*}
    A = I_{11},\quad B = \frac{I_{11}^2}{X} (X+6S),\quad
    C= \frac{I_{11}^3}{X} .
\end{align*}
The space $[X^{-(g-1)}\bR_{3g-3}]^{\reg}$ of polynomials in $S$ and $X^{\pm}$ introduced in Theorem \ref{fgproperty} has a very natural interpretation in terms of modular forms.
Indeed, we show in Proposition
\ref{prop_spaces_equality} that it can be identified 
with the space of quasimodular forms for $\Gamma_1(3)$ of weight $6g-6$ in the following way:
\begin{equation}\label{eqn_idgenerators}
[X^{-(g-1)}\bR_{3g-3}]^{\reg} = C^{-(2g-2)}\cdot \Q[A,B,C]_{6g-6}\,.  
\end{equation}

Therefore, we can rephrase Theorem \ref{fgproperty} 
as follows.

\begin{thm}\label{fgproperty1}
For every $g \geq 2$, we have
$$
 F_g^{\PE} \in C^{-(2g-2)} \cdot \mathbb Q[A,B,C]_{6g-6}\,.
$$
Moreover, we have $\deg_B F_g^{\PE} \leq 2g-3$.
\end{thm}

\subsection{Holomorphic anomaly equation for $(\PP^2,E)$}
\label{section_HAE_intro}

The holomorphic anomaly equation is a remarkable structure conjecturally underlying
Gromov-Witten theory of Calabi-Yau varieties \cite{BCOV1, BCOV2}.
In the past few years, a series of works has led to the proof of
several instances of the holomorphic anomaly equation: local $\PP^2$ \cite{LhoP,CI}, $\C^3/\Z^3$ \cite{LhoP2, CI}, local $\PP^1 \times \PP^1$ \cite{Lho, Wan2}, toric Calabi-Yau 3-folds (through the Eynard-Orantin topological recursion)\cite{EMO,EO,FLZ,FRZZ}, the formal quintic 3-fold
\cite{LhoP3}, the quintic $3$-fold
\cite{GJR18, CGLL2}, 
the formal elliptic curve \cite{Wan1}, elliptic orbifold projective lines \cite{MRS}, elliptic curves
\cite{ObPi} and elliptic fibrations \cite{ObPi2}.
We combine two a priori distinct directions of this wave of progress, the cases of local $\PP^2$ and of the elliptic curve, to formulate and prove a holomorphic equation for the maximal contact Gromov-Witten theory of $(\PP^2,E)$.

Let 
$$
F_{g,n}^{K_{\PP^2}} 
\coloneqq \left(Q\frac{d}{dQ}\right)^n  F_{g}^{K_{\PP^2}}.
$$
For $2g-2+n>0$, we have \cite{LhoP} $F_{g,n}^{K_{\PP^2}} \in \Q[S,X^{\pm},I_{11}^{-1}]$, homogeneous of degree $n$ with respect to $I_{11}^{-1}$.

We have the following {holomorphic anomaly equation} for local 
$\PP^2$, proved in various ways in \cite{LhoP}, 
\cite{CI} and \cite{EMO,EO,FLZ,FRZZ}: for 
$2g-2+n>0$,
\begin{equation}
  \frac{X}{3\, {I_{11}}^2} \frac{\partial}{\partial S} F_{g,n}^{K_{\PP^2}} = \frac{1}{2} \sum_{\substack{g_1+g_2 = g\\  n_1+n_2=n \\ 2g_i-2+n_i\geq 0}} \binom{n}{n_1}\ F_{g_1,n_1+1}^{K_{\PP^2}}\cdot   F_{g_2,n_2+1}^{K_{\PP^2}} + \frac{1}{2}F_{g-1,n+2}^{K_{\PP^2}}.
\end{equation}

We prove a holomorphic anomaly equation for the series $F_g^{\PE}$ of relative Gromov-Witten invariants of $(\PP^2,E)$.

We denote
$$
F_{g,n}^{\PE} \coloneqq \left(Q\frac{d}{dQ}\right)^n  F_{g}^{\PE}.
$$

\begin{thm}\label{fgproperty2}
For $2g-2+n>0$,
\[F_{g,n}^{\PE}\in \Q[S,X^{\pm},I_{11}^{-1}],\]
homogeneous of degree $n$ with respect to $I_{11}^{-1}$.
\end{thm}

\begin{thm} \label{HAEforrelative}
For $2g-2+n>0$, we have the following \emph{holomorphic anomaly equation}
\begin{equation} \label{eq_HAEPE}
  \frac{X}{3\, {I_{11}}^2} \frac{\partial}{\partial S} F_{g,n}^{\PE} = \frac{1}{2} \sum_{g_1+g_2 = g,  n_1+n_2=n \atop 2g_i-2+n_i\geq 0 \text{ for } i =1,2} \binom{n}{n_1}\ F_{g_1,n_1+1}^{\PE}\cdot   F_{g_2,n_2+1}^{\PE}
\end{equation}
\end{thm}

We note that the holomorphic anomaly equation
for $(\PP^2,E)$ does not contain a loop term $F_{g-1,n+2}^{\PE}$, unlike what happens for local $\PP^2$. A philosophical explanation for why the loop term is not expected for $(\PP^2,E)$ is that the invariants 
$N_{g,\beta}^{\PP^2/E}$ are defined by insertion of the class 
$\lambda_g$, which is known to vanish in restriction to the locus of curves with one non-separating node (see Equation (5) of \cite{FP2}).


We prove the holomorphic anomaly equation for $(\PP^2,E)$ using Theorem \ref{thm_locrel_SE},
the holomorphic anomaly equation for local $\PP^2$ and the holomorphic anomaly equation for the elliptic curve recently proved by Oberdieck and Pixton 
\cite{ObPi}.

We remark that  Theorem \ref{fgproperty2} can be rewriten in terms of quasimodular forms (similar to Theorem \ref{fgproperty1}) as
\begin{equation*}
F_{g,n}^{\PE} \in C^{-(2g-2+n)}\mathbb Q[A,B,C]_{6g-6+2n}  \,
\end{equation*} 
and the holomorphic anomaly equation of Theorem \ref{HAEforrelative}
can be rewritten as
\begin{equation}\label{eq_HAE_modular}
   \frac{\partial}{\partial B} F_{g,n}^{\PE} = \frac{1}{4}  \sum_{g_1+g_2 = g,  n_1+n_2=n \atop 2g_i-2+n_i\geq 0 \text{ for } i =1,2} \binom{n}{n_1}\ F_{g_1,n_1+1}^{\PE}\cdot   F_{g_2,n_2+1}^{\PE}\,.
\end{equation}







\subsection{Conifold gap conjecture}
\label{section:conifold_gap}

The conifold point is the point $q=-\frac{1}{27}$, that is, the cusp of the modular curve $Y_1(3)$ defined by the 
$\Gamma_1(3)$-equivalence class of $\tau=0$. Let $\tau_{\con}
\coloneqq -\frac{1}{3 \tau}$ be the modular coordinate in the neighborhood of the conifold point. For every $g \geq 2$, we define $F_{g,\con}^{K_{\PP^2}}$
(resp.\ $F_{g,\con}^{\PE}$), functions of
$\tau_\con$, by replacing in the expression of $F_g^{K_{\PP^2}}$
(resp.\ $F_g^{\PE}$) in terms of $A$, $B$, $C$:
\begin{enumerate}
\item $A(\tau)$ by $A(\tau_{\con})$,
\item $B(\tau)$ by $B(\tau_{\con})$,
\item $C(\tau)$ by  $C(\tau_\con)$.
\end{enumerate}
Conceptually, $F_{g,\con}^{K_{\PP^2}}$
(resp.\ $F_{g,\con}^{\PE}$) is the holomorphic part near the conifold point of the almost holomorphic modular function on 
$Y_1(3)$ whose holomorphic part is 
$F_g^{K_{\PP^2}}$ (resp.\ $F_g^{\PE}$)
near the large volume point $q=0$. We refer to \cite{ASYZ,Zho} for details.

Let $t_{\con}$ be the flat coordinate near the conifold point defined as the unique solution of \eqref{mirror_ODE} such that 
$t_\con =\frac{1}{\sqrt{3}}(1+27q)+O((1+27q)^2)$ near the conifold point.
Both $t_\con$ and $e^{2i\pi \tau_\con}$ are local coordinates near the conifold point.
In particular, we can view $F_{g,\con}^{K_{\PP^2}}$ (resp.\ $F_{g,\con}^{\PE}$) as functions of $t_\con$. As $C$ has a first order pole and $A$, $B$ are regular at the conifold point (see \cite{Zho}), it follows from  $F_{g}^{K_{\PP^2}} (\text{resp.}\,\,\, F_{g}^{\PE}) \in C^{-(2g-2)} \cdot \Q[A,B,C]_{6g-6}$ that 
\begin{equation}
F_{g,\con}^{K_{\PP^2}} = O\left( \frac{1}{t_{\con}^{2g-2}} \right),\,\,\,\,\,\,\,
F_{g,\con}^{\PE} = O\left( \frac{1}{t_{\con}^{2g-2}} \right)
\end{equation}
near the conifold point $t_{\con}=0$.

According to the conifold gap conjecture
\cite{HK1, HKQ, HKR}, we should have, for every $g \geq 2$,
\begin{equation} \label{eq:gap_local}
 F_{g,\con}^{K_{\PP^2}}
=\frac{B_{2g}}{2g(2g-2)}\frac{1}{t_{\con}^{2g-2}} +O(1)
\end{equation}
near the conifold point $t_{\con}=0$,
where $B_{2g}$ are the Bernoulli numbers.
This conjecture can be checked explicitly in low genus (see section 10.8 \cite{CI}) but is still open in general.

The holomorphic anomaly equation fixes the $B$-dependence of 
$F_g^{K_{\PP^2}}$ and so determines $F_g^{K_{\PP^2}}$ up to some ambiguity living in the $(2g-1)$-dimensional vector space  $C^{-(2g-2)} \cdot \Q[A,C]_{6g-6}$.
The coefficient of $1/t_{\con}^{2g-2}$ and the $(2g-3)$ vanishings of the coefficients of $1/t_{\con}^j$ for $1 \leq j \leq 2g-3$ predicted by 
\eqref{eq:gap_local},
combined with the fact that the $Q$-expansion of
$F_g^{K_{\PP^2}}$ has no constant term, uniquely fix this
ambiguity and so determine 
$F_g^{K_{\PP^2}}$ completely.


We formulate a version of the conifold gap conjecture for the series $F_g^{\PE}$.

\begin{conj} \label{conj_gap}
For every $g \geq 2$, we have 
\begin{equation} \label{eq:gap_relative}
F_{g,\con}^{\PE}
=-\frac{2^{2g-1}-1}{2^{2g-1}}
\frac{|B_{2g}|}{2g(2g-1)(2g-2)}
\frac{1}{t_{\con}^{2g-2}} +O(1)\,.
\end{equation}
\end{conj}

The holomorphic anomaly equation \eqref{eq_HAE_modular} fixes the $B$-dependence of 
$F_g^{\PE}$ and so determines $F_g^{\PE}$ up to some ambiguity living in the $(2g-1)$-dimensional vector space  $C^{-(2g-2)} \cdot \Q[A,C]_{6g-6}$.
The coefficient of $1/t_{\con}^{2g-2}$ and the $(2g-3)$ vanishings of the coefficients of $1/t_{\con}^j$ for $1 \leq j \leq 2g-3$ predicted by 
\eqref{eq:gap_relative},
combined with the fact that the $Q$-expansion of
$F_g^{\PE}$ has no constant term, uniquely fix this
ambiguity and so determine 
$F_g^{\PE}$ completely.

As reviewed in the next section, $F_{g,\con}^{\PE}$ coincides with the Nekrasov-Shatashvili limit of the refined topological string on local $\PP^2$.
Therefore, Conjecture \ref{conj_gap} can be viewed as a special case of the conjectural universal behavior of the refined topological string near a conifold point
predicted by physics arguments in
\cite{KW} and \cite[section 3.2]{HK}. 

It would be interesting to understand how 
\eqref{eq:gap_local} and \eqref{eq:gap_relative} are related through Theorem \ref{thm_locrel_SE}. We leave this question open.

\subsection{Nekrasov-Shatashvili limit of local $\PP^2$}

The results described in Sections 
\ref{localp2generators}--\ref{section:conifold_gap} concern the series 
$F_g^{\PE}$ of maximal contact relative Gromov-Witten invariants of $(\PP^2,E)$
with $\lambda_g$-insertion. However, they have exactly the form predicted in the string theory literature 
\cite{HK, HKK} for the Nekrasov-Shatashvili limit of the refined topological string on local $\PP^2$. We explain below that this fact is not a coincidence and was one of the main motivations for our work. 
In combination with \cite{Bou19b}, we obtain a proof of a mathematically precise version of these physics conjectures.

According to the Gromov--Witten/stable pairs correspondence \cite{MNOP}, the series 
$\bar{F}_g^{K_{\PP^2}}$ can be described in terms of the stable pairs Pandharipande-Thomas invariants $P_{d,n}$ of local $\PP^2$ \cite{PT}. More precisely, we have 
\begin{equation} \label{eq:mnop}
1+\sum_{d \geq 1} \sum_{n \in \Z}
P_{d,n} (-x)^n Q^d = \exp \left( \sum_{g \geq 0} \bar{F}_g^{K_{\PP^2}} u^{2g-2} \right) 
\end{equation} 
where $x=e^{iu}$ and $\bar{F}_g^{K_{\PP^2}} \coloneqq \sum_{d \geq 1} N_{g,d}^{K_{\PP^2}}Q^d$.

The stable pairs invariants $P_{d,n}$ are expected to admit a refinement $P_{d,n,j}$,
such that $P_{d,n}=\sum_j P_{d,n,j}$. For local $\PP^2$, there are several approaches to the definition of $P_{d,n,j}$: cohomological/motivic
\cite{JS,KS} or K-theoretic \cite{NO, CKK}, which conjecturally coincide. The refined topological string free energies $F_{g_1,g_2}^{K_{\PP^2},\text{ref}}$
are then defined by the expansion 
 \begin{equation} \label{eq:mnop_ref}
1+\sum_{d \geq 1} \sum_{n, j \in \Z}
P_{d,n,j}y^j (-x)^n Q^d = \exp \left( \sum_{g_1,g_2 \geq 0} \bar{F}_{g_1,g_2}^{K_{\PP^2},\text{ref}} (\epsilon_1+\epsilon_2)^{2g_1}(-\epsilon_1 \epsilon_2)^{g_2-1} \right)
\end{equation} 
where $x=e^{i \frac{\epsilon_1-\epsilon_2}{2}}$ and $y=e^{i \frac{\epsilon_1+\epsilon_2}{2}}$.
In the unrefined limit $\epsilon_1=-\epsilon_2=u$, $y=1$, \eqref{eq:mnop_ref} reduces to \eqref{eq:mnop} and so 
\[ \bar{F}_g^{K_{\PP^2}}=\bar{F}_{0,g}^{K_{\PP^2},\text{ref}} \,.\]
From string theory arguments, the refined series $\bar{F}_{g_1,g_2}^{K_{\PP^2},\text{ref}}$ are expected to behave in a way similar to the unrefined series 
$\bar{F}_g^{K_{\PP^2}}$ and in particular should satisfy finite generation, quasimodularity and an appropriate generalization of the holomorphic anomaly equation \cite{KW,HK, HKK}. These conjectures are widely open. 
The main issue is to get a geometric understanding of the change of variables $x=e^{i \frac{\epsilon_1-\epsilon_2}{2}}$ and $y=e^{i \frac{\epsilon_1+\epsilon_2}{2}}$. Even in the unrefined case, to prove the properties of the series 
$F_g^{K_{\PP^2}}$ directly from the stable pairs, that is using \eqref{eq:mnop} as a definition and without using the Gromov-Witten interpretation, seems challenging.

However, it is possible to make progress in the Nekrasov-Shatashvili limit $\epsilon_2 \rightarrow 0$, that is for the series $\bar{F}_{g,0}^{K_{\PP^2},\text{ref}}$, for which we can use an alternative definition.
Indeed, in the same way that the genus $0$ series $\bar{F}_0^{K_{\PP^2}}$ (more precisely, the genus $0$ Gopakumar-Vafa invariants $n_{0,d}^{K_{\PP^2}}$) can be described in terms of Euler characteristic of moduli spaces of one-dimensional semistable sheaves (\cite{Kat},\cite[Appendix A]{CMT}), the series $F_{g,0}^{K_{\PP^2},\text{ref}}$ are conjecturally described in terms of Betti numbers of moduli spaces of one-dimensional semistable sheaves.

For every $d>0$ and $\chi \in \Z$, let
$M_{d,\chi}$ be the moduli space of 
one-dimensional Gieseker semistable sheaves on $\PP^2$, of degree $d$ and Euler characteristic $\chi$.
We denote by $Ib_{j}(M_{d,\chi})$ the Betti numbers of the intersection cohomology of 
$M_{d,\chi}$.
According to \cite{Bou19c}, the odd-degree part of the intersection cohomology of $M_{d,\chi}$ vanishes: 
$Ib_{2k+1}(M_{d,\chi})=0$.
For every $d \in \Z_{>0}$ and $\chi \in \Z$, we define 
\[ \Omega_{d,\chi}^{\PP^2}(y^{\frac{1}{2}}) \coloneqq 
(-y^{\frac{1}{2}})^{-\dim M_{d,\chi}}
\sum_{j=0}^{\dim M_{d,\chi}}
Ib_{2j}(M_{d,\chi}) y^{j} \in \Z[y^{\pm \frac{1}{2}}]\,.\]
It is proved in \cite{Bou19c} that the $\Omega_{d,\chi}^{\PP^2}(y^{\frac{1}{2}})$ are the
refined Donaldson--Thomas invariants for
one-dimensional sheaves on $K_{\PP^2}$.
For $y^{\frac{1}{2}}=1$, $\Omega_{d,\chi}^{\PP^2}$ coincides with the genus $0$ Gopakumar-Vafa invariant $n_{0,d}^{K_{\PP^2}}$ (\cite{Kat},\cite[Appendix A]{CMT}). Therefore, one should view 
$\Omega_{d,\chi}^{\PP^2}(y^{\frac{1}{2}})$ as a refined genus $0$ Gopakumar-Vafa invariant of local $\PP^2$.

Tensoring by $\cO(1)$ gives an isomorphism between 
$M_{d,\chi}$ and $M_{d,\chi+d}$. Thus, 
we have $\Omega_{d,\chi}(y^{\frac{1}{2}})=\Omega_{d,\chi'}(y^{\frac{1}{2}})$ if $\chi$ and $\chi'$ are equal modulo $d$.
We define 
\[ \Omega_d^{\PP^2}(y^{\frac{1}{2}})
\coloneqq \frac{1}{d} \sum_{\chi \!\!\! \mod d} \Omega_{d,\chi}(y^{\frac{1}{2}})\,\]
by averaging over the $d$ possible values of 
$\chi$ modulo $d$.
It is conjectured\footnote{Note added in the final version (01/2021): this conjecture has now been proved by Maulik and Shen \cites{maulik2020cohomological}.} in \cite{Bou19c}
that $\Omega_{d,\chi}(y^{\frac{1}{2}})$ is in fact independent of $\chi$. Assuming this conjecture, $\Omega_d^{\PP^2}(y^{\frac{1}{2}})$ would be the common value of the $\Omega_{d,\chi}(y^{\frac{1}{2}})$.

Define 
\begin{equation} \label{FNS_def}
\bar{F}^{NS}
\coloneqq i \sum_{d \in \Z_{>0}}
\sum_{k \in \Z_{>0}} \frac{1}{k} \frac{\Omega_{d}(y^{\frac{k}{2}})}{
y^{\frac{k}{2}}-y^{-\frac{k}{2}}} Q^{kd} \in \Q(y^{\pm \frac{1}{2}})[\![Q]\!]\,.
\end{equation}

Using the change of variables $y=e^{i \hbar}$, we define series $\bar{F}_g^{NS} \in \Q[\![Q]\!]$ by the expansion
\begin{equation} \label{FgNS_def}
 \bar{F}^{NS}=\sum_{g \in \Z_{\geqslant 0}} (-1)^g \bar{F}_g^{NS} 
\hbar^{2g-1}\,.
\end{equation}
Conjecturally, we have $\bar{F}_g^{NS}
=\bar{F}_{g,0}^{K_{\PP^2},\text{ref}}$.
At this point, it is unclear how the definition of $\bar{F}^{NS}$ using one-dimensional sheaves is better than the definition of $\bar{F}_{g,0}^{K_{\PP^2},\text{ref}}$
using stable pairs: one still needs to understand geometrically the change of variables $y=e^{i \hbar}$.

It is precisely such understanding which has been obtained in \cite{Bou19c}.
More precisely, one of the main results of 
\cite{Bou19c}, building on 
\cite{Bou20,Bou19a,Bou19b, Gab},
is the equality
\begin{equation}  \label{eq:NS}
\bar{F}^{NS}_g =\bar{F}_g^{\PP^2/E} \,,
\end{equation}
for every $g \geq 0$, where $\bar{F}_g^{\PE}\coloneqq \sum_{d \geq 1} \frac{(-1)^{d+g-1}}{3d}N_{g,d}^{\PE}Q^d$.
In other words, the series $\bar{F}^{NS}_g$ have a Gromov-Witten interpretation, not as Gromov-Witten series of local $\PP^2$ but as Gromov-Witten series of $(\PP^2,E)$ !
Therefore, all the results of Sections
\ref{localp2generators}--\ref{section_HAE_intro} for 
$\bar{F}_g^{\PE}$ also hold for $\bar{F}_g^{NS}$ and they agree with the predictions of \cite{HK, HKK}.

\begin{thm} \label{thm_NS}
Using
\eqref{FNS_def}-\eqref{FgNS_def} 
as definition of the Nekrasov-Shatashvili limit of the refined topological string on local $\PP^2$, the series $\bar{F}_g^{NS}$ satisfy the finite generation, quasimodularity, holomorphic anomaly equation and low genus explicit formulae predicted by Huang and Klemm 
\cite{HK}.
\end{thm}

Independently from any motivation from physics, Theorem \ref{thm_NS}
provides a surprising way to construct
quasimodular forms from Betti numbers of moduli spaces of one-dimensional semistable sheaves on $\PP^2$.

Finally, we remark that, given 
\eqref{eq:NS}, we can view Theorem \ref{thm_locrel_SE} as expressing the difference between the unrefined limit and the Nekrasov-Shatashvili limit of the refined topological string in terms of the Gromov-Witten theory of the elliptic curve. 
Such relation does not seem to have been predicted in the physics literature.

\subsection{Outline of the paper}

In Section \ref{sec:locrel}, we prove Theorem \ref{thm_locrel_intro}, that is, the general form of the higher genus local-relative correspondence.
In Section \ref{sec:locrel_log_K3}, we prove Theorem \ref{thm_locrel_SE}, that is, the explicit form of the higher genus local-relative correspondence for log Calabi-Yau surfaces.
Starting with Section \ref{sec:fg_quasimod}, we focus on the case of $(\PP^2,E)$ and we prove the finite generation results (Theorems \ref{fgproperty}-\ref{fgproperty1}) and the low genus explicit formulae (Theorems \ref{f1formula_intro}-\ref{f2formula}).
The holomorphic anomaly equation for 
$(\PP^2,E)$ (Theorem \ref{HAEforrelative})
and the $S$-degree bound of Theorem \ref{fgproperty} are proven in Section 
\ref{sec:fg_quasimod}.
Appendix A describes a technical result used in Section \ref{sec:locrel}.

\subsection{Acknowledgement}We thank Michel van Garrel, Rahul Pandharipande, Hyenho Lho, Younghan Bae, Jie Zhou for helpful discussions. We also thank the anonymous referees for valuable suggestions.
P. B. is supported by Dr. Max R\"ossler, the Walter Haefner Foundation and the ETH 
Zurich Foundation. H. F. is supported by ERC-2012-AdG-320368-MCSK and SwissMAP.
S. G. is supported by NSFC 12061131014, 11890660, 11890661 and the 2019 National Youth Talent Support Program of China.
L. W. is supported by grant ERC-2017-AdG-786580-MACI.

This project has received funding from the European Research Council (ERC) under the European Union’s
Horizon 2020 research and innovation program (grant agreement No. 786580).

\section{Higher genus local-relative correspondence}\label{sec:locrel}
\subsection{Relative Gromov-Witten theory}
Foundations of relative Gromov-Witten invariants were made by Li-Ruan \cite{LR}, Ionel-Parker \cite{IP} and Eliashberg-Givental-Hofer \cite{EGH} in symplectic geometry and Li \cites{Jun1, Jun2}
in algebraic geometry. Our presentation is based on \cites{ Jun1, FWY}.

Let $X$ be a smooth projective variety and $D$ a smooth divisor. The intersection number of a curve class $\beta$ with a divisor $D$ is denoted by $\beta \cdot D$. 

A \emph{topological type} $\Gamma$ is a tuple $(g,n,\beta,\rho,\vec{\mu})$ where $g,n$ are non-negative integers, $\beta\in H_2(X,\Z)$ is a curve class and $\vec{\mu}=(\mu_1,\dotsc,\mu_\rho)\in \Z^\rho_{>0}$ is a partition of the number $\beta \cdot D$. 
 

Let $\bM_\Gamma(X,D)$ be the moduli of relative stable maps
with topological type $\Gamma$. There are evaluation maps
\begin{align*}
\ev_X=(\ev_{X,1},\ldots,\ev_{X,n})\colon &\bM_\Gamma(X,D)\rightarrow X^n, \\
\ev_D=(\ev_{D,1},\ldots,\ev_{D,\rho})\colon &\bM_\Gamma(X,D)\rightarrow D^\rho\,.
\end{align*}

The \emph{relative Gromov--Witten invariant with topological type $\Gamma$} is defined to be
\[
\langle \uepsilon \mid \ualpha \rangle_{\Gamma}^{(X,D)}
\coloneqq \displaystyle\int_{[\bM_{\Gamma}(X,D)]^{\on{vir}}} \ev_D^*\uepsilon \cup \ev_X^*\ualpha\,,
\]
where $\uepsilon\in H^*(D)^{\otimes \rho}$, $\ualpha\in H^*(X)^{\otimes n}$,
\begin{equation*}
\ev_D^*\uepsilon \coloneqq \prod\limits_{j=1}^\rho \ev_{D,j}^*\varepsilon_j, \quad \ev_X^*\ualpha \coloneqq \prod\limits_{i=1}^n \ev_{X,i}^*\alpha_i \,.
\end{equation*}

We also allow disconnected domains. Let $\Gamma=\{\Gamma^\partn\}$ be a set of topological types. The relative invariant with disconnected domain curves is defined by the product rule:
\[
\langle \uepsilon\mid \ualpha \rangle_{\Gamma}^{\bullet(X,D)} \coloneqq
\prod\limits_{\partn} \langle \uepsilon^{\partn} \mid \ualpha^{\partn} \rangle_{\Gamma^{\partn}}^{(X,D)}.
\]
Here $\bullet$ means possibly disconnected domains. We will call this $\Gamma$ a \emph{possibly disconnected topological type}. We now recall the definition of an {\emph {admissible graph}}.

\begin{defn}[\cite{Jun1}*{Definition 4.6}]\label{defn:adm}
{\emph {An admissible graph}} $\Gamma$ is a graph without edges plus the following data.
\begin{enumerate}
    \item An ordered collection of legs.
    \item An ordered collection of weighted roots.
    \item A function $\g:V(\Gamma)\rightarrow \Z_{\geq 0}$.
    \item A function $b:V(\Gamma)\rightarrow H_2(X,\Z)$.
\end{enumerate}
\end{defn}
Here, $V(\Gamma)$ denotes the set of vertices of $\Gamma$. Legs and roots are regarded as half-edges of the graph $\Gamma$. 
A relative stable morphism is associated to an admissible graph in the following way. Vertices in $V(\Gamma)$ correspond to the connected components of the domain curve. Roots and legs correspond to relative markings and interior markings, respectively. Weights on roots correspond to contact orders at the corresponding relative markings. 

The functions $\g$ and $b$ assign to a component its genus and degree respectively. We do not spell out the formal definitions in order to avoid heavy notation, but we refer the readers to \cite{Jun1}*{Definition 4.7}. 

\begin{rmk}
A (possibly disconnected) topological type and an admissible graph are equivalent concepts. Different terminologies emphasize different aspects. For example, admissible graphs will be glued at half-edges into actual graphs.
\end{rmk}

\subsection{Statement of the local-relative correspondence}
\label{sec:locrel_statement}
Let $X$ be a smooth projective variety over $\C$ and $D$ be a smooth effective divisor on $X$.
We further assume that $D$ is nef, that is $C \cdot D \geq 0$ for every curve $C$ in $X$.
Let $\beta$ be a curve class on $X$ such that $\beta \cdot D >0$. Let $\bM_g(\cO_X(-D),\beta)$ be the moduli space of stable maps to the total space of $\cO_X(-D)$. Here the domain curves contain no markings and we omit the number of markings in the notation $\bM_g(\cO_X(-D),\beta)$ for simplicity. The moduli stack $\bM_g(\cO_X(-D),\beta)$ is isomorphic to $\bM_g(X,\beta)$ thanks to the condition
$\beta \cdot D>0$. In this case, it was proved by van Garrel, Graber and Ruddat that 
\begin{thm}[\cite{GGR}]
\[[\bM_0(\cO_X(-D),\beta)]^{\vir}=\frac{(-1)^{(\beta \cdot D)-1}}{\beta \cdot D}F_*\left([\bM_{\Gamma}(X,D)]^{\vir}\right)\]
where  $\Gamma=(0,0,\beta,1,(\beta \cdot D))$
and $F:\bM_{\Gamma}(X,D)\rightarrow \bM_{0,0}(X,\beta)$ is the natural stabilization map which also forgets the unique relative marking. 
\end{thm}
Note that the topological type $\Gamma=(0,0,\beta,1,(\beta \cdot D))$ corresponds to genus-$0$ relative stable maps with maximal contact at $D$.
For convenience, whenever $\Gamma=(g,0,\beta,1,(\beta \cdot D))$ (genus-$g$ maximal contact), we always denote the above relative moduli space (with only one relative marking) as $\bM_{g}(X/D,\beta)$.

\begin{rmk}
In the treatment of \cite{GGR}, the log geometric definition of relative Gromov--Witten theory is used. But in this paper, we choose the expanded degeneration version because at the moment of writing, the references for the localization of the moduli of stable log maps are not available. Besides, since we only care about the pushforward of the relative virtual cycles to the moduli of stable maps, \cite{AMW} tells us that the two choices of definitions do not affect the goal of this paper.
\end{rmk}

We generalize the main result of \cite{GGR} to higher genera. More precisely,
we show that $[\bM_g(\cO_X(-D),\beta)]^{\vir}$ is of the form
\[\frac{(-1)^{\beta \cdot D-1}}{\beta \cdot D}F_*\left((-1)^g\lambda_g \cap [\bM_{g}(X/D,\beta)]^{\vir} \right)+\cdots\]
where $\lambda_g$ is the $g$-th Chern class of the Hodge bundle, $\cap$ is the cap product of between cycles and Chow cohomology and ``$\cdots$" consists of correction terms which will be made explicit later. Before we explicitly describe the correction terms, we need the following preparation.

\begin{defn}\label{defn:star}
A \emph{graph of star type} $\sG$ is a tuple $(V,E,\g,b)$ such that
\begin{enumerate}
\item The set $V$ of vertices admits a disjoint union decomposition 
$V=\{ v\} \coprod V_1$ such that the set $E$ of edges contains exactly one edge between $v$ and $v_i$
for every $v_i \in V_1$ and no other edges.
\item A map $\g \colon V\rightarrow \Z_{\geq 0}$.
\item A map $b \colon \{v\}\cup V_1\rightarrow H_2(D,\Z)\cup H_2(X,\Z)$ such that $b$ maps $v$ into $H_2(D,\Z)$ and maps $V_1$ into $H_2(X,\Z)$.
\end{enumerate}
\end{defn}
The automorphism group of $\sG$ consists of automorphisms of the graph $(V,E)$ which commute with $\g,b$. We denote it as $\Aut(\sG)$.

\begin{defn}
The topological type of a graph of star type is a tuple $(g,\beta)$ such that
\begin{enumerate}
\item $g$ is the summation of all genera (the values of $\g$).
\item $\beta=\iota_*b(v)+\sum_{v_i\in V_1}b(v_i)\in H_2(X,\Z)$ where $\iota \colon D\hookrightarrow X$ is the natural inclusion.
\end{enumerate}
\end{defn}
We denote by $G_{g,\beta}$ the set of graphs with star type whose topological type is $(g,\beta)$.

For each $\sG$, we define
\begin{equation}\label{eqn:MG}
\bM_{\sG}
\coloneqq 
\left(\prod_{v_i\in V_1}\bM_{\g(v_i)}(X/D,b(v_i))\right)\times_{D^{|E|}}\bM_{\g(v),|E|}(D,b(v))    
\end{equation}
where $\times_{D^{|E|}}$ is the fiber product identifying evaluation maps according to edges. The evaluation map from $\prod_{v_i\in V_1}\bM_{\g(v_i)}(X/D,b(v_i))$ to $D^{|E|}$ is determined by the relative markings. 

We define a virtual fundamental class on $\bM_{\sG}$ by ``gluing" the virtual fundamental classes on the factors (for example, it was also introduced in \cite[Equation (4)]{GV})
\[[\bM_{\sG}]^{\vir} := \Delta^![\prod_{v_i\in V_1}\bM_{\g(v_i)}(X/D,b(v_i))\times\bM_{\g(v),|E|}(D,b(v))]^{\vir},\]
where $\Delta:D^{|E|}\rightarrow D^{|E|}\times D^{|E|}$ is the diagonal map and $\Delta^!$ is the Gysin map.

There is a natural stabilization map
\[\mathfrak{s}\colon \prod_{v_i\in V_1}\bM_{\g(v_i)}(X/D,b(v_i))\longrightarrow \prod_{v_i\in V_1}\left(\bM_{\g(v_i),1}(X,b(v_i))\times_X D\right)\]
where $\times_X$ is the fiber product identifying the unique evaluation map and the inclusion map $D\hookrightarrow X$.

There is also a natural gluing map
\[\left(\prod_{v_i\in V_1}\left(\bM_{\g(v_i),1}(X,b(v_i))\times_X D\right)\right)\times_{D^{|E|}}\bM_{\g(v),|E|}(D,b(v))\longrightarrow \bM_g(X,\beta).\]
By composition of the stabilization and gluing maps, we get a map 
\[\tau_{\sG} \colon \bM_{\sG}\longrightarrow \bM_g(X,\beta).\]

Let $N_{D/X}$ be the normal bundle of $D$ in $X$ and let $N_{D/X}^{\vee}$ be its dual.
We consider the rank 2 vector bundle over $D$ given by
\begin{equation} \label{eqn:N}
N=N_{D/X}\oplus N_{D/X}^{\vee}\,.
\end{equation}
 There is a natural \emph{anti-diagonal} scaling action of $\GM$ on $N$ with weight $1$ on $N_{D/X}$ and weight $-1$
 on $N_{D/X}^{\vee}$. 

We have a universal diagram
\[
 \xymatrix{
 U\ar[r]^{f}\ar[d]^{\pi}& X\\
 \bM_{\g(v),|E|}(D,b(v))
}\]
where $\pi$ is the universal domain curve and $f$ is the universal stable map.
We view $-R^{\bullet}\pi_*f^*N$
as an element of the $K$-theory of $\bM_{\g(v),|E|}(D,b(v))$
and we consider its equivariant Euler class\footnote{It is a bivariant class because the equivariant Euler class can be expressed as a combination of Chern classes.}
\[e_{\GM}(-R^{\bullet}\pi_*f^*N)\in A^*\left(\bM_{\g(v),|E|}(D,b(v))\right)(t)\]
where $t$ is the equivariant parameter. Now let\footnote{The denominators of Equation \eqref{eqn:Cv} and \eqref{eqn:Cvi} need to be expanded as a power series in $t^{-1}$. The expansions terminate at a finite degree because the classes in the denominators other than $t$ are nilpotent.}
\begin{equation}\label{eqn:Cv}
C_v \coloneqq e_{\GM}(-R^{\bullet}\pi_*f^*N)\prod_{v_i\in V_1}\frac{(t+\ev_i^*c_1(N_{D/X}))(-1)^{d_i} d_i}{t+\ev_i^*c_1(N_{D/X})-d_i\psi_i}
\end{equation}
where 
\begin{equation} \label{eqn:ev_i}
 \ev_i\colon \bM_{\g(v),|E|}(D,b(v)) \longrightarrow D 
\end{equation}
is the evaluation map for the $i$-th marking, 
$d_i \coloneqq b(v_i)\cdot D$
and $\psi_i$ is the psi-class of the $i$-th marking. We have
\[C_v\in A^*\left(\bM_{\g(v),|E|}(D,b(v))\right)(t).\]

For each $v_i\in V_1$, we define
\begin{equation}\label{eqn:Cvi}
C_{v_i}\coloneqq \frac{t}{t+d_i\bar{\psi}+\bar{\ev}^* c_1(N_{D/X})}(-1)^{g(v_i)}\lambda_{g(v_i)}
\end{equation}
where 
\begin{equation} \label{eqn:bar_ev}
 \bar{\ev}\colon  \bM_{\g(v_i)}(X/D,b(v_i))
\longrightarrow D 
\end{equation}
is the evaluation map associated to the unique relative marking, and
$\bar{\psi}$ is the psi-class associated to the unique relative marking. 
We have 
\[ C_{v_i} \in A^*\left(\bM_{\g(v_i)}(X/D,b(v_i))\right)(t)\,.\]
Since there is only one relative marking, we have $\bar{\psi}=\mathfrak{s}_i^*\psi_1$ where 
\[\mathfrak{s}_i:\bM_{\g(v_i)}(X/D,b(v_i))\longrightarrow \bM_{\g(v_i),1}(X,b(v_i))\]
is the natural stabilization map and $\psi_1$ is the psi-class associated to the unique marking of $\bM_{\g(v_i),1}(X,b(v_i))$. So we may also treat $C_{v_i}$ as a pullback class via $\mathfrak{s}_i$. We define
\[C_{\sG}\coloneqq \left[p_v^*C_v\prod_{v_i\in V_1}p_{v_i}^*C_{v_i}\right]_{t^0}\]
where $p_v$, $p_{v_i}$ are projections from $\bM_{\sG}$ to the corresponding factors and $[\cdots]_{t^0}$ means that we take the constant term.

Now we are ready to state our higher genus local-relative 
correspondence.

\begin{thm}[=Theorem \ref{thm_locrel_intro}]\label{loc2rel}
The following relation holds in $A_*(\bM_g(X,\beta))$.
\begin{eqnarray*}
[\bM_g(\cO_X(-D),\beta)]^{\vir}=&&\frac{(-1)^{\beta \cdot D-1}}{\beta \cdot D}F_*\left((-1)^g\lambda_g \cap [\bM_{g}(X/D,\beta)]^{\vir}\right)\\
                              && + \sum_{\sG\in G_{g,\beta}}\frac{1}{|\Aut(\sG)|}(\tau_{\sG})_*\left(C_{\sG}\cap [\bM_{\sG}]^{\vir}\right).
\end{eqnarray*}
\end{thm}

The proof of Theorem \ref{loc2rel} is given in Sections
\ref{sec:deg}-\ref{sec:loc} and uses the degeneration and localization formulae.

\begin{rmk}\label{rmk:loc2rel_1}
For a graph of star type $\sG$, if $g(v)=0$ then it follows from the Riemann-Roch theorem that $e_{\GM}(-R^{\bullet}\pi_*f^*N)$
only contains negative powers of $t$. Therefore, $p_v^*C_v\prod_{v_i\in V_1}p_{v_i}^*C_{v_i}$ must also only contain negative powers of $t$, and so $C_{\sG}=0$. Thus, a non-vanishing contribution of $\sG$ is only possible if 
$g(v)>0$, and in particular if $g(v_i)<g$ for each $v_i\in V_1$. That also explains the absence of correction terms
if $g=0$.
\end{rmk}
\subsection{Degeneration} \label{sec:deg}
As the first step to prove Theorem \ref{loc2rel}, we apply the degeneration formula to the degeneration to the normal cone of the embedding $D\hookrightarrow X$. This step is identical to the corresponding step in the proof of the main theorem of \cite{GGR}. We simply recall the degeneration formula and some key lemmas in \cite{GGR} for the sake of completeness.

Let $\cX$ be the blow-up of $X\times \A^1$ along $D\times \{0\}$. The space $\cX$ still admits a projection to $\A^1$ and $\cX_0$, the fiber over $0$, is a union of the $\PP^1$ bundle $\PP_D(N_{D/X}\oplus \cO)$ and $X$ glued along the section $D\cong\PP_D(N_{D/X})\subset \PP_D(N_{D/X}\oplus \cO)$ and the hypersurface $D\subset X$. For convenience, we introduce the following notation.
\begin{itemize}
    \item Denote the $\PP^1$ bundle by $P_0$ and the
    component of $\cX_0$ isomorphic to $X$ by $X_0$.
    \item Denote the section $\PP_D(N_{D/X})\subset \PP_D(N_{D/X}\oplus \cO)$ by $D_\infty$ and the section $\PP_D(\cO)\subset \PP_D(N_{D/X}\oplus \cO)$ by $D_0$.
\end{itemize} 

Let $\ccL$ be the line bundle on $\cX$ associated with the divisor of the strict transform of $D\times \A^1\subset X\times \A^1$. Since we want to relate the local theory with the relative theory, we need to consider the total space $\Tot(\ccL)$ in the degeneration. The general fiber of $\Tot(\ccL)$ over $\A^1$ is isomorphic to the total space of $\cO_X(-D)$, which is the target space of our local theory. The special fiber at $0\in \A^1$ is a union of $\ccL|_{P_0}=\cO_{P_0}(-D_0)$ and $\ccL|_{X_0}=\cO_{X_0}$ glued along the corresponding divisors both of which are isomorphic to $D\times \A^1$. Note that the line bundle $\ccL|_{X_0}$ is isomorphic to the trivial bundle on $X$.

The degeneration formula expresses the virtual cycle of $\bM_g(\cO_X(-D),\beta)$ in terms of the ones of $\bM^\bullet_{\Gamma_1}(\ccL|_{P_0}, \ccL|_{D_\infty})$ and $\bM^\bullet_{\Gamma_2}(X\times \A^1, D\times \A^1)$ summing over all splittings $\mathfrak i = (\Gamma_1,\Gamma_2)$ of the genus-$g$ degree-$\beta$ curve type. The splitting forms a bipartite graph and must have matching contact orders at the corresponding relative markings. We follow the \cite[Definition 4.11]{Jun2} for treating the splitting of topological types, but we note that the more recent and general interpretation of the splitting is in terms of tropical curves and can be found in \cite{KLR}, for example. Since the moduli stack of stable maps to the stack of expanded degeneration admits a morphism to $\bM_g(X,\beta)$ induced by the projection of the target $\cX\rightarrow X$, we have the following version of degeneration formula,
\begin{equation}\label{eqn:deg}
\begin{split}
&[\bM_g(\cO_X(-D),\beta)]^{\vir} \\
= &\sum_{\mathfrak i=(\Gamma_1,\Gamma_2)} \dfrac{\prod_{i=1}^{m(\mathfrak i)} d_i}{\Aut(\mathfrak i)} \tau_*\Delta^! [\bM^\bullet_{\Gamma_1}(\ccL|_{P_0}, \ccL|_{D_\infty}) \times \bM^\bullet_{\Gamma_2}(X\times \A^1, D\times \A^1)]^{
\vir}\,,
\end{split}
\end{equation}
where there are $m(\mathfrak i)$ roots (relative markings) on each $\Gamma_1$ and $\Gamma_2$, $d_i$ are the weights (contact order) of the corresponding roots, $\Delta^!$ is the Gysin pullback induced by the diagonal $\Delta: (D\times \A^1)^{m(\mathfrak i)} \rightarrow (D\times\A^1\times D\times \A^1)^{m(\mathfrak i)}$
and $\tau$ is the forgetful map from the fiber product $\bM_{\mathfrak i}$ to $\bM_g(X,\beta)$.

First of all, there is a distinguished term in the degeneration formula where $\Gamma_1$ consists of one vertex of genus $0$, curve class $\beta\cdot D$ times of fiber class with one root of weight $\beta\cdot D$, and $\Gamma_2$ consists of one vertex of genus $g$, curve class $\beta$ and a weight-$(\beta\cdot D)$ root. Thanks to \cite{AMW}, \cite[Proposition 2.4]{GGR} can be stated in our setting in order to understand this term.

\begin{lem}{\cite[Proposition 2.4]{GGR}}\label{lem:leadcoef}
Let $\Gamma_1$ be a topological type of genus $0$, curve class $\beta\cdot D$ times the fiber class, and with one root of weight $\beta\cdot D$. Then
\[
\ev_*[\bM^\bullet_{\Gamma_1}(\ccL|_{P_0},\ccL|_{D_\infty})]^{\vir} = \dfrac{(-1)^{\beta\cdot D-1}}{(\beta\cdot D)^2} [D]\,,
\]
where $\ev$ is the evaluation map of the (unique) relative marking.
\end{lem}

The lemma implies that this distinguished term equals
\[\dfrac{(-1)^{\beta\cdot D-1}}{\beta \cdot D}F_*\left((-1)^g\lambda_g \cap [\bM_{g}(X/D,\beta)]^{\vir}\right)\]
in Theorem \ref{loc2rel}.

To understand the rest of the terms, we need \cite[Lemma 3.1]{GGR}. The idea is that in the relative theory of $(X\times \A^1, D\times \A^1)$, we rule out topological types $\Gamma_2$ with multiple relative contacts.
\begin{lem}{\cite[Lemma 3.1]{GGR}}\label{lem:root}
Let $\mathfrak i=(\Gamma_1,\Gamma_2)$ be a splitting such that there exists a vertex $v$ in $\Gamma_2$ having more than one root (roots correspond to relative markings). Then 
\[\Delta^![\bM^\bullet_{\Gamma_1}(\ccL|_{P_0}, \ccL|_{D_\infty}) \times \bM^\bullet_{\Gamma_2}(X\times \A^1, D\times \A^1)]^{
\vir} = 0\,.\]
\end{lem}
Although we use the expanded degeneration definition of relative Gromov--Witten theory, the proof of \cite[Lemma 3.1]{GGR} still applies here word-by-word because the gist of the proof is an intersection theoretic computation on the target spaces and the geometry of moduli spaces plays a minor role.


A splitting $\mathfrak i=(\Gamma_1,\Gamma_2)$ gives rise to a bipartite graph by gluing the corresponding roots. Lemma \ref{lem:root} tells us that every vertex of $\Gamma_2$ consists of only one root. This suggests that the bipartite graphs are comb-shaped and they match the underlying graphs of Definition \ref{defn:star}. Thus the degeneration formula is already giving us a general form that looks like Theorem \ref{loc2rel} except that the pushforward of virtual cycles of $\bM^\bullet_{\Gamma_1}(\ccL|_{P_0},\ccL|_{D_\infty})$ are not understood in general.

\subsection{Localization}\label{sec:loc}
We use the relative virtual localization formula \cite{GV} to understand the pushforward of $[\bM^\bullet_{\Gamma_1}(\ccL|_{P_0},\ccL|_{D_\infty})]^{\vir}$. First recall that $P_0$ has two sections: $D_0$ and $D_\infty$. The normal bundle of $D_0$ is $\cO_D(D)$ and the one of $D_\infty$ is $\cO_D(-D)$. The line bundle $\ccL|_{P_0}$ is naturally isomorphic to $\cO_{P_0}(-D_0)$. Let $\C^*$ act on $\ccL|_{P_0}$ as follows:
\begin{itemize}
    \item $\C^*$ acts fiberwise on $P_0$ so that the weight of fibers of the normal bundle of $D_0$ is $1$.
    \item $D_0$ becomes an invariant divisor under the above  $\C^*$-action on $P_0$ and let $\C^*$ act on $\ccL|_{P_0}$ in such a way that $\ccL|_{P_0}\cong \cO_{P_0}(-D_0)$ as equivariant line bundles.
\end{itemize}
Under this construction, the $\C^*$ acts fiberwise on $\ccL|_{D_0}$ with weight $-1$ and it acts on $\ccL|_{D_\infty}$ trivially. The fixed loci of $\ccL|_{P_0}$ consists of $D_0$ and the whole total space $\ccL|_{D_\infty}$. The $\C^*$-action induces an action on $\bM^\bullet_{\Gamma_1}(\ccL|_{P_0},\ccL|_{D_\infty})$. Invariant curves decompose into components mapped into $D_0$, components mapped into rubbers over $D_\infty$ and components\footnote{We call those $\C^*$-invariant components mapped to fibers of $P_0$ edge components.} mapped to fibers of $P_0$. The localization formula can be summarized as follows:
\begin{align}\label{eqn:virloc}
\begin{split}
&[\bM^\bullet_{\Gamma_1}(\ccL|_{P_0},\ccL|_{D_\infty})]^{\mathrm{eq},\vir} \\
=& \dfrac{[\bM^\bullet_{\Gamma_1}(\ccL|_{P_0},\ccL|_{D_\infty})^{\mathrm{simple}}]^{\vir}}{e_{\C^*}(N^{\vir})} + \sum\limits_{\eta=(\Gamma_1^{(0)},\Gamma_1^{(\infty)})} \dfrac{[\bM_\eta]^{\vir}}{e_{\C^*}(N^{\vir}_{\eta})} \,,
\end{split}
\end{align}
with the notation explained as follows. The first term corresponds to the fixed locus where the target does not degenerate (simple fixed locus according to \cite{GV}) and the second summand corresponds to the rest of fixed loci. 
We denote by $N^\vir$ and $N^{\vir}_{\eta}$ the corresponding virtual normal bundles. The index of the summation $\eta$ is a splitting of the topological type $\Gamma_1$ into $\Gamma_1^{(0)}$ and $\Gamma_1^{(\infty)}$ such that:
\begin{itemize}
    \item $\Gamma_1^{(0)}$ consists of a disjoint union of vertices with $k_\eta$ ordered half-edges (corresponding to markings) and the decoration of genus and curve class attached to each vertex. We denote by $\bM^\bullet_{\Gamma_1^{(0)}}(D)$ the corresponding moduli of stable maps to $D$ with possibly disconnected domain;
    \item $\Gamma_1^{(\infty)}$ is a possibly disconnected rubber graph (a possibly disconnected graph with assignments of genera, curve classes, inner markings, relative markings to both boundary divisors and their contact orders as detailed in \cite[Definition 2.4]{FWY}) with $k_\eta$ ordered $0$-roots (corresponding to relative markings over $D_0$).
    The corresponding moduli space is $\bM^{\bullet\sim}_{\Gamma_1^{(\infty)}}(\ccL|_{D_{\infty}} )$, i.e., the moduli space of relative stable maps to the rubber target over $\ccL|_{D_{\infty}}\cong D\times \A^1$ (a description can be found in, for example, \cite{GV}*{Section 2.4}).
\end{itemize}
Gluing $\Gamma_1^{(0)}$ and $\Gamma_1^{(\infty)}$ along these ordered $k_\eta$ half-edges forms a bipartite graph, 
and $\bM_\eta$ is the fixed locus corresponding to $\eta$. 
It is standard to describe an \'etale cover of $\bM_\eta$ as a fiber product of vertex moduli and edge moduli: \[\bM^{\bullet}_{\Gamma_1^{(0)}}({D_0})\times_{D^{k_\eta}} D^{k_\eta} \times_{(D\times\A^1)^{k_\eta}} \bM^{\bullet\sim}_{\Gamma_1^{(\infty)}}(\ccL|_{D_{\infty}} ) \rightarrow \bM_\eta\,,
\]
where the fiber products are over evaluation maps and the inclusion of $D^{k_\eta} \rightarrow (D\times \A^1)^{k_\eta}$ sends $D^{k_\eta}$ to $D^{k_\eta}\times\{0\}$. 



The combinatorics of graph splittings and the precise formulae of $e_{\C^*}(N^{\vir}_{\eta})$ are a priori very complicated. But similar to Lemma \ref{lem:root}, we have the following result to cut down the number of graph types on $\Gamma_1^{(\infty)}$.

\begin{lem}
If any vertex of $\Gamma_1^{(\infty)}$ has more than one $0$-roots,
then we have
\[[\bM_\eta]^\vir=0\,.\]
\end{lem}


Note that $[\bM_\eta]^\vir$ is a multiple of the pushforward of \[[\bM^{\bullet}_{\Gamma_1^{(0)}}({D_0})\times_{D^{k_\eta}} D^{k_\eta} \times_{(D\times\A^1)^{k_\eta}} \bM^{\bullet\sim}_{\Gamma_1^{(\infty)}}(\ccL|_{D_{\infty}} )]^\vir.\] The argument is the same as Lemma \ref{lem:root} and we omit the details. The conclusion of this lemma is that we can assume that each $\Gamma_1^{(0)}$ in the summation only consists of one single vertex. Thus, there is a unique edge between the vertex in $\Gamma_1^{(0)}$ and a vertex in $\Gamma_1^{(\infty)}$. Comparing with Definition \ref{defn:star} and Equation \eqref{eqn:MG}, we see that this shares the shape of graphs of star type, except that the vertices in $V_1$ represent relative stable maps to $X$ instead of relative stable maps to rubbers over $D$ which appears in our localization formula \eqref{eqn:virloc}. But combining the virtual localization formula \eqref{eqn:virloc} and the degeneration formula \eqref{eqn:deg} altogether, we will be able to match the graph sum with Theorem \ref{loc2rel}.

Let us recap the whole process. Starting with the topological type $\Gamma$, we split it into three parts $\Gamma_1^{(0)}, \Gamma_1^{(\infty)}, \Gamma_2$ and sum over cycles
\begin{align*}
[\bM^{\bullet}_{\Gamma_1^{(0)}}({D_0})\times_{D^{k_\eta}} D^{k_\eta} &\times_{(D\times \A^1)^{k_\eta}} \bM^{\bullet\sim}_{\Gamma_1^{(\infty)}}(\ccL|_{D_{\infty}} ) \\
&\times_{(D\times\A^1)^{m(\mathfrak i)}} \bM^\bullet_{\Gamma_2}(X\times\A^1, D\times\A^1)]^{\vir}
\end{align*}
capped with different cohomology classes. To simplify the situation, notice that the sum over $\Gamma_1^{(\infty)}$ and $\Gamma_2$ can be combined into a divisor in 
\[ \bM^{\bullet}_{\Gamma_2'}(X\times\A^1,D\times\A^1)=\bM^{\bullet}_{\Gamma_2'}(X,D)\times (\A^1)^{|V(\Gamma_2')|}\] corresponding to the locus where the target degenerates at least once (denoted by $\delta$). Here $\Gamma_2'$ is the gluing of $\Gamma_1^{(\infty)}$ and $\Gamma_2$ along the corresponding roots and $V(\Gamma_2')$ is the set of vertices of $\Gamma_2'$. In fact, $\delta$ can be written as another divisor that commonly appears in the virtual localization formula as follows.

According to \cite{GV}*{Section 2.5}, $\bM^{\bullet}_{\Gamma_2'}(X,D)$ admits a map to $\mathcal T$(or \cite{Jun1}*{Chapter 4} as $\mathfrak X^{rel}$), the Artin stack of expanded degenerations of $(X,D)$. Also, by \cite{GV}*{Section 2.5}, $\mathcal T$ is isomorphic to an open substack of $\mathfrak M_{0,3}$ (the Artin stack of prestable $3$-pointed rational curve) consisting of chains of curves that separate $\infty$ from $0,1$ (in \cite{GV}'s notation). There is a divisor corresponding to cotangent lines at $\infty$, and we denote by $\Psi$ its pullback to $\bM^\bullet_{\Gamma_2'}(X,D)$. 
\begin{lem}\label{lem:psi}
$\Psi$ is linearly equivalent to $\delta$. 
\end{lem}
\begin{proof}
When identifying $\mathcal T$ as an open substack of $\mathfrak M_{0,3}$, the relation $\Psi=\delta$ is simply the pullback of the relation $\psi_1=D_{1|23}$ on $A^2(\mathfrak M_{0,3})$.
\end{proof}

Now that $\Psi$ represents the decomposition of $\Gamma_2'$ into different $\Gamma_1^{(\infty)}$ and $\Gamma_2$, we can simplify the localization formula combined with the degeneration formula into the following form:
\begin{align}\label{eqn:combine}
\begin{split}
&[\bM_g(\cO_X(-D),\beta)]^{\vir} \\
= & \Bigg[\sum\limits_{\mathfrak i=(\Gamma_1',\Gamma_2')} \dfrac{1}{\Aut(\mathfrak i)} (\tau_{\mathfrak i})_*\bigg(
  p_1^*\left( \dfrac{e_{\C^*}(-R^{\bullet}(\pi_1)_*f_1^*N)}{\prod_{e_i\in \mathrm{HE}(\Gamma_1')} \left(\frac{t+\ev_i^*c_1(N_{D/X})}{d_i}-\psi_i\right)} \right) \\& p_2^*\left( \left( 1+\dfrac{\Psi}{-t-\Psi}\right) \cdot e(R^{1}(\pi_2)_*f_2^*\cO) \right)\text{Edge}(\mathfrak i)\cap
[\bM_{\Gamma_1'}(D_0)\times_{D^{k_{\mathfrak i}}}\bM^{\bullet}_{\Gamma_2'}(X,D)]^{\vir}   \bigg) \Bigg]_{t^0},
\end{split}
\end{align}
where 
\begin{itemize}
    \item $\mathfrak i$ is a splitting such that each vertex of $\Gamma_2'$ has one unique root (corresponding to relative marking); $\Gamma'_1$ consists of one single vertex which is allowed to be a degree-zero genus-zero unstable vertex;
    \item $\tau_{\mathfrak i}$ is the natural gluing  map
    \[
    \tau_{\mathfrak i}:\bM_{\Gamma_1'}(D_0)\times_{D^{k_{\mathfrak i}}}\bM^{\bullet}_{\Gamma_2'}(X,D) \rightarrow \bM_{g}(X,\beta)=\bM_{g}(\cO_X(-D),\beta);
    \]
    \item $\text{Edge}(\mathfrak i)$ is a certain bivariant class depending on edges in the bipartite graph splitting $\mathfrak i$ (corresponding to the edge contribution in the localization);
    \item $p_1$ and $p_2$ are the projections from $\bM_{\Gamma_1'}(D_0)\times_{D^{k_{\mathfrak i}}}\bM^{\bullet}_{\Gamma_2'}(X,D)$ to the first and second factors  respectively;
    \item $\mathrm{HE}(\Gamma_1')$ refers to the set of half-edges of $\Gamma_1'$, $k_{\mathfrak i}$ is the number of half-edges of $\Gamma_1'$ which is the same as the number of roots (or vertices) of $\Gamma_2'$;
    \item $N$ and $\ev_i$ are defined by \eqref{eqn:N} and 
    \eqref{eqn:ev_i} in Section \ref{sec:locrel_statement}, and $d_i$ is the weight\footnote{The weight of a root corresponds to the contact order of a relative marking.} of the root glued to the half-edge $e_i$. 
    \item $\pi_1, \pi_2$ and $f_1, f_2$ are respectively the universal curves and universal maps for  $\bM_{\Gamma_1'}(D_0)$ and $\bM^{\bullet}_{\Gamma_2'}(X,D)$ (similar to $\pi$ and $f$ used in Section \ref{sec:locrel_statement}).
\end{itemize}
We make the convention that $\bM_{\Gamma'_1}(D_0)=D$ when $\Gamma'_1$ consists of an unstable vertex. In this case, the whole $p_1^*(\dots)$ factor degenerates into $1$. 
We highlight a few key points regarding how to deduce this formula:

\begin{enumerate}
    \item Combining the degeneration formula and the localization formula, we have three levels $\Gamma^{(0)}_1, \Gamma_1^{(\infty)}$ and $\Gamma_2$ on the graph. To deduce this formula, we turn $\Gamma^{(0)}_1$ into $\Gamma_1'$ and combine $\Gamma_1^{(\infty)},\Gamma_2$ into $\Gamma_2'$. An illustration of an example can be the following\footnote{Each integer indicate the genus attached to the vertex. The red and blue vertices show clearly how the genera change after we combine $\Gamma_1^{(\infty)}$, $\Gamma_2$ into $\Gamma_2'$.}:
    
  \includegraphics[scale = 0.32]{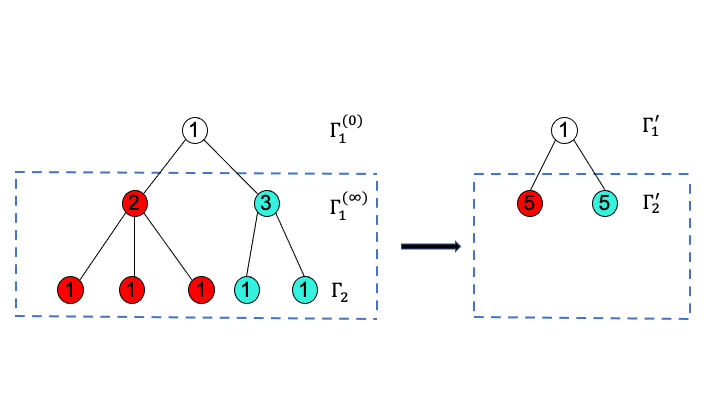}
   
    By Lemma \ref{lem:psi}, the sum over different decomposition of $\Gamma_2'$  corresponds to the divisor $\Psi$ written on the numerator in the factor \[p_2^*\left( \left( 1+\dfrac{\Psi}{-t-\Psi}\right) \cdot e(R^{1}(\pi_2)_*f_2^*\cO) \right).
    \]
    Once a decomposition of $\Gamma_2'$ is given, the $\Psi$ in the denominator becomes the corresponding target psi-class on the rubber moduli. Note that the extra summand ``$1$" corresponds to the simple fixed locus in the localization formula. 
    
    \item Recall that $N=N_{D/X}\oplus N^{\vee}_{D/X}$ under the anti-diagonal scaling action. The factor $e_{\C^*}(-R^{\bullet}(\pi_1)_*f_1^*N)$ comes from the moving part computation on vertices over $D_0$. $N_{D/X}$ comes from the log-tangent bundle of $P_0$, and $N^{\vee}_{D/X}$ comes from the restriction of $\ccL$ to $D_0$.
    \item The denominators
    \[
    \prod_{e_i\in \mathrm{HE}(\Gamma_1')} \left(\frac{t+\ev_i^*c_1(N_{D/X})}{d_i}-\psi_i\right),\quad -t-\Psi
    \]
    correspond to the smoothing of nodes and the smoothing of the target, respectively. 
    
    \item In the formula, the edges should contribute an automorphism factor $1/\prod_{i=1}^{k_\eta}d_i$. But the gluing of edge components with the relative markings on rubber components also contribute $\prod_{i=1}^{k_\eta}d_i$. These two factors cancel each other. Note that in the case of the simple fixed locus, there are no rubber components, but there is a factor $\prod_{i=1}^{k_\eta}d_i$ coming from the degeneration formula \eqref{eqn:deg}.
\end{enumerate}

One notices that the splitting $\mathfrak i$ already resembles the graphs of star type defined in Definition \ref{defn:star}. We are left to identify the formulae. It boils down to two things: matching suitable factors with the $C_v$ term defined in Equation \eqref{eqn:Cv}, and matching the rest with the $C_{v_i}$ terms defined in Equation \eqref{eqn:Cvi}.

\subsubsection{Matching with $C_v$}
By explicitly computing $\mathrm{Edge}(\mathfrak i)$, one can combine it with 
\begin{equation*}
p_1^*\left( \dfrac{e_{\C^*}(-R^{\bullet}(\pi_1)_*f_1^*N)}{\prod_{e_i\in \mathrm{HE}(\Gamma_1')} \left(\frac{t+\ev_i^*c_1(N_{D/X})}{d_i}-\psi_i\right)} \right)
\end{equation*}
to obtain $C_v$. Some details about computing $\mathrm{Edge}(\mathfrak i)$ are presented in the following.

$\mathrm{Edge}(\mathfrak i)$ encodes the deformation and obstruction contributions of edge components\footnote{It can be evaluated using the method of \cite[{Section 4}]{GP}. See also \cite{W}*{Appendix B}.}. Recall that edge components are those $\C^*$-invariant components which are mapped to fibers of $P_0$. For each half-edge $e_i$, we can construct an edge moduli space $\bM_{e_i}$ which consists of $\C^*$-invariant maps from a rational curve to a fiber of $P_0$ with multiplicity $d_i$. Obviously, $\bM_{e_i}$ can be parameterized by $D$. A priori, we should write 
$\bM_{\Gamma_1'}(D_0)\times_{D^{k_{\mathfrak i}}}\bM^{\bullet}_{\Gamma_2'}(X,D)$
as $\bM_{\Gamma_1'}(D_0)\times_{D^{k_{\mathfrak i}}}\prod_{e_i}\bM_{e_i}\times_{D^{k_{\mathfrak i}}}\bM^{\bullet}_{\Gamma_2'}(X,D)$. Then $\mathrm{Edge}(\mathfrak i)$ can be written as a product of contributions pulling back from different $\bM_{e_i}$(using the projections $p_{e_i}$).
Next, let us determine each individual contribution.

Let $\pi_{e_i}$ be the projection from the universal curves of $\bM_{e_i}$ and $f_{e_i}$ be the universal map. First, suppose that the unique vertex of $\Gamma_1'$ is stable. In this case, the contribution of $\bM_{e_i}$ to $\mathrm{Edge}(\mathfrak i)$ can be written as a product of three factors. 
The equivariant Euler class of the moving part\footnote{For a more precise definition, one should check \cite{GP}*{Section 1} for the definition of the virtual normal bundle, and look for the term ``moving part" after the proof of \cite{GP}*{Proposition 1}.} of $R^{\bullet}(\pi_{e_i})_*f_{e_i}^*TP_0(-\mathrm{log}\, D_\infty)$ gives the first factor
\begin{equation}\label{eqn:logfactor}
\dfrac{d_i^{d_i}}{(d_i)!(t+\ev_i^*c_1(N_{D/X}))^{d_i}}.
\end{equation}
The equivariant Euler class of the moving part of $R^{\bullet}(\pi_{e_i})_*f_{e_i}^*\left(\ccL|_{P_0}\right)$ contributes another factor:
\begin{equation}\label{eqn:midfac}
\dfrac{(d_i-1)!(-t-\ev_i^*c_1(N_{D/X}))^{d_i-1}}{d_i^{d_i-1}}.
\end{equation}
Finally, there is a factor $-(t+\ev_i^*c_1(N_{D/X}))^2$ coming from the gluing of the edge component and the component over $D_0$. Putting them together, we know that the contribution of $\bM_{e_i}$ to $\mathrm{Edge}(\mathfrak i)$ is $(-1)^{d_i}(t+\ev_i^*c_1(N_{D/X}))$. So we obtain
\[
\mathrm{Edge}(\mathfrak i)= \prod_{e_i\in \mathrm{HE}(\Gamma_1')}(-1)^{d_i}(t+\ev_i^*c_1(N_{D/X}))\,.
\]

Next, suppose that the vertex over $D_0$ is unstable (for example, the leading term in Theorem \ref{loc2rel}).
The valence-$2$ case (the unstable vertex has two half-edges) will contribute $0$ to the right-hand side of \eqref{eqn:combine} by the dimensional reason\footnote{In this case, one can check that the dimension of \[p_2^*\left( e(R^{1}(\pi_2)_*f_2^*\cO) \right)\cap [\bM_{\Gamma_1'}(D_0)\times_{D^{k_{\mathfrak i}}}\bM^{\bullet}_{\Gamma_2'}(X,D)]^{\vir} \]
is smaller than the dimension of $[\bM_g(\cO_X(-D),\beta)]^{\vir}$.}. When the unstable vertex has valence $1$, there is only one half-edge $e_1$ in $\mathrm{HE}(\Gamma_1')$. In this case, the gluing factor $-(t+ev_1^*c_1(N_{D/X}))^2$ disappears because the component over $D_0$ degenerates.
But there is an extra factor $(t+\ev^*_1c_1(N_{D/X}))/d_1$ coming from the moving part of the space of infinitesimal automorphisms of domain curves. Combining with \eqref{eqn:logfactor},\eqref{eqn:midfac}(setting $i=1$ there), we may conclude that in this case
\[
\mathrm{Edge}(\mathfrak i) = \frac{(-1)^{d_1-1}}{d_1} \,.
\]
This gives another explanation of the coefficient $\dfrac{(-1)^{\beta\cdot D-1}}{\beta\cdot D}$ appearing in the leading term of Theorem \ref{loc2rel} (comparing with Lemma \ref{lem:leadcoef}).

\subsubsection{Matching with $C_{v_i}$} First, each factor $(-1)^{g(v)}\lambda_{g(v)}$ of $C_{v_i}$ in \eqref{eqn:Cvi} comes from $e(R^{1}(\pi_2)_*f_2^*\cO)$.
To match the other factor in $C_{v_i}$ requires slightly more effort. 
For a vertex $v\in V(\Gamma)$, denote by $\gamma_v$ the graph consisting of a single vertex $v$ plus all the decorations on $v$. The goal now is to rewrite the pushforward of 
\[\bigg(  1+\dfrac{\Psi}{-t-\Psi} \bigg) \cap [\bM^\bullet_{\Gamma_2'}(X,D)]^\vir\]
in terms of pushforward classes from the product $\prod_{v\in V(\Gamma_2')}\bM_{\gamma_v}(X,D)$ and match it with the remaining factor of $C_{v_i}$. More precisely, we need the following lemma.

\begin{lem}\label{lem_split2}
We have the following identity:
\begin{align*}
&\tau_*\bigg(\dfrac{t}{t+\Psi} \cap [\bM^\bullet_{\Gamma_2'}(X,D)]^\vir\bigg) \\
=& \tau'_*\bigg(\prod_{v\in V(\Gamma_2')} p_v^*\bigg( \dfrac{t}{t+\Psi} \bigg) \cap \big[\prod_{v\in V(\Gamma_2')}\bM_{\gamma_v}(X,D)\big]^\vir\bigg),
\end{align*}
where $p_v$ is the projection to the factor corresponding to $v$, and the two $\Psi$ on both sides are the target psi-classes on their corresponding moduli spaces, and $\tau, \tau'$ are the corresponding stabilization maps to $\prod_{v\in V(\Gamma_2')}\bM_{g(v),1}(X,b(v))\times_{X} D$.
\end{lem}

Lemma \ref{lem_split2} is a special case of Lemma \ref{lem_split1}. Each factor on the right-hand side of Lemma \ref{lem_split2} matches with the $\dfrac{t}{t+d_i\bar\psi+\bar\ev^*c_1(N_{D/X})}$ part of $C_{v_i}$ because of the following two lemmas.

\begin{lem}
Let $e$ be a root with multiplicity $d$.
Then, we have the following identity on the moduli of relative stable maps $\bM_{\Gamma}(X,D)$:
\[
\Psi=d\psi_e+\ev_e^*c_1(N_{D/X}),
\]
where $\psi_e$ is the psi-class of the relative marking $e$ defined using the universal curve over $\bM_{\Gamma}(X,D)$.
\end{lem}
For a proof, see \cite{Ka}*{Theorem 5.13.1}.

\begin{lem}
If $\Gamma$ is a connected admissible graph with only one root. Let $\psi$ be the psi-class of the relative marking and recall that $\bar\psi$ is the pullback of the psi-class from the moduli of stable maps to $X$. Then, we have $\psi=\bar\psi$.
\end{lem}
\begin{proof}
Because $\Gamma$ has only one root and $D$ is nef, when the target expands, there can only be one component in the rubber. Due to the stability condition on the rubber, the relative marking cannot lie on an unstable component (a multiple cover of a fiber of a rubber target with only $2$ relative markings). Thus, the stabilization does not contract components containing the relative marking. As a result, the psi-class and the pullback psi-class are the same.
\end{proof}




\section{The case of log Calabi-Yau surfaces} 
\label{sec:locrel_log_K3}

Let $S$ be a smooth projective surface over 
$\C$ and $E$ a smooth effective anticanonical divisor on $S$ which is nef. By the adjunction formula, $E$ is a genus $1$ curve. In this section, we prove Theorem 
\ref{thm_locrel_SE} expressing the local series 
$F_g^{K_S}$ in terms of the relative series 
$F_g^{S/E}$ and the stationary series 
$F_{g,\ba}^E$ of the elliptic curve $E$.
We use systematically the notation introduced in 
Section \ref{sec:log_K3_intro}.

\subsection{Specializing Theorem \ref{loc2rel} to the pair $(S,E)$}
Using Theorem \ref{loc2rel}, we obtain a relation between the local invariants $N_{g,\beta}^{K_S}$
and the relative invariants $N_{g,\beta}^{S/E}$.

Let $N_{E/S}$ be the normal bundle to $E$ in $S$. 
We consider the rank $2$ vector bundle
$N \coloneqq N_{E/S}\oplus N_{E/S}^{\vee}$ over $E$
and the anti-diagonal scaling action of $\GM$ on $N$ with weight $1$ on $N_{E/S}$ and weight $-1$ on $N_{E/S}^{\vee}$. 
We denote by $t$ the corresponding equivariant parameter.
For every $d_E \geq 0$ and 
$\bd=(d_1,\cdots,d_n)$, we set
\[N_{h,d_E}^{\Etw}(\bd)
\coloneqq \int_{[\bM_{h,n}(E,d_E)]^{\vir}} \left( \prod_{j=1}^n \frac{t \ev_j^*\omega}{t-d_j \psi_j} \right) e_{\GM}(-R^{\bullet}\pi_*f^{*}N)\]
where $\omega$ is a point class for $E$. Note that the equivariant classes
\[\prod_{j=1}^n \frac{t \ev_j^*\omega}{t-d_j \psi_j}
\quad \text{and}\quad e_{\GM}(-R^{\bullet}\pi_*f^{*}N)\]
have equivariant degrees $n$ and $2h-2$\footnote{The equivariant degree of $e_{\GM}(-R^{\bullet}\pi_*f^{*}N)$ can be computed by the Riemann-Roch formula.} respectively. The summation of their degrees equals to the virtual dimension of $[\bM_{h,n}(E,d_E)]^{\vir}$. So the degree of $N_{h,d_E}^\Etw(\bd)\in \Q(t)$ must be zero. Note that the equivariant parameter $t$ has degree $1$. It further implies that $N_{h,d_E}^\Etw(\bd)\in \Q$, that is, does not depend on $t$.


\begin{prop}
For every $\beta \in H_2(S,\Z)$ such that 
$\beta \cdot E>0$ we have
\[N_{g,\beta}^{K_{S}}
= \frac{(-1)^{\beta \cdot E-1}}{\beta
\cdot E} N_{g,\beta}^{S/E}+
 \]
\[
\sum_{n\geq 0}\sum_{\substack{g=h+g_1+\dots+g_n\\
\beta=d_E[E]+\beta_1+\dots +\beta_n \\
d_E\geq 0,\,\beta_j\cdot E>0}}  
\frac{N_{h,d_E}^{\Etw}(\boldsymbol{\beta} \cdot E)}{|\Aut (\boldsymbol{\beta},\bg)|}
\prod_{j=1}^n \left((-1)^{\beta_j \cdot E}
(\beta_j \cdot E)
N_{g_j,\beta_j}^{S/E}\right)
\]
where 
\[ \boldsymbol{\beta} \cdot E =(\beta_1\cdot E,\dots,\beta_n \cdot E) \,,\]
\[ \frac{1}{|\Aut (\boldsymbol{\beta},\bg)|}\coloneqq \frac{1}{|\Aut ((\beta_1,g_1),\cdots,(\beta_n,g_n))|}\,.\]
\end{prop}

\begin{proof}
We apply Theorem \ref{loc2rel}.
Let $\sG\in G_{g,\beta}$ with $|V_1|=n$,
$g=h+\sum_{j=1}^n g_j$, $\beta=d_E[E]
+\sum_{j=1}^n \beta_j$. We denote $d_j
=\beta_j \cdot E$.
The contribution of $\sG$
is 
\[(\tau_{\sG})_* \left[p_v^*C_v\prod_{v_j\in V_1}p_{v_j}^*C_{v_j}\right]\]
where $\tau_{\sG}$ is the gluing map
\[\left(\prod_{v_j\in V_1}
\bM_{g_j}(S/E,\beta_j)
\right)\times_{E^n}\bM_{h,n}(E,d_E)\longrightarrow \bM_g(S,\beta).\]
According to \eqref{eqn:Cvi}, we have
\begin{equation*}
C_{v_j}=\frac{t}{t+d_j\bar{\psi}+\bar{\ev}^* c_1(N_{E/S})}(-1)^{g_j}\lambda_{g_j}\,.
\end{equation*}
The key point is that 
$\dim [\bM_{g_j}(S/E,\beta_j)]^{\vir}=g_j$.
Therefore the insertion of $(-1)^{g_j} \lambda_{g_j}$ already eats up all the dimension of $[\bM_{g_j}(S/E,\beta_j)]^{\vir}$ and so
$C_{v_j}$ reduces to $(-1)^{g_j} \lambda_{g_j}$ and the only possibly non-vanishing contribution of the class of the diagonal $E \hookrightarrow E \times E$ defining the gluing 
$\tau_{\sG}$ is $p_{v_j}^* 1 \times p_v^* \omega$ on 
$\bM_{g_j}(S/E,\beta_j)
\times \bM_{h,n}(E,d_E)$, where 
$1 \in H^0(E)$ and $\omega \in H^2(E)$.
Thus, the contribution of $\bM_{g_j}(S/E,\beta_j)$ is 
\[ \int_{[\bM_{g_j}(S/E,\beta_j)]^{\vir}}
(-1)^{g_j} \lambda_j =N_{g_j,\beta_j}^{S/E}\]
and the contribution of 
$\bM_{h,n}(E,d_E)$
is 
\[ \int_{[\bM_{h,n}(E,d_E)]^{\vir}}C_v \prod_{j=1}^n \ev_j^{*} \omega \,.\]
According to \eqref{eqn:Cv}, we have
\begin{equation*}
C_v=e_{\GM}(-R^{\bullet}\pi_*f^*N)\prod_{j=1}^n\frac{(t+\ev_j^*c_1(N_{E/S}))(-1)^{d_j} d_j}{t+\ev_j^*c_1(N_{E/S})-d_j\psi_j}\,.
\end{equation*}
As $c_1(N_{E/S})\cup \omega=0$ in $A^{\bullet}(E)$, 
we have
\[ C_v \prod_{j=1}^n \ev_j^{*} \omega 
=e_{\GM}(-R^{\bullet}\pi_*f^*N)\prod_{j=1}^n 
\frac{t (-1)^{d_j} d_j}{t-d_j\psi_j} \ev_j^{*}\omega\,,
\]
and so the contribution of $\bM_{h,n}(E,d_E)$ is indeed 
$N_{h,d_E}^{\Etw}(\bd)\prod_{j=1}^n (-1)^{d_j} d_j$.
\end{proof}

Expanding the denominator\footnote{The first few terms of the expansion are
\[\prod_{j=1}^n \frac{t}{t-d_j \psi_j}= 1+\frac{1}{t}\sum_{j=1}^n d_j\psi_j+\frac{1}{t^2}\left(\sum_{j=1}^n d_j^2\psi_j^2+\sum_{1\leq i< j\leq n}d_id_j\psi_i\psi_j\right)+\dots\]} in the formula for $N_{h,d_E}^\Etw(\bd)$,
and using that 
$\dim [\bM_{h,n}(E,d_E)]^{\vir}=2h-2+n$, 
we get 
\[ N_{h,d_E}^\Etw(\bd)=\sum_{\substack{\ba=(a_1,\dots,a_n)\\ a_j \geq 0\,, \sum_j a_j \leq 2h-2 }}
N_{h,d_E,\ba}^\Etw 
\prod_{j=1}^n
d_j^{a_j}\,,\]
where, for every 
$\ba=(a_1,\dots,a_n)$,
\[ N_{h,d_E,\ba}^\Etw \coloneqq \frac{1}{t^{\sum_{j=1}^n a_j}}
\int_{[\bM_{h,n}(E,d_E)]^{\vir}} \left( \prod_{j=1}^n  \ev_j^*(\omega) \psi_j^{a_j} \right) e_{\GM}(-R^{\bullet}\pi_*f^{*}N)\,.\]

Therefore, we have 
\[N_{g,\beta}^{K_{S}}
= \frac{(-1)^{\beta \cdot E-1}}{\beta
\cdot E} N_{g,\beta}^{S/E}+
 \]
\[
\sum_{n\geq 0} \!\!\!   \sum_{\substack{g=h+g_1+\dots+g_n\\
\beta=d_E[E]+\beta_1+\dots +\beta_n \\
d_E\geq 0,\,\beta_j\cdot E>0}}  
\sum_{\substack{\ba=(a_1,\dots,a_n)\\ a_j \geq 0\,, \sum_j a_j \leq 2h-2 }}
\frac{N_{h,d_E,\ba}^{\Etw}}{|\Aut (\boldsymbol{\beta},\bg)|} 
\prod_{j=1}^n \left((-1)^{\beta_j \cdot E}
(\beta_j \cdot E)^{a_j+1}
N_{g_j,\beta_j}^{S/E}\right)\,.
\]

Using generating series, we package the above recursive formula as follows. Let
\begin{equation}\label{eq:F_bar_K_S}
\bar {F}_g^{K_{S}}\coloneqq \sum_{\substack{\beta \\ \beta \cdot E >0}}N_{g,\beta}^{K_{S}}Q^{\beta}\,,
\end{equation}
\begin{equation}\label{eq:F_bar_SE}
\bar{F}_g^{S/E}\coloneqq \sum_{\substack{\beta \\ \beta \cdot E >0}}\frac{(-1)^{\beta \cdot E+g-1}}{\beta \cdot E}N_{g,\beta}^{S/E}Q^{\beta}\,,
\end{equation}
\begin{equation}\label{eq:F_E_tw}
F_{h,\ba}^\Etw \coloneqq -\frac{\delta_{h,1}\delta_{n,0}}{24}\log ((-1)^{E \cdot E}\tilde{\cQ}) + \sum_{d_E \geq 0}
N_{h,d_E,\ba}^{\Etw} ((-1)^{E \cdot E}\tilde{\cQ})^{d_E} \,,
\end{equation}
where the variable $\tilde{\cQ}$ in $F_{h, \ba}^\Etw$ is related to the variable $Q$ in $\bar{F}_{g}^{S/E}$ and $\bar{F}_g^{K_S}$ by the formula \eqref{eq:cQ}.

\begin{prop} \label{prop_gen_series_bar}
\[\bar{F}_g^{K_S}  = (-1)^g \bar{F}_g^{S/E}
+\frac{\delta_{g,1}}{24} \log Q^{E}
+\]
\[ 
\sum_{n\geq 0} 
\sum_{\substack{g=h+g_1+\dots+g_n,
\\ \ba=(a_1,\dots,a_n)\in \Z_{\geq 0}^n \\ 
(a_j,g_j) \neq (0,0),\,
\sum_{j=1}^n a_j \leq 2h-2} } 
\frac{F_{h, \ba}^\Etw}{|\Aut (\ba, \bg)|}
\prod_{j=1}^n   (-1)^{g_j-1} D^{a_j+2} \bar{F}_{g_j}^{S/E} \,.\]
\end{prop}

\begin{proof}
Given $\bg=(g_1,\dots,g_n)$ and $\ba=(a_1,\dots,a_n)$,
we have a disjoint sum decomposition
\[ \{1,...,n\}=I_{\bg,\ba} \coprod J_{\bg,\ba} \]
where $I_{\bg,\ba}$ is the subset of $j$ such 
$(g_j,a_j) \neq (0,0)$ and 
$J_{\bg,\ba}$ is the subset of $j$ such that 
$(g_j,a_j)=0$.
Denote $\ba' =(a_j)_{j \in I_{\bg,\ba}}$. 
If $a_j=0$, then there is no insertion of $\psi_j$
in $N^{\Etw}_{h,d_E,\ba}$, we can remove
$\ev_j^{*}(\omega)$ from the integral using the divisor equation and so 
we have $N^{\Etw}_{h,d_E,\ba}=(\prod_{j \in J_{\bg,\ba}} 
d_E) N^{\Etw}_{h,d_E,\ba'}$. There is one exception: we cannot apply the divisor equation if $d_E=0$, $n=1$, $h=0$, $g_1=0$ and $a_1=0$, in which case 
\[ N_{1,0,(0)}^{\Etw}=\int_{[\bM_{1,1}(E,0)]^{\vir}} \ev^{*}(\omega)=-\frac{1}{24}\,.\]
It follows that the correct general relation is 
\begin{equation}
\label{eq_relation}
    N^{\Etw}_{h,d_E,\ba}=-\frac{\delta_{h,1}\delta_{n,1}\delta_{\ba,(0)}}{24}+ N^{\Etw}_{h,d_E,\ba'} \prod_{j \in J_{\bg,\ba}} 
d_E \,.
\end{equation}
It follows that we can replace the sum over $\ba$
by a sum over
$\ba'$.

After summing over $\beta$ to form generating series, the factors indexed by $j \in J_{g,\ba}$ are absorbed in 
$F_{h,\ba}^{E,\tw}$ via the change of variables 
$Q \mapsto \cQ$. Indeed, according to 
formula \ref{eq:cQ2}, we have 
\[ \tilde{\cQ}=(-1)^{E\cdot E} \exp(-D^2 F_0^{S/E})=(-1)^{E\cdot E} Q^E \exp(-D^2 \bar{F}_0^{S/E})\]
and so
\[ ((-1)^{E \cdot E} \tilde{\cQ})^{d_E}
= Q^{d_E E} \sum_{l \geq 0} 
\frac{1}{l!}(-1)^l (d_E D^2 \bar{F}^{S/E}_0)^l \,. \]
\end{proof}

Recall from
\eqref{eq:F_K_S} and \eqref{eq:F_SE}
that

\[F_g^{K_S}=
-\frac{1-\delta_{(E\cdot E),0}}{3!(E\cdot E)^2}(\log Q^E)^3 \delta_{g,0}
+
\left(\frac{1-\delta_{(E\cdot E),0}}{(E\cdot E)}\frac{\chi(S)}{24}-\frac{1}{24}\right)(\log Q^E) \delta_{g,1} +\bar{F}_g^{K_S} \]
and
\[F_g^{S/E}=
-\frac{1-\delta_{(E\cdot E),0}}{3!(E\cdot E)^2}(\log Q^E)^3 \delta_{g,0}
-
\frac{1-\delta_{(E\cdot E),0}}{(E\cdot E)}\frac{\chi(S)}{24}(\log Q^E) \delta_{g,1} +\bar{F}_g^{S/E} \,.\]

\begin{prop} \label{prop_gen_series}
\[
    F_g^{K_S}  = (-1)^g F_g^{S/E}+\]
\[    \sum_{n\geq 0}  \!\!\! \!\!\! \!\!\! 
\sum_{\substack{g=h+g_1+\dots+g_n,
\\ \ba=(a_1,\dots,a_n)\in \Z_{\geq 0}^n \\ 
(a_j,g_j) \neq (0,0),\,
\sum_{j=1}^n a_j \leq 2h-2} } \!\!\! \!\!\! \!\!\! 
\frac{F_{h, \ba}^\Etw}{|\Aut (\ba, \bg)|}
\prod_{j=1}^n   (-1)^{g_j-1} 
\left( D^{a_j+2} F_{g_j}^{S/E}
+ (E \cdot E) \delta_{g_j,0} \delta_{a_j,1}\right)\,.
\]
\end{prop}

\begin{proof}
We rewrite Proposition
\ref{prop_gen_series_bar}
in terms of the series 
$F_g^{K_S}$ and $F_g^{S/E}$.
One needs to use that 
$F_1^{K_S}-\bar{F}_1^{K_S}
=-(F_1^{S/E}-\bar{F}_1^{S/E})
-\frac{1}{24}\log Q^E$
and $D^{a+2} \bar{F}_0^{S/E}
=D^{a+2} F_0^{S/E}
+(E \cdot E) \delta_{a,1}$ for 
$a \geq 1$. 
\end{proof}

\subsection{Twisted Gromov-Witten theory of the elliptic curve}
In this section, we compute the twisted Gromov-Witten series $F_{g,\ba}^{\Etw}$ of the elliptic curve in terms of the untwisted Gromov-Witten series $F_{g,\ba}^E$. 
Recall from \eqref{eq:F_E_tw}
that
\[ F_{g,\ba}^\Etw = -\frac{\delta_{g,1}\delta_{n,0}}{24}\log ((-1)^{E \cdot E}\tilde{\cQ}) + \bar{F}_{g,\ba}^{\Etw} \]
where
\[\bar{F}_{g, \ba}^\Etw \coloneqq \sum_{d_E \geq 0} 
\frac{ ((-1)^{E \cdot E} \tilde{\cQ})^{d_E}}{t^{\sum_{j=1}^n a_j}}
\int_{[\bM_{g,n}(E,d_E)]^{\vir}}
\left( \prod_{j=1}^n
\ev_j^*(\omega) \psi_j^{a_j}
\right)
e_{\GM}(-R^{\bullet}\pi_*f^{*}N)\,,\]
and from \eqref{eq:F_E} that 
\[ F_{g,\ba}^E = -\frac{\delta_{g,1}\delta_{n,0}}{24}\log ((-1)^{E \cdot E}\tilde{\cQ}) + \bar{F}_{g,\ba}^E \]
where 
\[\bar{F}_{g, \ba}^E \coloneqq \sum_{d_E \geq 0} 
\tilde{\cQ}^{d_E}
\int_{[\bM_{g,n}(E,d_E)]^{\vir}}
\prod_{j=1}^n
\ev_j^*(\omega) \psi_j^{a_j}
\,.\]

\begin{prop}
\label{prop_localcurve}
For every $\ba=(a_1,\cdots,a_n)
\in \Z_{\geq 0}^n$ such that $\sum_{j=1}^n a_j \leq 2g-2$,  
the twisted Gromov-Witten theory of the elliptic curve is related with the untwisted one via
$$
\bar{F}_{g, \ba}^\Etw = (-1)^{g-1} \frac{(E\cdot E)^m}{m!} 
\bar{F}_{g,( \ba,1^m)}^E
$$
where $m \coloneqq 2g-2-\sum_{j=1}^n a_j$ and
$(\ba, 1^m)=(a_1,\dots,a_n,\underbrace{1,\cdots, 1}_m)$.
\end{prop}


The proof of Proposition \ref{prop_localcurve} takes the remaining of this section.

Let $W$ be a rank $r$ vector bundle over a projective variety $X$. Let 
$\rho_i$ be the Chern roots of $W$. 
We fix a fiberwise action of $\C^{*}$ on 
$W$ and we denote by $\lambda_i$ the corresponding equivariant parameters.

Coates and Givental \cite{CG} have expressed in
terms of the Givental's quantization formalism the computation of the $W$-twisted Gromov-Witten theory in terms of the untwisted one obtained by applying the
Grothendieck-Riemann-Roch theorem to the 
universal curve over moduli spaces of stable maps \cite{Mum, FP}.

More precisely, according to the main result\footnote{Theorem 1 
of \cite{CG} computes the theory 
twisted by an arbitrary multiplicative characteristic class $\exp(\sum s_k \ch_k)$ and Corollary 1 of \cite{CG} describes
the twist by $e_{\GM}
(R^{\bullet} \pi_{*}f^{*}W)$.
In order to get the twist by
$e_{\GM}
(-R^{\bullet} \pi_{*}f^{*}W)=1/e_{\GM}
(R^{\bullet} \pi_{*}f^{*}W)$, we have to flip the sign of the coefficients $s_k$.} of \cite{CG},
(the stable part of) the generating function of the $W$-twisted
Gromov-Witten invariants
$$
  \int_{[\bM_{g,n}(X,d)]^{\vir}}
  \left(\prod_{j=1}^n \ev_j^*(\tau_j) \psi_j^{a_j}\right) e_{\GM}(-R^{\bullet}\pi_*f^{*}W) 
$$
can be computed via the quantization of the symplectic operator
$$
\Delta(z)\coloneqq \prod_{i=1}^r \sqrt{\lambda_i+\rho_i} \, \exp \Big( -  \sum_{ m>0} \frac{B_{2m}}{2m(2m-1)} \frac{z^{2m-1}}{(\lambda_i+\rho_i)^{2m-1}} \Big)
$$
together with the  quantization of the symplectic operator 
$$
\Psi(z)\coloneqq \prod_{i=1}^r \exp \Big( -\frac{\rho_i \ln \lambda_i}{z} + \frac{1}{z} \sum_{k>0} \frac{(-1)^{k-1} \rho_i^{k+1}}{k(k+1) \lambda_i^k} \Big)
$$
acting on the untwisted Gromov-Witten potential. 


For the case relevant to 
Proposition \ref{prop_localcurve}, 
$W$ is the rank $2$ vector bundle
$N= N_{E/S}\oplus N_{E/S}^{\vee}$
and we have
$$
\lambda_1=t,\quad \rho_1= (E \cdot E)\, \omega,\quad
\lambda_2=-t,\quad \rho_2= -(E \cdot E)\, \omega \,.
$$
In this case, the expressions for the symplectic operators become much simpler. Indeed, most of the terms are odd under 
$\rho_i \mapsto -\rho_i$, $\lambda_i \mapsto -\lambda_i$ and so most of the terms with $i=1$ cancel pairwise with the terms with $i=2$.
We have
$$
\Psi(z) = \exp\left( \frac{\log(-1)(E \cdot E) \,\omega }{z}  \right)
$$
and $\Delta(z)$ does not depend on $z$:
$$
\Delta(z) =\sqrt{-1}(t+(E \cdot E)\, \omega )\,.
$$

Using the divisor equation, we see that the  operator 
$\Psi$
acts like the change of variables (c.f. Remarks under Theorem $1'$ in \cite{CG})
$$
\tilde{\cQ} \rightarrow \tilde{\cQ} \, 
e^{\log (-1) (E \cdot E) \int_E \omega}  =  (-1)^{(E\cdot E)}\tilde{\cQ}\,.
$$

The action of $\Delta(z)$ can be viewed 
as an $R$-matrix action on the CohFT defined via the Gromov-Witten theory of elliptic curve. 



We recall the definition of the $R$-matrix (quantization) action (c.f. \cite{Giv5}, \cite{PPZ}, see also Sec.~2 of \cite{CGLL2}). Suppose that we have two
symplectic vector spaces $V, V'$ of the same dimension,
with units $1$, $1'$, and a symplectic transformation $R(z)\in \Hom(V,V')[[z]]$. Here by the symplectic condition $R(z)R^*(-z)=I$, we see that $R(z)$ is invertible and
$$
R^{-1}(z) = R^*(-z)\,.
$$
For each $2g-2+n>0$, let $F_{g,n}(-)=\langle -  \rangle_{g,n}$ be the (given) correlation functions with insertions in $V[[z]]$. We define $\hat R F_{g,n}$  by the following graph sum
$$
\hat R F_{g,n}(\tau_1 \psi_1^{k_1},\cdots,\tau_n \psi_n^{k_n}) :=
 \sum_{\Gamma \in G_{g,n}}\frac{1}{\Aut(\Gamma)} \Cont_{\Gamma}
$$
where $G_{g,n}$ is the set of stable graphs with genus $g$, $n$ legs, and the contributions $\Cont_\Gamma$ are defined via the following construction:
\begin{itemize}
    \item at each leg $l$ of $\Gamma$, we place an insertion
    $$
    R^{-1}(\psi_l)\tau_l \psi_l^{k_l} ;
    $$    
    \item at each edge $e=(v_1,v_2)$ of $\Gamma$, we place a bi-vector as a two-direction insertion
    $$ \qquad
    \frac{\sum_\alpha  e_\alpha \otimes e^\alpha -\sum_\alpha  R^{-1}(\psi_{(e,v_1)})e'_\alpha \otimes R^{-1}(\psi_{(e,v_2)}) e'{}^{\alpha} }{\psi_{(e,v_1)}+\psi_{(e,v_2)}}
    $$
    where $\{e_\alpha\}, \{e'_\alpha\} $ are arbitrary  bases of $V,V'$, with the dual basis $\{e^\alpha\}$ and $\{e'{}^{\alpha}\}$, respectively;
    \item at each vertex $v$ of $\Gamma$, we place the map
    \begin{multline*} \qquad \qquad
    \tau'_1 \psi_1^{l_1}\otimes\cdots \otimes\tau'_{n_v} \psi_{n_v}^{l_{n_v}}  \mapsto \\
    \sum_{k \geq 0} \frac{1}{k!} F_{g_v,n_v+k}\big( \tau'_1 \psi_1^{l_1},\cdots,\tau'_{n_v} \psi_{n_v}^{l_{n_v}}, T(\psi_{n_v+1}) ,\cdots, T(\psi_{n_v+k}) \big) ,
    \end{multline*}
    where $(\tau'_i,l_i)\in \{(\tau_1,k_1),\cdots,(\tau_n,k_n)\}$, and $T(z):=z \mathbf 1 -z\,R^{-1}(z)\mathbf 1'$.
\end{itemize}
In our case, $V=H^\bullet(E)(t)$
with pairing $(\alpha,\beta)=\int_E \alpha \wedge \beta$,   $
 e_{\mathbb C^*}(N) = (\lambda_1+\rho_1) (\lambda_2+\rho_2) = -(t+(E \cdot E)\omega)^2
$, and
$V'=H^\bullet(E)(t)$ with pairing 
$(\alpha,\beta)'=\int_E  {\alpha} \wedge 
 {\beta} \wedge e^{-1}_{\mathbb C^*}(N)=\int_E \frac{\alpha}{\sqrt{-1}(t+(E \cdot E)\omega)} \wedge 
\frac{\beta}{\sqrt{-1}(t+(E \cdot E)\omega)}$. The units $\mathbf 1$, $\mathbf 1'$  are both equal to $1\in H^0(E)$. 
The $R$-matrix
$R(z)=\Delta$
is independent of $z$  so the edge bi-vector is simply zero. Hence the contribution is non-vanishing only if the graph is a single vertex. 
Namely
\begin{multline*}
 \left< \tau_1 \psi_1^{a_1},\cdots,\tau_n \psi_n^{a_n} \right>^{E,\tw}_{g,n} = 
 \\
\sum_{k\geq 0} \frac{1}{k!} \left<R^{-1}(\psi_1)\tau_1 \psi_1^{a_1},\cdots,R^{-1}(\psi_n)\tau_n \psi_n^{a_n} , T(\psi_{n+1}),\cdots, T(\psi_{n+k}) \right>^E_{g,n+k}|_{\tilde{\cQ}\mapsto (-1)^{(E\cdot E)}\tilde{\cQ}}
\end{multline*}
where $\tau_j\in H^{\bullet}(E)$,
\[\left< \tau_1 \psi_1^{a_1},\cdots,\tau_n \psi_n^{a_n} \right>^{E,\tw}_{g,n}
\coloneqq
\sum_{d_E\geq 0}\tilde{\cQ}^{d_E}\int_{[\bM_{g,n}(E,d_E)]^{\vir}}
\left( 
\prod_{j=1}^n\psi_j^{a_j}\ev_j^*(\tau_j)
\right)
e_{\GM}(-R^{\bullet}\pi_*f^{*}N) ,\]
\[\left< \tau_1 \psi_1^{a_1},\cdots,\tau_n \psi_n^{a_n} \right>^{E}_{g,n}
\coloneqq
\sum_{d_E \geq 0}\tilde{\cQ}^{d_E}\int_{[\bM_{g,n}(E,d)]^{\vir}}
\prod_{j=1}^n\psi_j^{a_j}\ev_j^*(\tau_j)\,.\]
We have 
\[ R(z)=\Delta(z)=\sqrt{-1}(t+(E \cdot E)\omega) \,,\]
\[ R^{-1}(z)= \frac{1}{\sqrt{-1}t}-\frac{(E \cdot E) \omega}{\sqrt{-1} t^2}\,,\]
\[ T(z) =
z \left(1-\frac{1}{\sqrt{-1}t}
+\frac{(E \cdot E) \omega}{\sqrt{-1}t^2} \right)\,.
\]
We focus on the case where all the insertions 
$\tau_i$ are point classes $\omega$. We have
\begin{multline*}
 \left< \omega \psi_1^{a_1},\cdots,\omega \psi_n^{a_n} \right>^{E,\tw}_{g,n} = 
 \\
\sum_{k\geq 0} \frac{1}{k!} \left<R^{-1}(\psi_1)\omega \psi_1^{a_1},\cdots,R^{-1}(\psi_n)\omega \psi_n^{a_n} , T(\psi_{n+1}),\cdots, T(\psi_{n+k}) \right>^E_{g,n+k}|_{\tilde{\cQ}\mapsto (-1)^{(E\cdot E)}\tilde{\cQ}}\,.
\end{multline*}
For every $j$, we have $R^{-1}(\psi_j)\omega=\frac{1}{\sqrt{-1}t}\omega$
and so 
\begin{multline*}
 \left< \omega \psi_1^{a_1},\cdots,\omega \psi_n^{a_n} \right>^{E,\tw}_{g,n} = 
 \\
 \frac{1}{(\sqrt{-1}t)^n}
\sum_{k\geq 0} \frac{1}{k!} \left<\omega \psi_1^{a_1},\cdots,\omega \psi_n^{a_n} , T(\psi_{n+1}),\cdots, T(\psi_{n+k}) \right>^E_{g,n+k}|_{\tilde{\cQ}
\mapsto (-1)^{(E\cdot E)}\tilde{\cQ}}\,.
\end{multline*}
Writing 
\[ T(\psi_j) =
\left(1-\frac{1}{\sqrt{-1}t}\right)\psi_j
+ \frac{(E \cdot E)}{\sqrt{-1}t^2}\psi_j \omega
\]
and expanding, we obtain 
\begin{multline*}
 \left< \omega \psi_1^{a_1},\cdots,\omega \psi_n^{a_n} \right>^{E,\tw}_{g,n} = 
 \\
 \frac{1}{(\sqrt{-1}t)^n}
\sum_{m,l\geq 0} \frac{1}{m!l!} \left(\frac{(E \cdot E)}{\sqrt{-1}t^2}\right)^{m}
\left( 1-\frac{1}{\sqrt{-1}t}\right)^l \times\\
\left<\omega \psi_1^{a_1},\cdots,\omega \psi_n^{a_n},
\omega\psi_{n+1},\cdots,\omega \psi_{n+m},
\psi_{n+m+1},\cdots, \psi_{n+m+l} \right>^E_{g,n+m+l}|_{\tilde{\cQ}\mapsto (-1)^{(E\cdot E)}\tilde{\cQ}}\,.
\end{multline*}

The sum over $l$ can be evaluated using the dilaton equation
\begin{align*}
& \  \sum_{l\geq 0} \frac{1}{l!} 
\left( 1-\frac{1}{\sqrt{-1}t}\right)^l
\left< -,\psi_{n+m+1} ,\cdots,\psi_{m+m+l}  \right>^E_{g,n+m+l}\\
&=   \ 
\sum_{l \geq 0} \binom{-(2g-2+n+m)}{l} \left(\frac{1}{\sqrt{-1}t}-1\right)^l \left< -   \right>^E_{g,n+m} \\
&= \
\left( \frac{1}{\sqrt{-1}t}\right)^{ - (2g-2+n+m) } \left< -   \right>^E_{g,n+m}\,.
\end{align*}
Notice that in the second equality we have used the following generalized Newton's binomial theorem:  
$$
 \sum_{k>0}  \binom{\alpha}{k} x^k =(1+x)^{\alpha}.
$$
Hence, collecting the powers of $\sqrt{-1}$
and $t$, we obtain
\begin{multline*} 
     \left<  \omega \psi_1^{a_1},\cdots,\omega \psi_n^{a_n} \right>^{E,\tw}_{g,n} = 
     (\sqrt{-1} )^{2g-2} \times
    \\
    \sum_{m \geq 0} \frac{(E \cdot E)^m}{m!}
    t^{2g-2-m} \left< \omega \psi_1^{a_1},\cdots,\omega \psi_n^{a_n},
    \omega\psi_{n+1},\cdots,
    \omega \psi_{n+m}\right>^E_{g,n+m}|_{\tilde{\cQ}\mapsto (-1)^{(E\cdot E)}\tilde{\cQ}}.  
\end{multline*}
By dimension constraint, the correlator in the sum is non-vanishing only if $m=2g-2-\sum_{j=1}^n a_j$. Therefore, we have
\begin{multline*} 
     \left<  \omega \psi_1^{a_1},\cdots,\omega \psi_n^{a_n} \right>^{E,\tw}_{g,n} = 
     (-1)^{g-1}  t^{\sum_{j=1}^n a_j}
      \frac{(E \cdot E)^m}{m!} 
      \times
    \\
    \left< \omega \psi_1^{a_1},\cdots,\omega \psi_n^{a_n},
    \omega\psi_{n+1},\cdots,
    \omega \psi_{n+m}\right>^E_{g,n+m}|_{\tilde{\cQ}\mapsto (-1)^{(E\cdot E)}\tilde{\cQ}}  
\end{multline*}
where  $m=2g-2-\sum_{j=1}^n a_j$. Note that the term $t^{\sum_{j=1}^n a_j}$ cancels with the same term in the denominator of $\bar{F}_{g, \ba}^\Etw$.
This concludes the proof of 
Proposition \ref{prop_localcurve}.

\subsection{Conclusion of the proof of Theorem \ref{thm_locrel_SE}}



By Proposition \ref{prop_gen_series},
we have \[
    F_g^{K_S}  = (-1)^g F_g^{S/E}+\]
\[    \sum_{n\geq 0}  \!\!\!\! \!\!\!\! \!\!\!\!
\sum_{\substack{g=h+g_1+\dots+g_n,
\\ \ba=(a_1,\dots,a_n)\in \Z_{\geq 0}^n \\ 
(a_j,g_j) \neq (0,0),\,
\sum_{j=1}^n a_j \leq 2h-2} }  \!\!\!\! \!\!\!\! \!\!\!\!
\frac{1}{|\Aut (\ba, \bg)|}
F_{h, \ba}^\Etw 
\prod_{j=1}^n   (-1)^{g_j-1} 
\left( D^{a_j+2} \bar{F}_{g_j}^{S/E}
+ (E \cdot E) \delta_{g_j,0} \delta_{a_j,1}\right)\,.
\]

For every $\bg=(g_j)_j$
and  $\ba=(a_j)_j$ such that 
$(a_j,g_j) \neq (0,0)$ for every $j$
and
$\sum_j a_j \leq 2h-2$, we denote
$m =2h-2-\sum_{j=1}^n a_j$ and we define
\[ \tilde{\bg} \coloneqq (g_1,\cdots,g_n, \underbrace{0,\cdots,0}_m) \]
and 
\[ \tilde{\ba} \coloneqq (a_1,\cdots, a_n, 
\underbrace{1,\cdots, 1}_m ) \,.\]
We have $(\tilde{g}_j,\tilde{a}_j) \neq (0,0)$ for every 
$j$ and $\sum_j \tilde{a}_j
=2h-2$.

According to Proposition \ref{prop_localcurve}, we have 
\[ F^{\Etw}_{h,\ba}=(-1)^{h-1}\frac{(E \cdot E)^m}{m!}
F^E_{h,\tilde{a}}\,.\]

Therefore, we can rewrite the sum over $\ba=(a_j)_j$
with $\sum_j a_j \leq 2h-2$ as a sum over $\ba=(\bar{a}_j)_j$
with $\sum_j \bar{a}_j =2h-2$. More precisely, let 
$\bar{g}=(\bar{g}_j)_j$ and 
$\bar{a}=(\bar{a}_j)$ be of the form 
\[ \bar{\bg} = (\bar{g}_1,\cdots,\bar{g}_k, \underbrace{0,\cdots,0}_l) \]
\[ \bar{\ba}=(\bar{a}_1,\dots,\bar{a}_k,\underbrace{1,\cdots, 1}_l)\]
with $(\bar{a}_j,\bar{g}_j) \neq (0,0), (1,0)$ for $1 \leq j \leq k$, 
$g=h+\sum_{j=1}^k\bar{g}_j$ and $\sum_{j=1}^{k+l} \bar{a}_j=2h-2$. The total contribution of
$(\bar{\ba},\bar{\bg})$ in the rewritten formula is the sum of contributions of $(\ba,\bg)$ with 
$(\tilde{\ba},\tilde{\bg})=(\bar{\ba},\bar{\bg})$.
Such $(\ba,\bg)$ is of the form
\[ \bg = (\bar{g}_1,\cdots,\bar{g}_k, \underbrace{0,\cdots,0}_{l-m}) \]
\[ \ba=(\bar{a}_1,\dots,\bar{a}_k,\underbrace{1,\cdots, 1}_{l-m})\]
with $0 \leq m \leq l$. 
Using that 
\[ |\Aut(\ba,\bg)|=|\Aut(\{(\bar{a}_1,\bar{g}_1),\cdots,(\bar{a}_k,\bar{g}_k)\})|(l-m)!\,,\] 
we get that the total contribution of $(\bar{\ba},\bar{\bg})$ is given by
\begin{multline*}
\frac{(-1)^{h-1} F_{h,\bar{a}}^E}{|\Aut(\{(\bar{a}_1,\bar{g}_1),\cdots,(\bar{a}_k,\bar{g}_k)\})|}\left(\prod_{j=1}^k (-1)^{\bar{g}_j-1} D^{\bar{a}_j+1}F_{\bar{g}_j}^{S/E}\right) \times \\
\sum_{m=0}^l (-1)^{l-m}\frac{(E \cdot E)^m(D^3F_0^{S/E}+(E \cdot E))^{l-m}}{m!(l-m)!}.
\end{multline*}
Using the binomial theorem, this can be rewritten as 
\begin{multline*}
\frac{(-1)^{h-1} F_{h,\bar{a}}^E}{|\Aut(\{(\bar{a}_1,\bar{g}_1),\cdots,(\bar{a}_k,\bar{g}_k)\})|}\left(\prod_{j=1}^k (-1)^{\bar{g}_j-1} D^{\bar{a}_j+1}F_{\bar{g}_j}^{S/E}\right)
 (-1)^{l}\frac{(D^3F_0^{S/E})^{l}}{l!}\,.
\end{multline*}
As $|\Aut(\bar{\ba},\bar{\bg})|=|\Aut(\{(\bar{a}_1,\bar{g}_1),\cdots,(\bar{a}_k,\bar{g}_k)\})|l!$, we finally obtain
\[  \frac{(-1)^{h-1} F_{h,\bar{a}}^E}{|\Aut(\bar{\ba},\bar{\g})|}
\left(\prod_{j=1}^{k+l} (-1)^{\bar{g}_j-1} D^{\bar{a}_j+1}F_{\bar{g}_j}^{S/E}\right)\,.\]
This ends the proof of Theorem \ref{thm_locrel_SE}.



\subsection{Stationary Gromov-Witten theory of the elliptic curve}
\label{sec:gw_elliptic_curve}

According to Theorem \ref{thm_locrel_SE}, the local series $F_g^{K_S}$
and the relative series $F_g^{S/E}$
determine each other through 
the stationary series $F_{g,\ba}^E$ of the elliptic curve. In this section, we review the computation by Okounkov and Pandharipande \cite[\S 5]{OP1} of the stationary Gromov-Witten theory of the elliptic curve.

Recall that, for every $(a_1,\dots,a_n) 
\in \Z_{\geq 0}^n$, we denote by
\[\bar{F}_{g, \ba}^E \coloneqq \sum_{d_E \geq 0} 
\tilde{\cQ}^{d_E}
\int_{[\bM_{g,n}(E,d_E)]^{\vir}}
\prod_{j=1}^n
\ev_j^*(\omega) \psi_j^{a_j}
\]
the generating series of stationary Gromov-Witten invariants of the elliptic curve $E$.

For every $k\geq 1$, we consider the Eisenstein series
$$
E_{2k}(\tilde{\tau})=
1-\frac{4k}{B_{2k}}	 \sum_{n=1}^{\infty} \frac{n^{2k-1} \tilde{\cQ}^n}{1-\tilde{\cQ}^n} .
$$
where $\tilde{\cQ}=e^{2i\pi \tilde{\tau}}$
and the Bernoulli numbers $B_{2k}$ are
defined by $\frac{t}{e^t-1}=\sum_{n \geq 0} B_n \frac{t^n}{n!}$.
As functions on the upper half-plane $\{ \tilde{\tau} \in \C| \Ima \tilde{\tau}>0 \}$, $E_{2k}$ is modular of weight $2k$ for $SL(2,\Z)$ for every $k \geq 2$, and $E_2$ is quasimodular of weight $2$ for $SL(2,\Z)$. The ring $\C[E_2,E_4,E_6]$, graded by the weight, is exactly the graded ring
$\QMod(SL(2,\Z))$ of quasimodular forms for $SL(2,\Z)$ \cite{Zagier}.
For every $k \geq 0$, we denote by 
$\C[E_2,E_4,E_6]_k$ the weight $k$ subspace of $\C[E_2,E_4,E_6]$.

\begin{thm}[Okounkov-Pandharipande \cite{OP1}]\label{thm_quasimod_eliptic_curve}
For every $g \geq 0$, $n \geq 1$ and $\ba=(a_1,\dots,a_n) \in \Z_{\geq 0}^n$, we have 
\[ \bar{F}_{g,\ba}^E \in 
\Q[E_2,E_4,E_6]_{\sum_{j=1}^n (a_j+2)} \,.\]
\end{thm}

In fact, Okounkov and Pandharipande give an explicit formula computing $\bar{F}_{g,\ba}^E$ as a polynomial in $E_2$, $E_4$, $E_6$.

For every $n \geq 1$, let 
$$
F(z_1,\cdots,z_n) := \delta_{1,n} z_1^{-1}+ \sum_{\substack{\ba=(a_1,\cdots,a_n)\in \Z_{\geq 0}^n,\\ 2g-2 =\sum_{j=1}^n a_j }}  \bar{F}_{g,\ba}^E \cdot z_1^{a_1+1}\cdots z_n^{a_n+1}
$$
and let $F^\bullet(z_1,\cdots,z_n)$ be the disconnected generating function defined by
\footnote{The condition $I_j \neq \emptyset$ excludes genus-$1$ unmarked connected components and so only series $F(z_1,\dots, z_n)$ with $n \geq 1$ enter the formula. For $n=0$, the series 
$-\frac{1}{24}+\bar{F}_{1,\emptyset}^E
=-\log \eta$ is not quasimodular.}
$$
F^\bullet(z_1,\cdots,z_n) :=  \sum_{k>0} \sum_{\substack{I_1\sqcup \cdots \sqcup I_k = \{1,\cdots,n\}\\
I_j \neq \emptyset} }   \frac{1}{|\Aut (I_1,\cdots,I_k)|} \prod_{j=1}^k  F (z_{I_j}) \,.
$$
This relation can be inverted to compute the connected series $F(z_1,\dots,z_n)$
in terms of the disconnected ones.

We introduce the odd theta function
$$
\vartheta(\tilde{\tau},z):= \vartheta_{\frac{1}{2},\frac{1}{2}}(\tilde{\tau},z)=\sum_{k=-\infty}^{\infty} (-1)^k e^{(k+\frac{1}{2})z} e^{\pi i \tilde{\tau} (k+\frac{1}{2})^2 } 
$$ 
and following Bloch and Okounkov \cite{Ok} we denote 
\begin{align*}
 \Theta(z) &
 \coloneqq \eta(\tilde{\tau})^{-3} \vartheta(\tilde{\tau},z) \\ 
&=z \exp  
\left(  \sum_{k=1}^{\infty} \frac{B_{2k}}{2k (2k)!} E_{2k}(\tilde{\tau}) z^{2k} \right)
\end{align*}

\begin{thm}[Okounkov-Pandharipande \cite{OP1}]
\label{thm_ellipticGW}
For every $n \geq 1$, we have
$$
 {F^\bullet(z_1,\cdots,z_n)}  =  \sum_{\text{$\sigma$ permutation} \atop \text{of $z_1\cdots,z_n$}} 
 \frac{\det \Big[\frac{\Theta^{(j-i+1)}(z_{\sigma(1)}+\cdots+z_{\sigma(n-j)})}{(j-i+1)!}\Big]_{i,j=1}^n}{\Theta(z_{\sigma(1)})\Theta(z_{\sigma(1)}+z_{\sigma(2)})\cdots \Theta(z_{\sigma(1)}+\cdots+z_{\sigma(n)})}  
$$
where $\Theta^{(k)}$ is the $k$-th derivative of $\Theta$ with respect to $z$ and $\frac{\Theta^{(j-i+1)}}{(j-i+1)!}$ is interpreted as $0$ if $j-i+1<0$.
\end{thm}
For example, using that $\Theta'(0)=1$ and 
$\Theta''(0)=0$, we get
\begin{align*}
F(z) = &\  F^\bullet(z) =  \ \frac{1}{\Theta(z)} = {z}^{-1}-\frac{E_{{2}}}{24} \, z+ \left( {\frac {E_{{4}}}{2880}}+{\frac {{E_{{2
}}}^{2}}{1152}} \right) {z}^{3} \nonumber\\
 & \ \qquad \qquad \qquad -  \left( {\frac {E_{{6}}}{181440}}+{
\frac {E_{{2}}E_{{4}}}{69120}}+{\frac {{E_{{2}}}^{3}}{82944}} \right) 
{z}^{5}+\cdots\\
F(z_1,z_2) =& \  F^\bullet(z_1,z_2) -F^\bullet(z_1) F^\bullet(z_2) \\
=&  \frac{1}{\Theta(z_1+z_2)}\left( \frac{\Theta'(z_1)}{\Theta(z_1)}+\frac{\Theta'(z_2)}{\Theta(z_2)}\right) -\frac{1}{\Theta(z_1)\Theta(z_2)} \\
= & \ -\frac{ (E_2^2-E_4) \, z_1 z_2}{288} +{\frac {\left( 5\,{E_{{2}}}^{
3}-3\,E_{{2}}E_{{4}}-2\,E_{{6}} \right)  z_1^2 z_2^2}{25920}}\\
& \ +   {\frac { \left( 5\,{E_{{2}}}^{3}-E_{{2}}E_{{4}}-4\,E_{{6}
} \right) (z_{{1}}z_2^3+z_1^3 z_2)}{34560}}+\cdots
\end{align*}
Hence, we obtain
\begin{align}
F_{1,(0)}^E=&-\frac{E_2}{24} \,, \label{eq_F10}\\
F_{1,(0^2)}=&-\frac{1}{12 \cdot 24}(E_2^2-E_4) \,, \label{eq_F100} \\
     F_{2,(1^2)}^E= &  -\frac{1}{9^2 \cdot 320} (2E_6+3E_2E_4-5E_2^3)\,, \label{eq_F211} \\
     F_{2,(2)}^E \    =&  \frac{1}{9 \cdot 640} (2E_4+5E_2^2)\,. \label{eq_F22}
\end{align}

\section{Finite generation and quasimodularity for $(\PP^2,E)$}

\label{sec:fg_quasimod}

\subsection{Quasimodular forms for $\Gamma_1(3)$}
\label{section:quasimod}

We refer to \cite{DS,Zagier} for basics on modular and quasimodular forms.
For every $\Gamma$ congruence subgroup of $SL(2,\Z)$, we denote by $\Mod(\Gamma)=\bigoplus_{k \geq 0} \Mod(\Gamma)_k$
the ring of modular forms for $\Gamma$, graded by the weight, and by $\QMod(\Gamma)=\bigoplus_{k \geq 0} \QMod(\Gamma)_k$ the ring of quasimodular forms for $\Gamma$, also graded
by the weight.

We focus on the congruence subgroups $\Gamma_1(3)$
and $\Gamma_0(3)$ defined by





\[ \Gamma_1(3)=
\Big\{ \begin{pmatrix}
 a & b \\
 c & d
\end{pmatrix}
\in SL(2,\Z)|
\begin{pmatrix}
 a & b \\
 c & d
\end{pmatrix}
=
\begin{pmatrix}
 1& * \\
 0 & 1
\end{pmatrix}
\mod 3
\Big\}\,,\]

\[ \Gamma_0(3)=
\Big\{ \begin{pmatrix}
 a & b \\
 c & d
\end{pmatrix}
\in SL(2,\Z)|
\begin{pmatrix}
 a & b \\
 c & d
\end{pmatrix}
=
\begin{pmatrix}
 *& * \\
 0 & *
\end{pmatrix}
\mod 3
\Big\}\,.\]
We have $\Gamma_1(3) \subset \Gamma_0(3)$, subgroup 
of index $2$, with $-I \in \Gamma_0(3)$ and $-I \notin \Gamma_1(3)$. It follows that the modular (resp. quasimodular)
forms for $\Gamma_0(3)$ are exactly the modular 
(resp. quasimodular) forms for $\Gamma_1(3)$ of even weight.






Let 
\begin{align*}
A(\tau)&\coloneqq \left( \frac{\eta(\tau)^9}{\eta(3\tau)^3}+27\frac{\eta(3\tau)^9}{\eta(\tau)^3} \right)^{\frac{1}{3}}=1 + 6 \cQ + 6 \cQ^3 + 6 \cQ^4 + 12 \cQ^7 + 6 \cQ^9 + \dots \\
C(\tau)&\coloneqq \frac{\eta(\tau)^9}{ \eta(3\tau)^3}
=1 - 9 \cQ + 27 \cQ^2 - 9 \cQ^3 - 117 \cQ^4 + 216 \cQ^5 + 27 \cQ^6 - 450 \cQ^7 + \dots\,
\end{align*}
where 
$\eta(\tau)\coloneqq \cQ^{\frac{1}{24} } \prod_{n=1}^\infty (1-\cQ^n)$
is the Dedekind eta function and $\cQ=e^{2i\pi \tau}$.
It follows from the known modular properties of the 
eta function that $A$ and $C$ are modular forms for $\Gamma_1(3)$ respectively of weight $1$ and $3$. We refer to section 3 of \cite{Maier2} for details\footnote{In terms of the notation $\cA_3$ and $\cB_3$ of \cite{Maier2}, we have $A=\cA_3$ and $C=\cB_3^3$.}.
The combination 
\[ \frac{A^3 - C}{27}=\frac{\eta(3\tau)^9}{\eta(\tau)^3}
=\cQ + 3 \cQ^2 + 9 \cQ^3 + 13 \cQ^4 + 24 \cQ^5 + 27 \cQ^6 + 50 \cQ^7 +\dots\]
 is a cusp form of weight $3$ for $\Gamma_1(3)$.

\begin{lem}\label{lem_alg_ind_A_C}
The functions $A$ and $C$ are algebraically independent over $\C$.
\end{lem}

\begin{proof}
If there exists a non-trivial polynomial relation between 
$A$ and $C$, then there exists a non-trivial polynomial relation between $A$ and $C$ which is homogeneous for the weight and so of the form 
$\sum_{n+3m=k} a_{nm}A^n C^m$ for some $k$. If such non-trivial relation existed, then, dividing by $A^k$, we would get that $C/A^3$ is solution of a non-trivial polynomial equation with complex coefficients and so $C/A^3$
would be constant, which is not the case as $C/A^3=1-27\cQ+O(\cQ^2)$.
\end{proof}

According to \cite[Figure 3.4]{DS}\footnote{In \cite{DS}, the formula is a priori valid only for $k \geq 2$. We get that 
$\dim \Mod(\Gamma_1(3))_1=1$ because it is $>0$ by the existence of $A$, and $\leq 1$ as $\dim \Mod(\Gamma_1(3))_2=1$.}, the dimension of $\Mod(\Gamma_1(3))_k$ is $\lfloor \frac{k}{3}\rfloor+1$. Therefore, Lemma \ref{lem_alg_ind_A_C}
implies\footnote{Therefore, $A$ is the unique weight $1$ modular form for $\Gamma_1(3)$ with constant term $1$, and so can be described as the weight $1$ Eisentein series $E_1^{\psi,1}
=1+6 \sum_{n \geq 1} \sum_{d|n} \psi(d) \cQ^n$, where $\psi$ is the non-trivial character of $(\Z/3)^{*} \simeq \{\pm 1\}$
(extended by $\psi(d)=0$ if $3|d$), or as the theta series of the hexagonal lattice $\Z[\sqrt{-3}]$. In particular, the coefficient of $\cQ^n$ in $A$ is the number of $(x,y)\in \Z^2$
such that $|x+e^{\frac{2i\pi}{3}}y|^2=x^2-xy+y^2=n$
(see \cite[Exercise 4.11.5]{DS}).} that 
\[ \Mod(\Gamma_1(3))=\C[A,C]\,.\]

By Proposition 20 of \cite{Zagier}, the Eisenstein series $E_2(\tau)$ for $SL(2,\Z)$ is algebraically independent of 
$\Mod(\Gamma_1(3))$ over $\C$ and 
\[ \QMod(\Gamma_1(3))=\Mod(\Gamma_1(3))[E_2]\,.\]
The difference $3E_2(3\tau)-E_2(\tau)$ is modular for $\Gamma_1(3)$. Using that the space of weight $2$ modular forms for $\Gamma_1(3)$ is of dimension $1$, we find that $3E_2(3\tau)-E_2(\tau)=2A^2$. Therefore,  we can use 
\[ B(\tau)\coloneqq \frac{1}{4}\big( E_2(\tau) +3E_2(3\tau))
=1-6\cQ-18\cQ^2-42\cQ^3-42\cQ^2+\cdots\]
instead of $E_2$ as depth 1 weight 2 generator and we have 
\[ \QMod(\Gamma_1(3))=\C[A,B,C]\,.\]

The ring $\Q[A,B,C]$ is closed under the differential operator 
$$\partial_\tau:= \textstyle  \frac{1}{2\pi i} \frac{d}{d\tau}=\cQ \frac{d}{d \cQ}.$$ 
Indeed, $\partial_\tau$ maps
weight $k$ quasimodular forms to weight $k +2$ quasimodular forms by a straightforward 
calculation from the definition of quasimodularity, and in particular for
$\QMod(\Gamma_1(3))=\C[A,B,C]$
this allows to show easily
the Ramanujan-type identities 
\begin{equation}\label{eqn_Ramid}
\begin{split}
\partial_\tau A = \frac{1}{6} A&(B+A^2) -\frac{C}{3},\\
\partial_\tau B = \frac{1}{6} (B^2-A^4),&\qquad
\partial_\tau C = \frac{1}{2} C(B-A^2)\,,
\end{split}
\end{equation}
because the quasi-modular properties of both sides are
known and so it is enough to identify finitely many terms of the $\cQ$-expansions.

\subsection{Mirror of local $\PP^2$ and quasimodular forms}
We review the relation between the mirror family of local $\PP^2$ and modular forms, following \cite{ASYZ, Zho, CI}. 

Let $\mathbb{H}=\{ \tau \in \C | \Ima \tau >0\}$ be the upper half-plane. We consider the modular curve $Y_1(3)=[\mathbb{H}/\Gamma_1(3)]$. It is a smooth orbifold, whose coarse moduli space can be identified with $\{ q \in \C | q \neq -\frac{1}{27},0\} \cup \{ \infty \}$, and with a single $\Z/3$-orbifold point at $q=\infty$. The modular curve $Y_1(3)$ has two cusps, given by $q=0$ and $q=-\frac{1}{27}$
(corresponding respectively to the $\Gamma_1(3)$-equivalence classes of 
$\tau=i\infty$ and $\tau=0$). In the context of mirror symmetry, where $Y_1(3)$ is viewed as the stringy K\"ahler moduli space of local $\PP^2$, the point $q=0$ is the large volume point, $q=-\frac{1}{27}$ is the conifold point and $q=\infty$ is the orbifold point.

The coordinate $q$ on $Y_1(3)$ is expressed
in terms of $\tau$ \cite{Maier1}\footnote{The description of the Hauptmodul of the genus $0$ modular curve $Y_1(3)$ in terms of eta functions goes back to Klein and Fricke at the end of the 19th century.} by 
\begin{equation}\label{eq:hauptmodul}
\frac{1}{1+27q}=1+27 \frac{\eta(3 \tau)^{12}}{\eta(\tau)^{12}}\,,
\end{equation}
that is, denoting $X \coloneqq (1+27q)^{-1}$, by
\begin{equation}
X=\frac{A^3}{C} \,.
\end{equation}
The periods of the universal family of elliptic curves with $\Gamma_1(3)$-level structure are solution of the Picard-Fuchs equation
\cite{Maier2} 
\begin{equation}\label{PF_ODE}
\left[\left(q \frac{d}{dq}\right)^2
-3q
\left(3 \left(q \frac{d}{dq}\right) +1 \right)
\left(3 \left(q \frac{d}{dq}\right) +2 \right) \right]\Pi=0 \,.
\end{equation}

In Section \ref{localp2generators}, we defined the functions $S$, $I_{11}=q\frac{dI_1}{dq}$, $I_{12}=q\frac{dI_2}{dq}$
in terms of solutions $I_1$ and $I_2$ of the differential equation \eqref{mirror_ODE} describing genus $0$ mirror symmetry for local $\PP^2$. As the
differential equation \eqref{mirror_ODE} is obtained from \eqref{PF_ODE} by applying 
$q \frac{d}{dq}$ and flipping the sign of $q$, we deduce that $\tau$, viewed as a multivalued function of $q$, is given by 
\begin{equation} 
\tau =\frac{1}{2}+\frac{1}{2\pi i}\frac{I_{12}(q)}{I_{11}(q)}\,,
\end{equation}
that is,
\begin{equation} \label{eq:cQ_formula}
\cQ = e^{2 \pi i \tau}=-\exp \left(\frac{I_{12}(q)}{I_{11}(q)} \right) \,.
\end{equation}
One should not confuse the variables $q$, $\cQ$, $Q$:
\begin{itemize}
\item $q$ is such that $(1+27q)^{-1}$ is a globally defined coordinate on $Y_1(3)$\,,
\item $\cQ=e^{2i\pi \tau}$ is the flat modular coordinate for the family of elliptic curves parametrized by $Y_1(3)$\,,
\item $Q$ is the flat coordinate determined 
by the mirror of local $\PP^2$.
\end{itemize}

The variables $q$ and $\cQ$ are related by 
\eqref{eq:hauptmodul}-\eqref{eq:cQ_formula}.
The variables $q$ and $Q$ are related by the mirror transformation \eqref{eq:mirror_map}.

According to
\cite{ASYZ,Maier1,Maier2,Zho}, the functions $X$, $I_{11}$, $S$ and the quasimodular forms $A$, $B$, $C$ determine each other through the identities
\begin{equation}\label{eq:ABS_SIX}
    A = I_{11},\quad B = \frac{I_{11}^2}{X} (X+6S),\quad
    C= \frac{I_{11}^3}{X} \,,
\end{equation}
of inverse
\begin{equation}\label{eq:SIX_ABC}
X=\frac{A^3}{C},\quad I_{11}=A,\quad 
S=\frac{1}{6}\frac{AB-A^3}{C}\,.
\end{equation}
In particular, viewed as functions of $\tau$, $X$ is a modular function of weight $0$ for $\Gamma_1(3)$, $I_{11}$ is a modular form of weight $1$ for $\Gamma_1(3)$ and $S$ is a quasimodular function of weight $0$ for $\Gamma_1(3)$.

\begin{lem}\footnote{In \cite{LhoP}, the algebraic independence of $S$, $X$, $I_{11}$ is not considered as known, and explicit lifts of functions to the ring $\Q[S,X]$ are constructed. This extra work is not necessary by Lemma \ref{lem_alg_ind_S_X}.} \label{lem_alg_ind_S_X}
The functions $S$, $X$, $I_{11}$
are algebraically independent over $\C$.
\end{lem}

\begin{proof}
It is a direct corollary of the algebraic independence of the quasimodular forms $A$, $B$, $C$ reviewed in Section \ref{section:quasimod}.
\end{proof}

By Lemma \ref{lem_alg_ind_S_X}, the ring of functions generated by $S$ and $X$ is the polynomial ring
 $$
\bR\coloneqq \mathbb Q[X,S]\,.
 $$
We define a grading on $\bR$
by $\deg X=\deg S =1$ 
and denote by $\bR_{\leq k}$ the subspace of polynomials with degree no more than $k$. 
 
We consider the vector space 
$[ X^{-(g-1)} \cdot \bR_{\leq 3g-3} ]^{\reg}\,,
$ where $[ - ]^\reg$   (the ``orbifold regularity" condition) is defined by
\begin{equation}
    [ - ]^{\reg}:=\{f(X,S) :  3 \deg_X f +  \deg_S f \geq 0\}  .
\end{equation}

\begin{prop} \label{prop_spaces_equality}
The expression of $S$ and $X$ in terms of the quasimodular forms $A$, $B$, $C$
induces an identification 
\[ [ X^{-(g-1)} \cdot \bR_{\leq 3g-3} ]^{\reg}
=C^{-(2g-2)} \cdot \Q[A,B,C]_{6g-6}\]
for every $g \geq 2$.
\end{prop}

\begin{proof}
We first prove that 
$[ X^{-(g-1)} \cdot \bR_{\leq 3g-3} ]^{\reg} \subset C^{-(2g-2)} \cdot \Q[A,B,C]_{6g-6},.$ Let $X^{-(g-1)} \cdot X^j S^k \in [ X^{-(g-1)} \cdot \bR_{\leq 3g-3} ]^{\reg}$.
According to \eqref{eq:SIX_ABC}, the $C$-degree of $X^{-(g-1)} \cdot X^j S^k$ is $-j-k+(g-1)\geq -(2g-2)$, using that $j+k \leq 3g-3$ by definition of $\bR_{\leq 3g-3}$. The $B$-degree is $k \geq 0$. The $A$-degree of each term is $\geq -(3g-3)+3j+k$, which is $\geq 0$ by definition of $[-]^{\reg}$. Therefore, $X^{-(g-1)} \cdot X^j S^k = C^{-(2g-2)}f(A,B,C)$ for some 
$f \in \Q[A,B,C]$. As $X^{-(g-1)} \cdot X^j S^k$ is of weight $0$ and $C$ of weight $3$, $f$ is of weight $6g-6$.

We prove that conversely $ C^{-(2g-2)} \cdot \Q[A,B,C]_{6g-6} \subset [ X^{-(g-1)} \cdot \bR_{\leq 3g-3} ]^{\reg}$. 
Let $C^{-(2g-2)} \cdot A^m B^n C^l \in
C^{-(2g-2)} \cdot \Q[A,B,C]_{6g-6}$.
According to \eqref{eq:ABS_SIX}, the power of $I_{11}$ is equal to the weight.
As $C^{-(2g-2)} \cdot A^m B^n C^l$ is of weight $0$, $C^{-(2g-2)} \cdot A^m B^n C^l$ is independent of $I_{11}$ and is only a function of $S$ and $X$.
The $S$-degree of  $C^{-(2g-2)} \cdot A^m B^n C^l$ is $n \geq 0$. 
The $X$-degree of each term in $C^{-(2g-2)} \cdot A^m B^n C^l$ is $\geq (2g-2)-n-l
$ and $\leq (2g-2)-l$. Therefore, $\deg_X + \deg_S \leq (2g-2)-l \leq 2g-2$.
On the other hand, $3 \deg_X + \deg_S
\geq (6g-6)-3n-3l+n=(6g-6)-2n-3l$, which is 
$\geq 0$ as  $A^m B^n C^l$ is of weight $6g-6$, $B$ is of weight $2$ and $C$ is of weight $3$.
\end{proof}

\begin{rmk}
It follows from the proof of Proposition 
\ref{prop_spaces_equality} that the constraint defined by $[-]^{\reg}$ is equivalent to the absence of negative powers of $A$. According to \cite{Maier2}, $A$ has a zero at the orbifold point $q=\infty$. Therefore, $[-]^{\reg}$ is indeed the condition imposed by the regularity at the orbifold point.
\end{rmk}

We consider the differential operator 
\[ D=3Q\frac{d}{dQ} \,.\]

\begin{lem}
\begin{align}
D & = 3I_{11}^{-1}\cdot  q \frac{d}{dq} \label{D_Q_q}\\
&= 3 C^{-1}\cdot  \cQ \frac{d}{d\cQ} \,.\label{D_q_cQ}
\end{align}
\end{lem}

\begin{proof}
The equality \eqref{D_Q_q} is clear as $Q=e^{I_{1}(q)}$ by \eqref{eq:mirror_map} and \eqref{def_I11}.

In order to prove \eqref{D_q_cQ}, we use 
\eqref{eq:cQ_formula}:
\[ \cQ \frac{d}{d\cQ}=I_{22}^{-1} q \frac{d}{dq}\]
where 
\[ I_{22} \coloneqq q \frac{d}{dq}
\left( \frac{I_{12}(q)}{I_{11}(q)} \right)\,.\]
Theorem 2 in \cite{ZaZi08} shows that
$
I_{22}=\frac{X}{I_{11}^2}$.
We conclude using that
$C=\frac{I_{11}^3}{X}$ according to \eqref{eq:ABS_SIX}.
\end{proof}

\begin{lem} \label{lem_D_weight}
For every $n,k \geq 0$, we have 
\[ D( C^{-n} \cdot \Q[A,B,C]_k )
\subset C^{-(n+1)} \cdot \Q[A,B,C]_{k+2} \,.\]
\end{lem}

\begin{proof}
This follows from \eqref{D_q_cQ} and 
\eqref{eqn_Ramid}. The important point is that $\partial_\tau C$ is divisible by $C$.
\end{proof}

\begin{lem}\label{lem_SX_derivative}
We have
\begin{equation} \label{eq:SX_derivative}\textstyle
    q\frac{d}{dq} S = -S^2+\frac{X-1}{3}\,S-\frac{X(X-1)}{9}\,,\quad 
q\frac{d}{dq} X = X(X-1)\,.
\end{equation}
\end{lem}

\begin{proof}
The formula for $X$ is clear as $X=(1+27q)^{-1}$ by definition.

The formula for $S$ follows from the differential equation \eqref{mirror_ODE}.
Alternatively, one can use the expression 
\eqref{eq:SIX_ABC} in terms of quasimodular forms and \eqref{eqn_Ramid}-\eqref{D_q_cQ}.
\end{proof}



\subsection{Quasimodular forms: from $SL(2,\Z)$ to $\Gamma_1(3)$}

As reviewed in Section \ref{sec:gw_elliptic_curve}, the ring of quasimodular forms for $SL(2,\Z)$ is generated by the Eisenstein series
$E_2(\tilde{\tau})$, $E_4(\tilde{\tau})$,
$E_6(\tilde{\tau})$.
On the other hand, we have seen in Section \ref{section:quasimod} that the ring of quasimodular forms for $\Gamma_1(3)$ is generated by the functions $A(\tau)$, 
$B(\tau)$, $C(\tau)$.

The following result shows that after the change of variables $\tilde{\tau}=3\tau$ the ring of quasimodular forms for $SL(2,\Z)$ embeds in the ring of quasimodular form for $\Gamma_1(3)$.

\begin{prop} \label{prop_from_sl2_to_gamma}
We have the embedding of graded rings
\[ \Q[E_2(3 \tau),E_4(3\tau),E_6(3 \tau)]
\subset \Q[A(\tau),B(\tau),C(\tau)]\,.\]
Explicitly, we have the identities
\begin{align}
  3 \,  E_2(3\tau) \ =& \ \  {2B(\tau)+A(\tau)^2}   \  \   =   \ \frac{3\, {I_{11}}^2}{X} (X+4S),\quad \nonumber \\
   9\,  E_4(3\tau) \  =& \ \    {A(\tau)^4+8\,A(\tau)C(\tau)}  \   = \  \frac{{I_{11}}^4}{X} (X+8),\quad \label{modularformtogenerator} \\ \quad  
   27\,  E_6(3\tau) \  = & \ \ {-A(\tau)^6+20\,A(\tau)^3C(\tau)+8\, C(\tau)^2}  \ = \  \frac{- {I_{11}}^6}{ X^2} (X^2-20 X-8) .  \qquad \nonumber 
\end{align}
\end{prop}

\begin{proof}
For every $n \geq 1$, if $f(\tau)$ is a modular (resp.\ quasimodular) form for $SL(2,\Z)$ of weight $k$ then $f(n \tau)$ is a modular 
(resp.\ quasimodular) form for $\Gamma_0(n)$
of weight $k$.

Once we know the modularity properties of each side, the identities 
\eqref{modularformtogenerator} are easy to prove: it is enough to match finitely many terms of the $\cQ$-expansions. Expressions in terms of $X$, $I_{11}$ and $S$ follow from 
\eqref{eq:ABS_SIX}.
\end{proof}





\subsection{Genus $0$ invariants of $(\PP^2,E)$}
Applying Theorem \ref{thm_locrel_SE}
for $g=0$ (which reduces in this case to the genus $0$ local-relative 
correspondence of \cite{GGR}), we get
\[ F_0^{\PP^2/E}=F_0^{K_{\PP^2}} \,.\]

\begin{lem}
\begin{align}
&DF_0^{K_{\PP^2}}=-I_2\, \label{eq:D_F0}\\ 
&D^2 F_0^{K_{\PP^2}} = -3 \frac{I_{12}}{I_{11}} \, \label{eq:D2_F0}\\
&D^3 F_0^{K_{\PP^2}} = - 9C^{-1}=- \frac{9X}{I_{11}^3} \label{eq:D3_F0}\,.
\end{align}
\end{lem}

\begin{proof}
The formula \eqref{eq:D_F0} is
the genus $0$ mirror theorem for $K_{\PP^2}$ \cite{Giv4, LLY,CKYZ}.
Formula \eqref{eq:D2_F0} follows directly from \eqref{eq:D_F0} and \eqref{D_Q_q}.

Taking the derivative of \eqref{eq:D2_F0}, we obtain
$D^3 F_0^{K_{\PP^2}}=-\frac{9I_{22}}{I_{11}}$
where 
\[ I_{22}= q \frac{d}{dq}
\left( \frac{I_{12}(q)}{I_{11}(q)} \right)\,.\]
Theorem 2 in \cite{ZaZi08} shows that
$
I_{22}=\frac{X}{I_{11}^2}$.
We end the proof of \eqref{eq:D3_F0} using that
$C=\frac{I_{11}^3}{X}$ according to \eqref{eq:ABS_SIX}.
\end{proof}

\begin{prop} \label{prop_Dn_F0}
For every $n \geq 1$, we have 
\[ D^{n+2} F_0^{K_{\PP^2}} \in C^{-n} 
\Q[A,B,C]_{2n-2} \,.\]
\end{prop}

\begin{proof}
The case $n=1$ is clear by \eqref{eq:D3_F0}. The general case follows by induction on $n$
from Lemma \ref{lem_D_weight}.
\end{proof}

For example, using \eqref{D_Q_q}-\eqref{eq:SX_derivative} or \eqref{D_q_cQ}-\eqref{eqn_Ramid}, we get
\begin{equation} \label{eq:D4_F0}
D^4 F_0^{K_{\PP_2}} = \frac{81 SX}{I_{11}^4}=\frac{27}{2} C^{-2}(B-A^2)\,.
\end{equation}

\subsection{Genus $1$ invariants of $(\PP^2,E)$}

\begin{thm}[=Theorem \ref{f1formula_intro}]\label{f1formula}
We have 
\begin{align*}
    F_1^{\PE}
&=-\frac{1}{24} \log (-\cQ) 
+\frac{1}{2}\sum_{n \geq 1} \log(1-\cQ^n)
-\frac{1}{2} \sum_{n \geq 1} \log(1-\cQ^{3n})\\
&= -\frac{1}{24} \log q + \frac{1}{24}
\log(1+27q)\,.
\end{align*}
\end{thm}

\begin{proof}
According to Theorem \ref{thm_locrel_SE}, we have 
\[ F_1^{K_{\PP^2}}=-F_1^{\PP^2/E}+F_{1,\emptyset}^E\,.\]
By formulae (A.3) and (A.15) of
\cite{Hu}, we have 
\begin{align*}
   F_1^{K_{\PP^2}}
&=-\frac{1}{12} \log q - \frac{1}{2}
\log I_{11}-\frac{1}{12} \log(1+27q)\\
&=-\frac{1}{12}\log (-\cQ) -\frac{1}{2}
\sum_{n \geq 1} \log(1-\cQ^{3n})
-\frac{1}{2} \sum_{n\geq 1} \log(1-\cQ^n)\,.
\end{align*}
On the other hand, by \eqref{eq:F_E}
we have 
\[ F_{1,\emptyset}^E=-\frac{1}{24}(-\tilde{\cQ})+\bar{F}_{1,\emptyset}^E\]
and it is well-known \cite{Dij} that 
\[ \bar{F}_{1,\emptyset}^E=
-\sum_{n \geq 1} \log(1-\tilde{\cQ}^n)\,.\]
As $\tilde{\cQ}=\cQ^3$, we get 
\[ F_{1,\emptyset}^E=-\frac{1}{8} \log (-\cQ) 
-\sum_{n \geq 1} \log(1-\cQ^{3n})\,,\]
and so 
\begin{align*}
F_1^{\PP^2/E}=&F_{1,\emptyset}^E- F_1^{K_{\PP^2}}\\
=&-\frac{1}{8} \log (-\cQ) 
-\sum_{n \geq 1} \log(1-\cQ^{3n})\\
&+\frac{1}{12}\log (-\cQ) +\frac{1}{2}
\sum_{n \geq 1} \log(1-\cQ^{3n})
+\frac{1}{2} \sum_{n\geq 1} \log(1-\cQ^n)\\
=&-\frac{1}{24} \log (-\cQ) 
+\frac{1}{2}\sum_{n \geq 1} \log(1-\cQ^n)
-\frac{1}{2} \sum_{n \geq 1} \log(1-\cQ^{3n})\,,
\end{align*}
which equals 
\[ -\frac{1}{24} \log q + \frac{1}{24}
\log(1+27q)\]
by formula (A.14) of \cite{Hu}.
\end{proof}

\begin{rmk}
Expanding the right-hand side of Theorem
\ref{f1formula}, we get
\[F_1^{\PE}=-\frac{1}{24}\log Q+\frac{7Q}{8}-\frac{129Q^2}{16}+\frac{589Q^3}{6}-\frac{43009Q^4}{32}+
\frac{392691Q^5}{20}+\dots\]
\end{rmk}

\begin{lem}
\begin{equation} \label{eq:D_F1}
DF_1^{\PE}=-\frac{1}{8} \frac{X}{I_{11}}=-\frac{1}{8} \frac{A^2}{C}\,.
\end{equation}
\end{lem}

\begin{proof}
Using Theorem \ref{f1formula}, we obtain
\begin{align*}
 Q \frac{dF_1^{\PE}}{dQ}&=-\frac{1}{24}\frac{1}{I_{11}}
q\frac{d}{dq}(\log q - \log (1+27q))
=
-\frac{1}{24}\frac{1}{I_{11}}\left( 1-\frac{27q}{1+27q} \right)\\
&=-\frac{1}{24}\frac{(1+27q)^{-1}}{I_{11}}
=-\frac{1}{24}\frac{X}{I_{11}}\,.
\end{align*}
We get the expression in terms of quasimodular forms using \eqref{eq:SIX_ABC}.
\end{proof}

\begin{prop} \label{prop_Dn_F1}
For every $n \geq 1$, we have 
\[ D^{n} F_1^{\PE} \in C^{-n} 
\Q[A,B,C]_{2n} \,.\]
\end{prop}

\begin{proof}
The case $n=1$ is clear by \eqref{eq:D_F1}. The general case follows by induction on $n$
from Lemma \ref{lem_D_weight}.
\end{proof}

For example, using \eqref{D_Q_q}-\eqref{eq:SX_derivative} or \eqref{eqn_Ramid}-\eqref{D_q_cQ}, we get
\begin{equation} \label{eq:D2_F1}
D^2 F_1^{\PE} =
\frac{3X}{8I_{11}^2}
\left(S-\frac{2}{3}(X-1)\right)
=\frac{A}{16C^2}(-5A^3+AB+4C)\,.
\end{equation}

\begin{rmk}
We have $\frac{\partial}{\partial S}
(DF_1^{\PE})=0$, as predicted by the holomorphic anomaly equation 
\eqref{eq_HAEPE}.
\end{rmk}

\subsection{Proof of finite generation and quasimodularity}
We prove the finite generation statements of Theorems \ref{fgproperty} and \ref{fgproperty1}.
By Proposition \ref{prop_spaces_equality}, it is enough to prove the finite generation
part of Theorem
\ref{fgproperty}.

The finite generation property for local $\PP^2$ is known by \cite{LhoP, CI}: we have 
\begin{equation}\label{eq:fg_local_P2}
F_g^{K_{\PP^2}}\in C^{-(2g-2)} \cdot  \Q[A,B,C]_{6g-6}
\end{equation}
for every $g \geq 2$.
More precisely, 
the result proved in \cite{LhoP} is slightly weaker, using the generator $L=X^{1/3}$ instead of $X$, not getting the optimal degree bound on $X$, and not mentioning the ``orbifold regularity". However, using the $R$-matrix techniques used in \cite{LhoP2} and its appendix, it is possible to prove that $F_g^{K_{\PP^2}} \in [ X^{-(g-1)} \cdot \bR_{\leq 3g-3} ]^{\reg}$ for every $g \geq 2$. Such refinement of the $R$-matrix is described in \cite{GJR18} for the proof of a ``graded finite generation" for the quintic $3$-fold. Once we know that $F_g^{K_{\PP^2}} \in [ X^{-(g-1)} \cdot \bR_{\leq 3g-3} ]^{\reg}$, we get that $F_g^{K_{\PP^2}} \in C^{-(2g-2)} \cdot  \Q[A,B,C]_{6g-6}$ by Proposition \ref{prop_spaces_equality}.
Alternatively, one can use \cite{CI} which proves directly the result in terms of quasimodular forms.

We show by induction on $g$ that, for every 
$g \geq 2$, we have 
\[
 F_g^{\PE} \in C^{-(2g-2)} \cdot \mathbb Q[A,B,C]_{6g-6}\,.
\]

Let $g \geq 2$. According to Theorem \ref{thm_locrel_SE}, we have
\[(-1)^{g-1} F_g^{\PE}= -F_g^{K_{\PP^2}}
+\]
\[    \sum_{n\geq 0} 
\!\!\!\!\!  \sum_{\substack{g=h+g_1+\dots+g_n,
\\ \ba=(a_1,\dots,a_n)\in \Z_{\geq 0}^n \\ 
(a_j,g_j) \neq (0,0),\,
\sum_{j=1}^n a_j=2h-2} }  \!\!\! \!\! 
\frac{(-1)^{h-1} F_{h, \ba}^E}{|\Aut (\ba, \bg)|}
\prod_{j=1}^n   (-1)^{g_j-1} 
 D^{a_j+2} F_{g_j}^{\PE}
\,,
\]
where the variable $\tilde{\cQ}$ in the definition \eqref{eq:F_E} of 
$F_{h,\ba}^E$ is expressed in terms of the variable $Q$ in the definitions 
\eqref{eq:F_KP2} and \eqref{eq:F_P2}
of $F_g^{K_{\PP^2}}$ and $F_g^{\PE}$ by 
$\eqref{eq:cQ2}$:
\[ \tilde{\cQ}=\exp \left( -D^2 F_0^{K_{\PP^2}} \right)\,.\] 

By \eqref{eq:fg_local_P2}, we know that 
$F_g^{K_{\PP^2}} \in C^{-(2g-2)} \cdot \mathbb Q[A,B,C]_{6g-6}$.
Therefore, it remains to show that each summand \begin{equation}
\frac{(-1)^{h-1}F_{h, \ba}^E}{|\Aut (\ba, \bg)|}\prod_{j=1}^n (-1)^{g_j-1}  D^{a_j+2} F_{g_j}^{\PE}
\end{equation}
belongs to $C^{-(2g-2)} \cdot \mathbb Q[A,B,C]_{6g-6}$.

Terms with $n=0$ only arises for 
$h=1$ and so $g=1$. Thus, for $g \geq 2$, only the terms with $n\geq 1$ contribute and by Theorem \ref{thm_quasimod_eliptic_curve} the 
series $F_{h,\ba}^E$ are quasimodular as functions of $\tilde{\tau}$, where $\tilde{\cQ}=e^{2i\pi \tilde{\tau}}$. More precisely, we have 
\[ F_{h,\ba}^E \in \Q[E_2(\tilde{\tau}),E_4(\tilde{\tau}),
E_6(\tilde{\tau})]_{\sum_{j=1}^n(a_j+2)} \,.\]

Our aim is to show quasimodularity as functions of $\tau$, that is given by \eqref{eq:cQ_formula} as
\[ \cQ=e^{2i\pi \tau}=-\exp \left( \frac{I_{12}}{I_{11}} \right) \,.\]
By \eqref{eq:D2_F0}, we have 
$D^2 F_0^{K_{\PP^2}} = -3 \frac{I_{12}}{I_{11}}$, so $\tilde{\tau}=3 \tau$ and
$\tilde{\cQ}=\cQ^3$.  Using 
Proposition \ref{prop_from_sl2_to_gamma}, we deduce that 
\[ F_{h,\ba}^E \in \Q[A,B,C]_{\sum_{j=1}^n
(a_j+2)}\,.\]
By induction on the genus, we know that for every $g_j \geq 2$
\[ F_{g_j}^{\PE} \in C^{-(2g_j-2)} \cdot \mathbb Q[A,B,C]_{6g_j-6}\,,\] and so
\[ D^{a_j+2} F_{g_j}^{\PE} \in C^{-(2g_j-2+a_j+2)} \cdot \mathbb Q[A,B,C]_{6g_j-6+2a_j+4}\] using
Lemma \ref{lem_D_weight}. By Propositions 
\ref{prop_Dn_F0} and \ref{prop_Dn_F1}, this result also holds for $g_j=0$ (in this case, $a_j \geq 1$ so $a_j+2 \geq 3)$ and 
$g_j=1$.

Therefore,
\begin{equation*}
\frac{(-1)^{h-1}F_{h, \ba}^E}{|\Aut (\ba, \bg)|}\prod_{j=1}^n (-1)^{g_j-1}  D^{a_j+2} F_{g_j}^{\PE}
\in C^{-(2g-2)} \cdot \mathbb Q[A,B,C]_{6g-6}
\end{equation*} 
follows from 
\[h+\sum_{j=1}^n g_j=g,\quad \sum_{j=1}^n a_j=2h-2.\]


\subsection{Genus $2$ invariants of $(\PP^2,E)$}
\label{sec:genus_2}
We prove Theorem \ref{f2formula}.

By Theorem \ref{thm_locrel_SE}, we have
\begin{align*}
F_2^{K_{\PP^2}}  =  & \ 
F_2^{\PE} + D^2 F_1^{\PE} \cdot  F_{1,(0)}^E
-\frac{1}{2} \big( D^3 F_0^{\PE}\big)^2  \cdot   F_{2,(1,1)}^E  
+ D^4 F_0^{\PE} \cdot   F_{2,(2)}^E \,.
\end{align*}
By \cite{LhoP}, we have 
\[F_2^{K_{\PP^2}}
=\frac{5}{8}\frac{S^3}{X}
+\frac{1}{8}S^2
+\frac{1}{96}SX
+\frac{X^2}{4320}
+\frac{X}{4320}
-\frac{1}{2160}\,.\]
Using \eqref{eq:D2_F1}-\eqref{eq_F10}-\eqref{modularformtogenerator}, we get
\[ D^2 F_1^{\PE} \cdot F_{1,(0)}^E=-\frac{S^2}{16}+\frac{5SX}{192}
-\frac{S}{24}+\frac{X^2}{96}-\frac{X}{96}\,.\]
Using \eqref{eq:D3_F0}-\eqref{eq_F211}-\eqref{modularformtogenerator}, we obtain
\[ -\frac{1}{2} (D^3 F_0^{\PE})^2 \cdot F_{2,(1,1)}^E 
=-\frac{S^3}{2X}-\frac{3S^2}{8}-\frac{11 SX}{120}+\frac{S}{60}
-\frac{X^2}{135}+\frac{7X}{1080}+\frac{1}{1080}\,.\]
Using \eqref{eq:D4_F0}-\eqref{eq_F22}-\eqref{modularformtogenerator}, we have
\[D^4F_0^{\PE} \cdot F_{2,(2)}^E
=\frac{9S^3}{8X}+\frac{9S^2}{16}+\frac{47SX}{640}+\frac{S}{40}\,.\]
Therefore, we find that in $F_2^{\PE}$ the coefficient of $\frac{S^3}{X}$
is 
\[ \frac{5}{8}-\left( -\frac{1}{2}+\frac{9}{8}\right)=0\,,\]
the coefficient of $S^2$ is 
\[ \frac{1}{8}-\left( -\frac{1}{16}-\frac{3}{8}+\frac{9}{16}\right)=0\,,\]
the coefficient of $SX$
is 
\[ \frac{1}{96}-\left(\frac{5}{192}-\frac{11}{120}+\frac{47}{640} \right)=\frac{1}{384}\,,\]
the coefficient of $X^2$ is 
\[\frac{1}{4320}-\left(\frac{1}{96}-\frac{1}{135} \right)=-\frac{1}{360}\,,\]
the coefficient of $X$ is 
\[\frac{1}{4320}-\left(-\frac{1}{96}+\frac{7}{1080} \right)=\frac{1}{240}\,,\]
and the constant term is 
\[ -\frac{1}{2160}-\frac{1}{1080}=-\frac{1}{720}\,.\]
This concludes the proof of Theorem 
\ref{f2formula}.

\begin{rmk}
Expanding the right-hand side of Theorem 
\ref{f2formula}, we get
\[ F_2^{\PE}=\frac{29Q}{640}-\frac{207Q^2}{64}+\frac{18447Q^3}{160}-\frac{526859Q^4}{160}+\frac{5385429Q^5}{64}+\dots\]
Using \eqref{eq:SIX_ABC}, we can rewrite Theorem \ref{f2formula} as
\begin{equation} 
F_2^{\PP^2/E}=\frac{1}{11520 C^2}
(-37A^6+5A^4 B+48 A^3 C-16C^2)\,.
\end{equation}
Taking the $S$-derivative of Theorem \ref{f2formula}, we obtain
\[ \frac{\partial}{\partial S}F_2^{\PE}
=\frac{X}{384} \,,\]
and so, using \eqref{eq:D_F1}, 
\[ \frac{3X}{I_{11}^2} 
\frac{\partial}{\partial S}F_2^{\PE}
=\frac{1}{2}(DF_1^{\PE})^2 \,,\]
as predicted by the holomorphic anomaly equation \eqref{eq_HAEPE}.
\end{rmk}

\section{Holomorphic anomaly equation for $(\PP^2,E)$}

\label{sec:HAE}

In this section, we prove Theorem \ref{HAEforrelative}, that is the holomorphic anomaly equation for the series 
$F_{g,n}^{\PE}$. 

We will use the following definitions.
\begin{defn}
A partition $\ba$ of length $n$ is an ordered set $(a_1,\cdots,a_n)$ such that
$$
a_1\geq a_2\geq \cdots \geq a_n\geq 0
$$
\end{defn}
Note we allow the entries $a_i$ to be zero, which is different from the ordinary definition of a partition.
\begin{defn} \label{defncuppartition}
For any two partitions $\ba$ and $\mathbf b$ of length $n_1$ and $n_2$ respectively. We define
$$
\ba \cup \mathbf b
$$
to be the partition of length $n_1+n_2$ with entries which are exactly the entries of $\ba$ and $\mathbf b$ with decreasing ordering.
\end{defn}

\subsection{Holomorphic anomaly equation for the elliptic curve}
We first review some known results for the elliptic curve. Recall that we denote 
\[\left< \tau_1 \psi_1^{k_1},\cdots,\tau_n \psi_n^{k_n} \right>^{E}_{g,n}
\coloneqq \sum_{d\geq 0}\tilde{\cQ}^{d}\int_{[\overline{M}_{g,n}(E,d)]^{\vir}}\prod_{i=1}^n\psi_i^{k_i}\ev_i^*\tau_i\]
the generating series of Gromov-Witten invariants of the elliptic curve $E$, with $\tau_i \in H^{\bullet}(E)$. Recall also that we denote by $\omega \in H^2(E)$ the (Poincar\'e dual) class of a point
and 
\begin{equation*}
F^E_{g,\ba}
\coloneqq 
-\frac{\delta_{g,1} \delta_{n,0} }{24}\log 
\left((-1)^{E \cdot E}\tilde{\cQ}\right)
+\sum_{d\geq 0}{\tilde{\cQ}^d}\left< \omega \psi_1^{a_1},\dots,\omega \psi_n^{a_n}
\right>_{g,n,d}^E\,,
\end{equation*}
where the stationary invariants 
$\left< \omega \psi_1^{a_1},\dots,\omega \psi_n^{a_n}
\right>_{g,n,d}^E$ 
have been introduced in \eqref{eq:gw_E}.

Using the
polynomiality of the double ramification cycle in the parts of the ramification
profiles, Oberdieck and Pixton \cite{ObPi}
proved the following holomorphic anomaly equation for the Gromov-Witten theory of the elliptic curve: for 
$2g-2+n>0$, we have 
\begin{multline*}
  -24 \,  {\textstyle  \frac{ \partial}{\partial{E_2}}} \left< \omega \psi_1^{a_1},\cdots, \omega \psi_n^{a_n} \right>^E_{g,n} = \\
  \sum_{g_1+g_2 = g, \atop \ba' \cup \ba'' = \ba}    \big\langle \omega \psi_1^{a'_1},\cdots, \omega \psi_s^{a'_s} ,1 \big\rangle^E_{g_1,s+1} \big\langle  \omega \psi_1^{a''_1},\cdots, \omega \psi_{n-s}^{a''_{n-s}},1 \big\rangle^E_{g_2,n-s+1} \qquad \quad \\
+  \big\langle \omega \psi_1^{a_1},\cdots, \omega \psi_n^{a_n} ,1,1 \big\rangle^E_{g-1,n+2} 
-2 \sum_{j=1}^n \big\langle \omega \psi_1^{a_1},\cdots, \psi_j^{a_j+1}, \cdots,\omega \psi_n^{a_n}\big\rangle^E_{g,n} .
\end{multline*}

By the Virasoro constraints proved by Okounkov and Pandharipande \cite{OP3}, we have
\begin{multline*}
 \sum_{j=1}^n \big\langle \omega \psi_1^{a_1},\cdots, \psi_j^{a_j+1}, \cdots,\omega \psi_n^{a_n}\big\rangle^E_{g,n} \\
 = \sum_{1\leq i \neq j\leq n} \binom{a_i+a_j+1}{a_i} \big\langle \omega \psi_1^{a_1},\cdots, \widehat{\omega\psi_i^{a_i}},\cdots, \widehat{\omega\psi_j^{a_j}}, \cdots,\omega \psi_{n}^{a_n}, \omega \psi_{i}^{a_i+a_j}\big\rangle^E_{g,n-1} \,,
 \end{multline*}
where as usual the caret
denotes omission.

Together with the string equation, we obtain
the following form of the holomorphic anomaly equation for the series $F_{g,\ba}^E$.
Let $\ba$ be a partition of $2h-2$, i.e. $\sum_{i=1}^n a_i=2h-2$, then, for $2h-2+n>0$, we have
\begin{align}
-24 \,  {\textstyle  \frac{ \partial}{\partial{E_2}}}  F_{h,\ba}^E =  & \  \sum_{1\leq i,j\leq n} F^E_{h-1,\ba-\vec e_i -\vec e_j} + \sum_{h_1+h_2=h \atop  \ba' \cup \ba'' = \ba }\sum_{1\leq i\leq l(\ba')\atop 1\leq j\leq l(\ba'')}F^E_{h_1,\ba'-\vec e_i}F^E_{h_2,\ba''-\vec e_j} \nonumber
\\& \ \  \    
-
2 \sum_{1\leq i\neq j \leq n}  \binom{a_i+a_j+1}{a_i} F^E_{h,\ \mathcal G_{ij}(\ba)}  \label{HAEforE}
\end{align}
Here for a partition $\ba =(a_1,\cdots,a_n)$, we define the gluing operation by
$$
\mathcal G_{ij}(\ba) = (a_1,\cdots \hat a_i,\cdots \hat a_j,\cdots, a_n) \cup ( a_{i}+a_j),
$$
and for the vectors $\vec e_i$ ($i=1,\cdots,n$):
$$
\vec e_i = (0,\cdots, 1,\cdots, 0)  \  \text{  with  $1$ lies in the $i$-th component},
$$
we define (c.f. Definition \ref{defncuppartition})
$$
\ba -\vec e_i := (a_1,\cdots,\widehat a_i,\cdots,a_n) \cup (a_i-1),\quad
$$
$$ 
\quad
\ba -\vec e_i -\vec e_j := \begin{cases} \qquad 
(a_1,\cdots,\widehat a_i,\cdots,a_n) \cup (a_i-2), & \text{ if $i=j$, }\\
(a_1,\cdots,\widehat a_i,\cdots,\widehat a_j,\cdots,a_n) \cup (a_i-1,a_j-1),& \text{ otherwise. }
\end{cases} 
$$

\begin{ex} 
Using \eqref{eq_F10}--\eqref{eq_F22}, one can check directly that 
$$
-12  \,  {\textstyle  \frac{ \partial}{\partial{E_2}}} F_{2,(1^2)}^E =     F^E _{1,(0^2)} + \big(F^E _{1,(0)}\big)^2
- 6 F^E_{2,(2)}\,,
$$
$$ -24 {\textstyle  \frac{ \partial}{\partial{E_2}}} F_{2,(2)}^E=F_{1,(0)}^E \,.$$
\end{ex}

\subsection{Holomorphic anomaly equation for local $\PP^2$}
We denote by $F_{g,n}^{K_{\PP^2}}$ the generating function for the local $\PP^2$ theory with $n$ insertions of the hyperplane classes. By the divisor equation, 
$$
F_{g,n}^{K_{\PP^2}} = \Big(Q\frac{d}{dQ}\Big)^n  F_{g}^{K_{\PP^2}}.
$$
We have the following {holomorphic anomaly equation} which was proved using various techniques in \cite{LhoP}, 
\cite{CI} and \cite{EMO,EO,FLZ,FRZZ}:
\begin{equation} \label{HAEforKP2}
  \frac{X}{3\, {I_{11}}^2} \cdot  \frac{\partial}{\partial S} F_{g,n}^{K_{\PP^2}} = \frac{1}{2} \sum_{\substack{g_1+g_2 = g\\  n_1+n_2=n \\ 2g_i-2+n_i\geq 0}} \binom{n}{n_1}\ F_{g_1,n_1+1}^{K_{\PP^2}}\cdot   F_{g_2,n_2+1}^{K_{\PP^2}} + \frac{1}{2}F_{g-1,n+2}^{K_{\PP^2}}.
\end{equation}

\subsection{Proof of the holomorphic anomaly equation for $(\PP^2,E)$}

We prove the holomorphic anomaly equation
\eqref{eq_HAEPE} for $F_{g,n}^{\PE}$
(Theorem \ref{HAEforrelative})  by induction on the genus $g$. We have 
$F_{0,n}^{\PE}=F_{0,n}^{K_{\PP^2}}$ and so 
\eqref{eq_HAEPE} holds for $g=0$ by 
\eqref{HAEforKP2}.

Let $g$ and $n$ such that $2g-2+n>0$.
By taking derivatives of both sides of 
Theorem
\ref{thm_locrel_SE}, we obtain
\begin{align}  \label{PErecursioninsertion}
F_{g,n}^{K_{\PP^2}}  
= & \  (-1)^g F_{g,n}^{\PE} +\!\! \!\!\!\! \sum_{0<h\leq g,\atop \gab  \in \mathcal P_{g,n}(h)} \!\!\!\!\!\!  (-1)^{h-1} F_{h,\ba}^E \cdot \Cont^{\PE}_{\gab}\quad  \\
 = & \ (-1)^g F_{g,n}^{\PE}+ 9\, F_{1,(0)}^E \cdot  (-1)^{g-2}F^{\PE}_{g-1,n+2}+\!\! \!\!\!\! \sum_{0<h\leq g,\atop \gab  \in \mathcal P^+_{g,n}(h)} \!\!\!\!\!\! 
(-1)^{h-1} F_{h,\ba}^E \cdot \Cont^{\PE}_{\gab}\nonumber
\end{align}
where we define
$$
\mathcal P_{g,n}(h):= \left\{ \gab= \{(g_i,a_i,B_i)\}_{i=1}^l  \Big|   \substack {h+\sum_i g_i=  g,\  \sum_i a_i = 2h-2, \\   \sum_i |B_i| =n, \  2g_i-2+2a_i+2 |B_i|\geq 0,\\ g_i,\ a_i \in\Z_{\geq 0},\  \coprod_i B_i=\{1,2,\cdots,n\}}  \right\}, $$
$$
\mathcal P^+_{g,n}(h):=\mathcal P_{g,n}(h) \Big\backslash  \Big\{  \big(g-1,0,\{1,2,\cdots,n\}\big) \Big\}, 
$$
$$
\Cont^{\PE} _{\gab}:=\frac{1}{|\Aut \gab|}\prod_i (-1)^{g_i-1}  3^{a_i+2} F_{g_i,a_i+b_i+2}^{\PE}.
$$
Note that $\bb$ corresponds to the assignment of insertions, and we allow $B_i$ to be empty in the definition of 
$\mathcal P_{g,n}(h)$. Given an element $\{(g_i,a_i,B_i)\}_{i=1}^l\in \mathcal P_{g,n}(h)$, we set $b_i=|B_i|$ to be the number of elements in $B_i$, and set $\bg=\{g_1,\cdots,g_l\}$, $\ba=\{a_1,\cdots,a_l\}$, $\bb=\{B_1,\cdots,B_l\}$.
We use $l(\bg),l(\ba),l(\bb)$ to denote the lengths of $\bg,\ba,\bb$ respectively. Here all the lengths are equal to $l$.
Finally, $\Aut \gab$ is the symmetry group consisting of permutation symmetries of $\gab$.

\medskip
We denote 
$\partial_S:=\frac{X}{3\, {I_{11}}^2} \cdot  \frac{\partial}{\partial S}$.
By \eqref{modularformtogenerator} we have
$$\textstyle 
 \partial_S  \ =  \ 
\frac{1}{18} \cdot   24  \frac{\partial}{\partial_{E_2 }}  \,.
$$
By applying the operator $\partial_S$ on \eqref{PErecursioninsertion}
and using
\eqref{eq_F10}, we obtain
\begin{align}
  {\partial_S} F_{g,n}^{K_{\PP^2}} 
 = & \ (-1)^g {\partial_S} F_{g,n}^{\PE}+  \frac{(-1)^{g-1}}{2} F^{\PE}_{g-1,n+2}   +
 \!\! \!\!\!\!\!\! \sum_{0<h\leq g,\atop \gab  \in \mathcal P_{g,n}(h)} \!\!\!\!\!\! \!\!\!\! 
(-1)^{h-1} F_{h,\ba}^E \cdot {\partial_S}\Cont^{\PE}_{\gab}\quad \nonumber \\
& \qquad  +   
 \!\! \!\!\!\! \sum_{0<h\leq g,\atop \gab  \in \mathcal P^+_{g,n}(h)} \!\!\!\!\!\!\!\!\! 
(-1)^{h-1} {\partial_S} F_{h,\ba}^E \cdot \Cont^{\PE}_{\gab}\\
 = & \ (-1)^g {\partial_S} F_{g,n}^{\PE}+  \frac{1}{2}  (-1)^{g-1}F^{\PE}_{g-1,n+2}   + \cC_1 +\cC_2+\cC_3+\cC_4
 \nonumber
\end{align}
where $\cC_1$ is the contribution of  the last term in the first line, and $\cC_2, \cC_3, \cC_4$ are the three contributions obtained by applying the holomorphic anomaly equation \eqref{HAEforE} to the term in the second line:
\begin{align*} 
\cC_1 :=  & \
 \!\! \!\!\!\! \sum_{0<h\leq g,\atop \gab  \in \mathcal P_{g,n}(h)} \!\!\!\!\!\!  (-1)^{h-1} F_{h,\ba}^E \cdot {\partial_S}\Cont^{\PE}_{\gab},\quad\\
\cC_2 :=  &\frac{1}{18} \ 
 \!\! \!\!\!\! \sum_{0<h\leq g,\atop \gab  \in \mathcal P^+_{g,n}(h)} \!\!\!\!\!\! 
 \sum_{1\leq i,j\leq l(\ba)} (-1)^h  F^E_{h-1,\ba-\vec e_i -\vec e_j}\cdot \Cont^{\PE}_{\gab},\\
\cC_3 :=  &\frac{1}{18} \ 
 \!\! \!\!\!\! \sum_{0<h\leq g,\atop \gab  \in \mathcal P^+_{g,n}(h)} \! \sum_{h_1+h_2=h \atop  \ba' \sqcup \ba'' = \ba } \sum_{1\leq i\leq l(\ba')\atop 1\leq j\leq l(\ba'')}
(-1)^{h_1-1}
F^E_{h_1,\ba'-\vec e_i}
(-1)^{h_2-1}
F^E_{h_2,\ba''-\vec e_j}
 \cdot \Cont^{\PE}_{\gab},\\
\cC_4 :=  &\frac{1}{18} \ 
 \!\! \!\!\!\! \sum_{0<h\leq g,\atop \gab  \in \mathcal P^+_{g,n}(h)} \!\!\!\!\!\! 
 { \sum_{1\leq i\neq j \leq l(\ba)}  2\, \binom{a_i+a_j+1}{a_i}(-1)^{h-1} F^E_{h,\ \mathcal G_{ij}(\ba)}}  
  \cdot \Cont^{\PE}_{\gab}.
\end{align*}

On the other hand, we can first apply the 
holomorphic anomaly equation \eqref{HAEforKP2} for local $\PP^2$. We then apply equation \eqref{PErecursioninsertion} to the right-hand side of \eqref{HAEforKP2}. We have
\begin{align*}
  {\partial_S} F_{g,n}^{K_{\PP^2}} 
 = & \ \frac{1}{2}
 \Big( (-1)^{g-1} F_{g-1,n+2}^{\PE} +\!\! \!\!\!\! \sum_{0<h\leq g-1,\atop \gab  \in \mathcal P_{g-1,n+2}(h)} \!\!\!\!\!\!   
(-1)^{h-1} F_{h,\ba}^E \cdot \Cont^{\PE}_{\gab} \Big)\\
   & \  + \frac{1}{2} \sum_{\substack{g_1+g_2 = g\\  n_1+n_2=n \\ 2g_i-2+n_i\geq 0}} \binom{n}{n_1}  \prod_{i=1}^2  \Big(  (-1)^{g_i} F_{g_i,n_i+1}^{\PE} +\!\! \!\!\!\! \sum_{0<h\leq g_i,\atop \gab  \in \mathcal P_{g_i,n_i+1}(h)} \!\!\!\!\!\!\!\!\!   (-1)^{h-1} F_{h,\ba}^E \cdot \Cont^{\PE}_{\gab} \Big) \\
    = & \ \frac{(-1)^{g-1}}{2}    F_{g-1,n+2}^{\PE} + \cC_2'   
   +  \frac{(-1)^{g}}{2} \!\!\!\!\!\! \sum_{\substack{g_1+g_2 = g\\  n_1+n_2=n \\ 2g_i-2+n_i\geq 0}} \!\!\! \binom{n}{n_1} (-1)^{g_1-1}F_{g_1,n_1+1}^{\PE}
(-1)^{g_2-1}F_{g_2,n_2+1}^{\PE} \\
 & \ +\cC_1'+\cC_3'  ,
\end{align*}
where $\cC_2'$ is the contribution of the last term in the first line, and $\cC_1', \cC_3'$ are the two types of contributions of the terms in the second line:
\begin{align*}
\cC_2'
 = &\frac{1}{2} \  \!\! \!\!\!\! \sum_{0<h\leq g-1,\atop \gab  \in \mathcal P_{g-1,n+2}(h)} \!\!\!\!\!\!   (-1)^{h-1}F_{h,\ba}^E \cdot \Cont^{\PE}_{\gab}, \\
\cC_1'
 = & \   \sum_{\substack{g_1+g_2 = g\\  n_1+n_2=n \\ 2g_i-2+n_i\geq 0}} \binom{n}{n_1}     (-1)^{g_1} F_{g_1,n_1+1}^{\PE} \
\!\! \!\!\!\! \sum_{0<h\leq g_2,\atop \gab  \in \mathcal P_{g_2,n_2+1}(h)} \!\!\!\!\!\! 
(-1)^{h-1} F_{h,\ba}^E \cdot \Cont^{\PE}_{\gab}, \\
\cC_3'
 = & \   \frac{1}{2} \sum_{\substack{g_1+g_2 = g\\  n_1+n_2=n \\ 2g_i-2+n_i\geq 0}} \binom{n}{n_1}  \!\! \!\!\!\!\!\!\! \sum_{0<h\leq g_1,\atop \gab  \in \mathcal P_{g_1,n_1+1}(h)} 
\!\!\!\!\!\!\!\!\!   
(-1)^{h-1} F_{h,\ba}^E \cdot \Cont^{\PE}_{\gab} \!\! \!\!\!\!\!\!\! \sum_{0<h\leq g_2,\atop \gab  \in \mathcal P_{g_2,n_2+1}(h)} \!\!\!\!\!\!\!\!\!   
(-1)^{h-1} F_{h,\ba}^E \cdot \Cont^{\PE}_{\gab}.
\end{align*}
Suppose that the holomorphic anomaly equation for $(\PP^2,E)$ (Theorem \ref{HAEforrelative}) holds  for $g'<g$. Namely
\begin{equation}  \label{HAEforrelativeg}
  \partial_S    F_{g',n}^{\PE} = \frac{1}{2} \sum_{\substack{g_1+g_2 = g'\\  n_1+n_2=n \\ 2g_i-2+n_i\geq 0}} \binom{n}{n_1}\ F_{g_1,n_1+1}^{\PE}\cdot   F_{g_2,n_2+1}^{\PE} \quad \text{ for $g'<g$ }.
\end{equation}
We prove the holomorphic anomaly equation for genus $g$ case by showing that
$$
\cC_1+\cC_4 = \cC_1',\qquad
\cC_2 = \cC_2',\qquad
\cC_3 = \cC_3'.\quad
$$
These three identities follow from the following three lemmas.

\begin{lem}
Suppose that the holomorphic anomaly equation for $(\PP^2,E)$ holds for all $g'<g$. Then we have
$$
\cC_1+\cC_4 = \cC_1'\,.$$
\end{lem}
\begin{proof}
By using the fact that 
\[\binom{a_i+a_j+1}{a_i}+\binom{a_i+a_j+1}{a_j}=\binom{a_i+a_j+2}{a_i+1},\]
we may write $\cC_4$ as 
\begin{equation}\label{pflm1-1}
\frac{1}{18}\sum_{0<h\leq g,\atop \gab  \in \mathcal P_{g,n}(h)} 
 { \sum_{1\leq i\neq j \leq l(\ba)} \binom{a_i+a_j+2}{a_i+1} (-1)^{h-1}
F^E_{h,\ \mathcal G_{ij}(\ba)}}  
  \cdot \Cont^{\PE}_{\gab}.
  \end{equation}
In the summation, we can replace $\mathcal P^+_{g,n}(h)$ by $\mathcal P_{g,n}(h)$ since $l(\ba)\geq 2$.

Let us fix a $\gab=\{(g_i,a_i,B_i)\}_{i=1}^l \in \mathcal P_{g,n}(h)$ and $1\leq s\neq t\leq l$. If we sum over those $i\neq j$ in \eqref{pflm1-1} such that
\[(g_i,a_i,B_i)=(g_s,a_s,B_s),\quad (g_j,a_j,B_j)=(g_t,a_t,B_t),\]
then the total contribution can be written as
\begin{equation}\label{pflm1-2}
\begin{split}
&(-1)^{h-1} 
F^E_{h,\ \mathcal G_{ij}(\ba)} \frac{1}{|\Aut(P_1)|}\Cont_{P_1}^{\PE}\cdot\\
&\left(-\frac{1}{2}(-1)^{g_s+g_t-1}3^{a_s+a_t+2}\binom{a_s+a_t+2}{a_s+1}F_{g_s,a_s+b_s+2}^{\PE}F_{g_t,a_t+b_t+2}^{\PE}\right)
\end{split}
\end{equation}
where $P_1=\{(g_k,a_k,B_k)\in \gab \mid k\neq s,t\}$ and
\begin{equation*}
\Cont_{P_1}^{\PE}=\prod_{k\neq s,t}(-1)^{g_k-1}3^{a_k+2}F_{a_k+b_k+2}^{\PE}.
\end{equation*}
Now let us vary $\gab\in P_{g,n}(h)$. We sum over those \eqref{pflm1-2} with
\[\gab=P_1\cup\{(g_s,a'_s,B'_s)\}\cup\{(g_t,a'_t,B'_t)\}\]
such that $a'_s+b'_s=a_s+b_s$ and $a'_t+b'_t=a_t+b_t$. Since $\sum a'_i=2h-2$ is fixed, we also deduce that $a'_s+a'_t=a_s+a_t$ and $b'_s+b'_t=b_s+b_t$. Then we get 
\begin{align*}
(-1)^{h-1} F^E_{h,\ \mathcal G_{ij}(\ba)} \frac{\Cont_{P_1}^{\PE}}{|\Aut(P_1)|}\cdot&\left(-\frac{1}{2}(-1)^{g_s+g_t-1}3^{a_s+a_t+2}F_{g_s,a_s+b_s+2}^{\PE}F_{g_t,a_t+b_t+2}^{\PE}\right)\\
&\sum_{\substack{a'_s+b'_s=a_s+b_s\\ 0\leq a'_s\leq a_s+a_t\\0\leq b'_s\leq b_s+b_t }}\binom{a_s+a_t+2}{a'_s+1}\binom{b_s+b_t}{b'_s}.
\end{align*}
Using the Vandermonde's identity
\begin{align*}
\sum_{\substack{a'_s+b'_s=a_s+b_s\\ 0\leq a'_s\leq a_s+a_t\\0\leq b'_s\leq b_s+b_t }}&\binom{a_s+a_t+2}{a'_s+1}\binom{b_s+b_t}{b'_s}=\\
&\binom{a_s+b_s+a_t+b_t+2}{a_s+b_s+1}-\binom{b_s+b_t}{a_s+b_s+1}-\binom{b_s+b_t}{a_t+b_t+1},
\end{align*}
the above equation can be further written as
\[T_1+T_2\]
where
\begin{align*}
T_1=& (-1)^{h-1} F^E_{h,\ \mathcal G_{ij}(\ba)} \frac{\Cont_{P_1}^{\PE}}{|\Aut(P_1)|}\\
&\left(-\frac{1}{2}(-1)^{g_s+g_t-1}3^{a_s+a_t+2}\binom{a_s+b_s+a_t+b_t+2}{a_s+b_s+1}F_{g_s,a_s+b_s+2}^{\PE}F_{g_t,a_t+b_t+2}^{\PE}\right),\\
T_2=&
(-1)^{h-1} F^E_{h,\ \mathcal G_{ij}(\ba)} \frac{\Cont_{P_1}^{\PE}}{|\Aut(P_1)|}\cdot\left(\binom{b_s+b_t}{a_s+b_s+1}+\binom{b_s+b_t}{a_t+b_t+1}\right)\\
&\frac{1}{2}(-1)^{g_s+g_t-1}3^{a_s+a_t+2}F_{g_s,a_s+b_s+2}^{\PE}F_{g_t,a_t+b_t+2}^{\PE}.
\end{align*}
Using the holomorphic anomaly equation \eqref{HAEforrelativeg} for $g'<g$, it is easy to see that the total contribution of those $T_1$ when we vary $\gab$ and $s,t$ is
\begin{align*}
-\cC_1+\sum_{0<h\leq g,\atop \gab  \in \mathcal P_{g,n}(h)}& 
(-1)^{h-1} F_{h,\ba}^E
\frac{1}{\Aut\gab}\sum_{1\leq i\leq l(\ba)}\prod_{j\neq i}(-1)^{g_j-1}3^{a_j+2}F_{g_j,a_j+b_j+2}^{\PE}\\
&\left(\sum_{g'_1+g'_2=g_i}(-1)^{g'_1}F_{g'_1,1}^{\PE}(-1)^{g'_2-1}3^{a_2+2}F_{g'_2,a_i+b_i+3}^{\PE}\right)
\end{align*}
It is easy to check that
\begin{align*}
\sum_{0<h\leq g,\atop \gab  \in \mathcal P_{g,n}(h)}\!\!\!\!\!\!\!& 
(-1)^{h-1} F_{h,\ba}^E \frac{1}{\Aut\gab}\sum_{1\leq i\leq l(\ba)}\prod_{j\neq i}(-1)^{g_j-1}3^{a_j+2}F_{g_j,a_j+b_j+2}^{\PE}\\
&\left(\sum_{g'_1+g'_2=g_i}(-1)^{g'_1}F_{g'_1,1}^{\PE}(-1)^{g'_2-1}3^{a_2+2}F_{g'_2,a_i+b_i+3}^{\PE}\right)
\end{align*}
equals 
\[ \sum_{g_1+g_2=g}    (-1)^{g_1} F_{g_1,1}^{\PE} \cdot \sum_{0<h\leq g_2,\atop \gab  \in \mathcal P_{g_2,n+1}(h)} 
\!\!\!\!\! (-1)^{h-1} F_{h,\ba}^E \cdot \Cont^{\PE}_{\gab}.\]
So it remains to check the total contribution of those $T_2$ to $\cC_1$ when we vary $\gab$ and $s,t$ is
\[\sum_{g_1+g_2=g \atop n_1+n_2=n,n_1>0} \binom{n}{n_1}     (-1)^{g_1} F_{g_1,n_1+1}^{\PE} \!\!\!\!\!\! \sum_{0<h\leq g_2,\atop \gab  \in \mathcal P_{g_2,n_2+1}(h)}
\!\!\!\!\!\!  
(-1)^{h-1} F_{h,\ba}^E \cdot \Cont^{\PE}_{\gab}\]
which is obvious.
\end{proof}



\begin{lem}\label{lem-tt}
$$
\cC_2 = \cC_2'
$$
\end{lem}
\begin{proof}
Given a term 
\[\frac{1}{18}(-1)^h F^E_{h-1,\ba-\vec e_i -\vec e_j}\cdot \Cont^{\PE}_{\gab}\]
in $\cC_2$, it can be rewritten as
\[\frac{|\Aut(\bg',\ba',\bb')|}{|\Aut\gab|}\frac{1}{2}
(-1)^{h'-1}
F^E_{h',\ba'}\cdot\Cont^{\PE}_{(\bg',\ba',\bb')}\]
where $h'=h-1$, $\bg'=\bg$, $\ba'=\ba-\vec e_i -\vec e_j$ and $\bb'$
is a partition of the set $\{1,2,\cdots,n+2\}$ which can be determined from $\bb$ by adding $n+1$ to the set $B_i$ and adding $n+2$ to the set $B_j$. Obviously, $(\bg',\ba',\bb')\in P_{g-1,n+2}(h')$. So 
\[\frac{1}{2}(-1)^{h'-1}
F^E_{h',\ba'}\cdot\Cont^{\PE}_{(\bg',\ba',\bb')}\]
becomes one summand in $\cC_2'$. Now Lemma \ref{lem-tt} follows from the fact that for a fixed $(\bg',\ba',\bb')\in P_{g-1,n+2}(h')$, there are exactly $\frac{|\Aut\gab|}{|\Aut(\bg',\ba',\bb')|}$ choices of $\gab,i,j$ which gives $(\bg',\ba',\bb')$ via the above procedure. Actually, we see that $\gab$ can be determined from $(\bg',\ba',\bb')$. The only flexibility comes from the choices of $i$ and $j$.

\end{proof}

\begin{lem}$$
\cC_3 = \cC_3'
$$
\end{lem}
\begin{proof}
The proof is similar to the proof of Lemma \ref{lem-tt}. We omit the details here.
\end{proof}

\subsection{$S$-degree bound on $F_g^{\PE}$}

We prove the $S$-degree bound of Theorem \ref{fgproperty} (equivalently the $B$-degree bound of Theorem \ref{fgproperty1}), that is, for every $g \geq 2$, 
\begin{equation}\label{eq:S_bound} 
\deg_S F_g^{\PE} \leq 2g-3\,.
\end{equation}
As we only have 
\begin{equation*} 
\deg_S F_g^{K_{\PP^2}} \leq 3g-3
\end{equation*} 
in general, the bound 
\eqref{eq:S_bound} is not an obvious consequence of Theorem \ref{thm_locrel_SE}
and  
requires a non-trivial cancellation of higher degree terms.
For example,
we have observed such cancellation in the genus $2$ computation of Section 
\ref{sec:genus_2} (vanishing of the terms in $S^3/X$ and $S^2$).
Rather than trying to prove directly this cancellation in general, we show that 
\eqref{eq:S_bound}
follows from the holomorphic anomaly equation 
\eqref{eq_HAEPE}.

We prove \eqref{eq:S_bound} by induction on $g$. Using 
\eqref{D_Q_q}, we rewrite the
holomorphic anomaly equation \eqref{eq_HAEPE} as
\begin{equation} \label{eq:HAE_rw}
\frac{X}{3} \frac{\partial}{\partial S}F_g^{\PE}
= \frac{1}{4}
\sum_{g_1+g_2 = g, \atop g_i >0 \text{ for } i =1,2} \ 
\left( q\frac{d F_{g_1}^{\PE}}{dq}
\right)
\cdot 
\left( q \frac{ dF_{g_2}^{\PE}}{dq} 
\right)\,.
\end{equation}
By induction, we have for $g_j \geq 2$
\[ \deg_S F_{g_j}^{\PE} 
\leq 2g_j-3\,,\] 
and so using \eqref{eq:SX_derivative},
\[ \deg_S \left( q \frac{d F_{g_j}^{\PE}}{dq} \right)
\leq 2g_j-2 \,.\]
By \eqref{eq:D_F1}, this bound also holds for $g_j=1$. Therefore, the $S$-degree of the right-hand side of \eqref{eq:HAE_rw} is  $\leq (2g_1-2)+(2g_2-2)=2g-4$ and so the $S$-degree of $F_g^{\PE}$ is $\leq 2g-3$.

\appendix
\section{Product formulae for the relative theory}\label{sec:appx}
In this section, we mainly want to prove a product formula relating the relative theory of possibly disconnected domain and the one with connected domain (see Lemma \ref{lem_split1}). It implies Lemma \ref{lem_split2}
and so completes the last step in the localization calculation of Section \ref{sec:loc}. Before proving Lemma \ref{lem_split1}, we need a product formula relating the rubber theory of possibly disconnected domain and the one with connected domain. 

Let $D$ be a smooth projective variety and $L$ a line bundle on $D$. We consider a rubber target over $D$ obtained as
a chain of $\mathbb P_D(L\oplus \cO)$. Let $\bM_{\Gamma}^{\bullet\sim}(D)$ be a moduli space of relative stable maps to the rubber target, where $\Gamma$ is a possibly disconnected rubber graph (see \cite[Definition 2.4]{FWY}). The rubber theory is relative to the two ends of the rubber target. One of the ends of the rubber target has normal bundle $L^\vee$, and we denote target psi-class associated to this end by $\Psi_\infty$ (see also \cite{GV}*{Section 2.5}). The target psi-class associated to the other end will be denoted by $\Psi_0$.

For a vertex $v\in V(\Gamma)$, recall that
we denote by $\gamma_v$ the graph consisting of a single vertex $v$ plus all the decorations on $v$, and by
$g(v)$, $n(v)$, $b(v)$ the genus, number of markings and the curve class of the vertex $v$.
We say that a vertex $v$ is 
\emph{unstable} if the curve class of $v$ is pushed forward to $0$ on $D$, $v$ has only two relative markings and no absolute markings. We say that a vertex $v$ is \emph{stable} if it is not unstable. Let $V^{\mathrm s}(\Gamma)$ be the set of stable vertices and $V^{\mathrm{us}}(\Gamma)$ the set of unstable vertices. In the following theorem, we need to consider the stabilization map from $\bM_{\Gamma}^{\bullet\sim}(D)$ to the moduli of stable maps of $D$. Note that stabilization does not make sense on components corresponding to unstable vertices.
Our convention for the stabilization map
\[
\tau \colon \bM_{\Gamma}^{\bullet\sim}(D) \rightarrow \prod_{v\in V^{\mathrm s}(\Gamma)}\bM_{g(v),n(v)}(D,b(v))\times D^{|V^{\mathrm{us}}(\Gamma)|}
\]
is that we firstly stabilize stable components and send unstable components to the corresponding points in $D$. If $v$ is an unstable vertex, the two relative markings have the same multiplicity, that we denote by $d_v$.

\begin{thm}\label{thm_appd}
We have the following identity.
\begin{align*}
&\tau_*\Big(\dfrac{1}{t-\Psi_\infty} \cap [\bM^{\bullet\sim}_{\Gamma}(D)]^{\vir}\Big) \\
=& \dfrac{1}{\prod_{v\in V^{\mathrm{us}}(\Gamma)} d_v}
\tau'_*\left(\prod\limits_{v\in V^{\mathrm s}(\Gamma)}p_v^*\left(\dfrac{1}{t-\Psi_\infty}\right) \cap \big[\prod\limits_{v\in V^{\mathrm s}(\Gamma)}\bM^{\sim}_{\gamma_v}(D)\times D^{|V^{\mathrm{us}}(\Gamma)|} \big]^{\vir}\right),
\end{align*}
where each $\Psi_{\infty}$ on the left-hand side and right-hand side are the target psi-classes on their corresponding moduli, 
$p_v$ is the projection of the product of rubber moduli to $\bM^\sim_{\gamma_v}(D)$, and $\tau, \tau'$ are the corresponding stabilization maps to $\prod_{v\in V^{\mathrm s}(\Gamma)}\bM_{g(v),n(v)}(D,b(v))\times D^{|V^{\mathrm{us}}(\Gamma)|}$.
\end{thm}

\begin{proof}
The theorem is an application of \cite[Theorem 4.1]{FWY2}. Note that \cite[Theorem 4.1]{FWY2} is stated for connected domains. But our situation requires disconnected domain and it is straightforward to check that the same proof works for disconnected domains. We omit the details here.

Let us recall the content of the theorem. In \cite[Theorem 4.1]{FWY2}, we consider $P_{D_0,r}$, which is the $r$th root stack of $P:=\mathbb P_D(L\oplus \cO)$. After the root stack construction, there are two invariant substacks $\cD_0, \cD_\infty$ under the fiberwise $\C^*$ action. $\cD_0$ is the one isomorphic to the root gerbe $\sqrt[r]{D/L}$ and $\cD_\infty$ is the one isomorphic to $D$.

The result \cite{FWY2}*{Theorem 4.1} compares rubber theory with the orbifold Gromov--Witten theory of the gerbe $\cD_0$. The topological data $\Gamma$ imposes contact orders on the two ends of the rubber target. In \cite{FWY2}, the end with normal bundle $L$ is called the $0$-side, and the other end with normal bundle $L^\vee$ is called the $\infty$-side. For the gerbe theory over $\cD_0$, we still use $\Gamma$ to represent the topological data where the weights of roots (contact order conditions) are replaced by suitable ages. A relative marking on the $0$-side of order $\mu$ corresponds to an orbifold marking of age $\mu/r$, whereas a relative marking on the $\infty$-side of order $\mu$ corresponds to an orbifold marking of age $(r-\mu)/r$. Under this convention, denote by $\bM^\bullet_{\Gamma}(\cD_0)$ the corresponding moduli of twisted stable maps to the gerbe $\cD_0$.

For simplicity, we assume all vertices on $\Gamma$ are stable. Denote the forgetful maps by the following.
\begin{align*}
  \begin{split}
  \tau_1:&\bM^\bullet_{\Gamma}(\cD_0)\rightarrow \prod\limits_{v\in V(\Gamma)}\bM_{g(v),n(v)}(D,b(v)),\\ 
  \tau_2:&\bM_{\Gamma}^{\bullet\sim}(D)\rightarrow \prod\limits_{v\in V(\Gamma)}\bM_{g(v),n(v)}(D,b(v))\,.
  \end{split}
\end{align*}
On the root gerbe $\cD_0$, there is a universal line bundle $L_r$. Let $\pi:\mathcal
C\rightarrow \bM^\bullet_{\Gamma}(\cD_0)$ be the universal curve and $f:\mathcal C\rightarrow
\cD_0$ the map to the target. Let $\mathcal L_r = f^*L_r$ and 
\[-R^*\pi_*\mathcal L_r:=R^1\pi_*\mathcal L_r-R^0\pi_*\mathcal L_r \in K^0(\bM^\bullet_{\Gamma}(\cD_0)).\]

As mentioned above, \cite[Theorem 4.1]{FWY2} can be generalized to moduli with disconnected domain using the same proof. Let $g$ be the arithmetic genus of the disconnected curve corresponding to the graph $\Gamma$. The disconnected version of the theorem implies that the following identity
\begin{align*}
\begin{split}
&(\tau_2)_*\left(\dfrac{1}{t-\Psi_\infty} \cap [\bM_{\Gamma}^{\bullet\sim}(D)]^{\vir}\right) \\
=&\dfrac{\left[  (\tau_1)_*\left(\sum\limits_{i=0}^\infty \left(\dfrac{t}{r}\right)^{g-i-1}c_i(-R^*\pi_*\mathcal L_r) \cap [\bM^\bullet_{\Gamma}(\cD_0)]^{\vir} \right)  \right]_{r^0}} {{\prod\limits_{i=1}^{\rho_\infty}\left(1+\dfrac{\ev_i^*c_1(L)-\nu_i\psi_i}{t}\right)}}  \\
\end{split}  
\end{align*}
holds in $A_*(\prod\limits_{v\in V(\Gamma)}\bM_{g(v),n(v)}(D,b(v)))[t,t^{-1}]$ as Laurent polynomials in the formal variable $t$.

The right-hand side decomposes into a product of such expressions according to the connected component of domain curves, because the disconnected moduli in orbifold theory is a product of connected moduli, and $-R^*\pi_*\mathcal L_r$ decomposes into a sum accordingly. Applying the original connected version of \cite[Theorem 4.1]{FWY2} to each factor, we conclude that the right-hand side is nothing but the pushforward of a product of 
$\left(\dfrac{1}{t-\Psi_\infty}\right)\cap [\bM_{\gamma_v}^\sim(D)]^\vir$, where $\Psi_\infty$ should be treated as the target psi-class of $\bM_{\gamma_v}^\sim(D)$.

If there are unstable vertices in $\Gamma$, the unstable vertices in $\Gamma$ correspond to unstable vertices in the localization of the root stack $P_{D_0,r}$. It is straightforward to add in factors of $D$ in the statement of the theorem and to match with the statement that we want to prove.
\end{proof}

Theorem \ref{thm_appd} has interesting corollaries. If we take the $1/t$ coefficient of the theorem, we find the following corollary.
\begin{cor}
If $\Gamma$ has more than two stable vertices, we have that \[\tau_*([\bM_{\Gamma}^{\bullet\sim}(D)]^\vir)=0\,.\]
\end{cor}

If we take the coefficient of $1/t^{|V^s(\Gamma)|}$, we obtain the following corollary.
\begin{cor}
\begin{align*}
&\tau_*(\Psi_{\infty}^{(|V^{\mathrm s}(\Gamma)|-1)}\cap [\bM_{\Gamma}^{\bullet\sim}(D)]^\vir)\\
= &\dfrac{1}{\prod_{v\in V^{\mathrm{us}}(\Gamma)} d_v}\tau'_* \left(\big[\prod\limits_{v\in V^{\mathrm s}(\Gamma)}\bM^{\sim}_{\gamma_v}(D)\times D^{|V^{\mathrm{us}}(\Gamma)|}\big]^{\vir} \right).
\end{align*}
\end{cor}

More importantly, we can use Theorem \ref{thm_appd} to deduce Lemma \ref{lem_split2}. Using the notations of Section 
\ref{sec:locrel},
we repeat below the statement of Lemma \ref{lem_split2} for convenience.

\begin{lem}\label{lem_split1}
We have the following identity:
\begin{align*}
&\tau_*\bigg(\dfrac{t}{t+\Psi} \cap [\bM^\bullet_{\Gamma_2'}(X,D)]^\vir\bigg) \\
=& \tau'_*\bigg(\prod_{v\in V(\Gamma_2')} p_v^*\bigg( \dfrac{t}{t+\Psi} \bigg) \cap \big[\prod_{v\in V(\Gamma_2')}\bM_{\gamma_v}(X,D)\big]^\vir\bigg),
\end{align*}
where $p_v$ is the projection to the factor corresponding to $v$, and $\tau, \tau'$ are the corresponding stabilization maps to \[\prod_{v\in V(\Gamma_2')}\bM_{g(v),n(v)}(X,b(v))\times_{X^{r(v)}} D^{r(v)}\] with $n(v)$ the number of half-edges on $v$ and $r(v)$ the number of roots on $v$.
\end{lem}
Lemma \ref{lem_split2} is a special case of Lemma \ref{lem_split1} with $n(v)=r(v)=1$.

\begin{proof}
First of all, we have a weaker product rule without involving psi-classes:
\begin{equation}\label{lem_split}
\tau_*[\bM^{\bullet}_{\Gamma_2'}(X,D)]^{\vir} = \tau'_*\bigg( \big[\prod_{v\in V(\Gamma_2')}\bM_{\gamma_v}(X,D) \big]^{\vir} \bigg).
\end{equation}
A proof can be formulated by passing to moduli of stable log maps using \cite{AMW} and work out the perfect obstruction theory and virtual cycles (for example, see \cite{KLR}*{Section 9.2, 9.3}).


Next, observe that $\dfrac{\Psi}{-t-\Psi}$ is in fact $\dfrac{\delta}{-t-\Psi}$ where $\delta$ is the divisor corresponding to the locus where the target degenerates. Expanding $\delta$, we have
\begin{align}\label{eqn:productrule}
\begin{split}
&\left(\dfrac{\delta}{-t-\Psi}\right) \cap [\bM^\bullet_{\Gamma_2'}(X,D)]^\vir \\
=& \sum\limits_{\mathfrak i=((\Gamma_2')_{1},(\Gamma_2')_{2})} \dfrac{\prod_{i=1}^{m(\mathfrak i)} d_i}{\Aut(\mathfrak i)} (\tau_{\mathfrak i})_* \Bigg( \bigg( \dfrac{1}{-t-\Psi_0} \bigg) \cap [\bM^{\bullet\sim}_{(\Gamma_2')_{1}} (D) \times_{D^{k_{\mathfrak i}}} \bM^\bullet_{(\Gamma_2')_{2}}(X,D)]^{\vir}  \Bigg),
\end{split}
\end{align}
where it is easy to see that the restriction of $\Psi$ to $\bM^{\bullet\sim}_{(\Gamma_2')_{1}}(D)$ becomes $\Psi_0$. The splitting of $\Gamma_2'$ into $(\Gamma_2')_{1},(\Gamma_2')_{2}$ are determined by the splitting of each component of $\Gamma_2'$. Thus, the next step is to split the virtual classes according to the components of $(\Gamma_2')_{1}$ and $(\Gamma_2')_{2}$. First of all, Equation \eqref{lem_split} (simply replace $\Gamma_2'$ by $(\Gamma_2')_{2}$) tells us that the pushforward of $[\bM^\bullet_{(\Gamma_2')_{2}}(X,D)]^{\vir}$ splits into a product of cycles according to each component (product rule). As to the rubber moduli, Theorem \ref{thm_appd} implies that $\dfrac{1}{-t-\Psi_0} \cap [\bM^{\bullet\sim}_{(\Gamma_2')_{1}} (D)]^{\vir}$ also satisfies the product rule. Note that the statement of the theorem uses $\Psi_\infty$ in order to match with \cite[Theorem 4.1]{FWY2}. But switching the rubber target upside down turns $\Psi_\infty$ into $\Psi_0$ with the rest unchanged. To fit the exact statement of the theorem, we also need to change $t$ into $-t$.

Now we can apply \eqref{eqn:productrule} to the left-hand side of Lemma \ref{lem_split1}. Note that a vertex in $(\Gamma'_2)_1$ can be either stable or unstable. On the right-hand side, we expand the product $\prod\limits_{v\in V(\Gamma'_2)}p_v^*\left(1+\dfrac{\Psi}{-t-\Psi}\right)$ and apply \eqref{eqn:productrule} again to each $\dfrac{\Psi}{-t-\Psi}\cap [\bM_{\gamma_v}(X,D)]^\vir$. The summand $1$ in $p_v^*\left(1+\dfrac{\Psi}{-t-\Psi}\right)$ covers the case when $\gamma_v$ splits into $((\gamma_v)_1,(\gamma_v)_2)$ where all the vertices in $(\gamma_v)_1$ (corresponding to rubber moduli) are unstable. Such cases might appear on the left-hand side of Lemma \ref{lem_split1} but are missing when applying \eqref{eqn:productrule} to each $\dfrac{\Psi}{-t-\Psi}\cap [\bM_{\gamma_v}(X,D)]^\vir$ on the right-hand side. It is straightforward to check that both sides match.
\end{proof}

\bibliographystyle{amsxport}
\bibliography{universal-BIB}
\end{document}